%% file: sfm-mn.tex
\documentclass[english,microtype]{scrartcl}

\usepackage[T1]{fontenc}
\usepackage[utf8]{inputenc}
\usepackage[english]{babel}
\usepackage{lmodern}
\usepackage{exscale}
\usepackage[babel]{microtype}
\usepackage{amsmath, amssymb, mathtools, mathrsfs}
\usepackage[only,mapsfrom,llbracket,rrbracket]{stmaryrd}
\usepackage{xparse}
\usepackage[inline]{enumitem}
\usepackage{wrapfig}
\usepackage{xcolor}

\usepackage{tikz}
\usetikzlibrary{cd,shapes.geometric,positioning,patterns,calc}
\usepackage[autostyle]{csquotes}
\usepackage[sorting=nyt, maxnames=10]{biblatex}
\usepackage[numbered]{bookmark}
\usepackage{hyperref}
\usepackage{amsthm, thmtools}
\usepackage[xspace]{ellipsis}
\hypersetup{
  pdfencoding=auto,      
  colorlinks,
  urlcolor = {blue!80!black},
  linkcolor = {red!70!black},
  citecolor = {green!50!black}
}

\numberwithin{figure}{section}
\numberwithin{equation}{section}
\numberwithin{table}{section}
\NewDocumentCommand{\declthm}{m m O{plain}}{
  \declaretheorem[
  name = #2,
  sibling = equation,
  style = #3
  ]
  {#1}
}
\declaretheorem[sibling=equation]{theorem}
\declaretheorem[sibling=equation]{proposition}
\declaretheorem[sibling=equation]{corollary}
\declaretheorem[sibling=equation]{lemma}
\declaretheorem[sibling=equation,style=definition]{definition}
\declaretheorem[sibling=equation,style=remark]{example}
\declaretheorem[sibling=equation,style=remark]{remark}
\declaretheorem[name=Description, sibling=equation,style=remark]{mydescription}
\declaretheorem[name=Theorem]{theoremI}

\input{math}

\author{Najib Idrissi\thanks{Université de Paris and Sorbonne Université, CNRS, IMJ-PRG, F-75006 Paris, France. Email: \href{mailto:najib.idrissi-kaitouni@imj-prg.fr}{najib.idrissi-kaitouni@imj-prg.fr}}}
\title{Formality of a higher-codimensional Swiss-Cheese operad}
\date{November 3, 2020}

\begin{document}

\maketitle
\begin{abstract}
  We study bicolored configurations of points in the Euclidean $n$-space that are constrained to remain either inside or outside a fixed Euclidean $m$-subspace, with $n - m \ge 2$.
  We define a higher-codimensional variant of the Swiss-Cheese operad, called the complementarily constrained disks operad $\VSC_{mn}$, associated to such configurations.
  The operad $\VSC_{mn}$ is weakly equivalent to the operad of locally constant factorization algebras on the stratified space $\{\R^{m} \subset \R^{n}\}$.
  We prove that this operad is formal over $\R$.
\end{abstract}

\tableofcontents

\section*{Introduction}

The little disks operads $\DD_{n}$ (for $n \ge 1$) represent operations acting on $n$-fold loop spaces.
They have had many applications in homotopy theory over the years (see e.g.~\cite{Fresse2019} for a survey).
Elements of $\DD_{n}(k)$ consist of configurations of $k$ disks inside the unit $n$-disk.
Voronov's Swiss-Cheese operads $\SC_{n}$~\cite{Voronov1999} (for $n \ge 2$) are relative versions of the little disks operads.
They encode central actions of $\DD_{n}$-algebras on $\DD_{n-1}$-algebras.
Elements of $\SC_{n}(k,l)$ are given by configurations of $k$ half-disks and $l$ full disks in the unit upper half-disk.

In this article, we introduce higher-codimensional variants of the Swiss-Cheese operads $\VSC_{mn}$ (for $n-2 \ge m \ge 1$), called the ``complementarily constrained (little) disks operads''.\footnote{In previous versions of this article, this operad was denoted $\mathsf{VSC}_{mn}$ and called the ``variant Swiss-Cheese operad''.}
These operads encode actions of $\DD_{n}$-algebras on a $\DD_{m}$-algebras by central derivations (see Section~\ref{sec:pres-homol}).
Elements of $\VSC_{mn}(k,l)$ are given by configurations of $k$ ``terrestrial'' disks centered on $D^{m} \subset D^{n}$ and  $l$ ``aerial'' entirely contained in $D^{n} \setminus D^{m}$.

A fundamental property of $\DD_{n}$ is its formality~\cite{Kontsevich1999,Tamarkin2003,LambrechtsVolic2014,Petersen2014,FresseWillwacher2015}: the cohomology $H^{*}(\DD_{n};\Q)$ is quasi-isomorphic (as a cooperad in CDGAs) to the forms on $\DD_{n}$.
Unlike the little disks operads, the Swiss-Cheese operads are not formal~\cite{Livernet2015,Willwacher2017a}.
In this paper, we establish the following result:

\begin{theoremI}[See Theorem~\ref{thm:main}]\label{thm:intro}
  For $n - 2 \ge m \ge 1$, the complementarily constrained disks operad $\VSC_{mn}$ is formal over $\R$.
\end{theoremI}

The operad $\VSC_{(n-1)n}$ contains the Swiss-Cheese operad $\SC_{n}$ as a suboperad of connected components.
It thus follows from arguments of Livernet~\cite{Livernet2015} that $\VSC_{(n-1)n}$ is not formal (Remark~\ref{rmk:livernet}).
By~\cite[Section~5.1]{Willwacher2017a}, nonformality of $\SC_{n}$ is equivalent to nonformality of the inclusion $\DD_{n-1} \hookrightarrow \DD_{n}$ (established in~\cite{TurchinWillwacher2018}).
Since the inclusion $\DD_{m} \hookrightarrow \DD_{n}$ is formal for $n \ge m+2$, Theorem~\ref{thm:intro} might not be a surprise.
However, Willwacher studied another generalization of the Swiss-Cheese operad, the ``extended Swiss-Cheese operad'' $\ESC_{mn}$~\cite{Willwacher2017a}.
He proved that its formality is equivalent to the formality of $\DD_{m} \subset \DD_{n}$~\cite[Section~5.1]{Willwacher2017a}.
The difference with $\VSC_{mn}$ is that in $\VSC_{mn}$, the aerial disks are forbidden from touching the ``ground'' $D^{m}$, whereas this is allowed in $\ESC_{mn}$ (see Remark~\ref{rmk:difference}).
The argument used for the formality of $\ESC_{mn}$ thus does not seem easily adaptable: $\VSC_{mn}$ is obtained by removing a subspace from $\ESC_{mn}$, an operation which is usually difficult to deal with in homotopy theory.
It is not clear that formality of $\DD_{m} \subset \DD_{n}$ directly implies formality of $\VSC_{mn}$, or conversely.

\paragraph{Motivation}
\label{sec:motivation}

The general motivation for this article comes from the study of configuration spaces started in previous works.
Campos--Willwacher~\cite{CamposWillwacher2016} and the author~\cite{Idrissi2018b} provided combinatorial models for the real homotopy types of configuration spaces of simply connected closed smooth manifolds.
Campos, Lambrechts, Willwacher, and the author~\cite{CamposIdrissiLambrechtsWillwacher2018}  provided similar models for configuration spaces of compact, simply connected smooth manifolds of dimension $\ge 4$.
Campos, Ducoulombier, Willwacher, and the author~\cite{CamposDucoulombierIdrissiWillwacher2018} studied framed configuration spaces of orientable closed manifolds, i.e.\ configurations of points equipped with trivializations of the tangent spaces.

In each of these articles, knowing models for the little disks operads or their variants was essential.
Configuration spaces of $\R^{n}$ are intimately linked to the little disks operad $\DD_{n}$, and closed manifolds are locally homeomorphic to $\R^{n}$.
The formality of $\DD_{n}$, and more precisely its proof by Kontsevich and Lambrechts--Volić, was thus essential in finding models for configurations spaces of closed manifolds in~\cite{CamposWillwacher2016,Idrissi2018b}.
Similarly, a manifold with boundary is locally homeomorphic to the upper half-space $\mathbb{H}^{n} \subset \R^{n}$, and configuration spaces of $\mathbb{H}^{n}$ are linked to the Swiss-Cheese operad $\SC_{n}$.
While $\SC_{n}$ is not formal, Willwacher~\cite{Willwacher2015a} defined a model for the real homotopy type of $\SC_{n}$ which was used extensively in~\cite{CamposIdrissiLambrechtsWillwacher2018}.
For framed configuration spaces~\cite{CamposDucoulombierIdrissiWillwacher2018}, we used the model for the framed little disks operad due to Khoroshkin--Willwacher~\cite{KhoroshkinWillwacher2017}.

Our goal is to study the configuration spaces of the complement $N \setminus M$, where $N$ is a closed $n$-manifold and $M$ is a closed sub-$m$-manifold of codimension $\ge 2$ (e.g.\ the complement of a knot $S^{3} \setminus K$).
Such a pair $(N,M)$ is locally homeomorphic to the stratified space $(\R^{n},\R^{m})$.
Using the analogy above, configuration spaces of $(\R^{n},\R^{m})$ are linked to the operad $\VSC_{mn}$.
Based on the previous works cited above, we are led to expect that models for $\VSC_{mn}$ will produce models for $\Conf_{N \setminus M}$ by adapting and generalizing the proof, just like models for $\DD_{n}$ produced to models for configuration spaces of closed manifolds.
In this article, we find that models for $\VSC_{mn}$ are as simple as possible: the operad is formal, i.e.\ the cohomology $H^{*}(\VSC_{mn})$ is a model.

\paragraph{Proof strategy and outline}

The proof of our theorem is inspired by Kontsevich's~\cite{Kontsevich1999} proof of the formality of the little disks operad and its improvement by Lambrechts--Voli\'{c}~\cite{LambrechtsVolic2014}.
We first define the compactifications $\VFM_{mn}(k,l)$ of $\Conf_{\R^{m}(k)} \times \Conf_{\R^{n} \setminus \R^{m}}(l)$, inspired by the Fulton--MacPherson compactification~\cite{AxelrodSinger1994,FultonMacPherson1994,Sinha2004}.
These compactifications form an operad with the homotopy type of $\VSC_{mn}$.
We build an intermediate cooperad, $\vgraphs_{mn}$, which serves as a bridge between the cohomology $H^{*}(\VFM_{mn})$ and the piecewise semi-algebraic forms (see~\cite{HardtLambrechtsTurchinVolic2011}) $\OmPA^{*}(\VFM_{mn})$.
The definition of $\vgraphs_{mn}$ is inspired by  Willwacher's model for the Swiss-Cheese operad~\cite{Willwacher2015a} and by~\cite[Section~8]{KhoroshkinWillwacher2017}.
The map $\vgraphs_{mn} \to \OmPA^{*}(\VFM_{mn})$ is defined by integrals.
We cannot find a direct map $\vgraphs_{mn} \to H^{*}(\VFM_{mn})$, as the differential of $\vgraphs_{mn}$ depends on non-explicit integrals.
Using vanishing results on the cohomology of some graph complex, we are able to simplify $\vgraphs_{mn}$ up to homotopy, and then map it to $H^{*}(\VFM_{mn})$.

Section~\ref{sec:prerequisites} contains background on operads, the little disks operads, and piecewise semi-algebraic forms.
In Section~\ref{sec.defin-comp-with}, we define the compactifications $\VFM_{mn}$ and we compare them with $\VSC_{mn}$.
We give examples of $\VSC_{mn}$-algebras based on relative iterated loop spaces.
In Section~\ref{sec:cohom-sfm}, we compute the cohomology of $\VSC_{mn}$.
We give a presentation of its homology $\vsc_{mn} = H_{*}(\VSC_{mn})$.
In Section~\ref{sec.graph-complexes}, we start by reviewing Kontsevich's proof of the formality of the little disks operad, and we define the cooperad $\vgraphs_{mn}$ that will be used to adapt that proof to $\VFM_{mn}$.
We moreover construct the map from $\vgraphs_{mn}$ to $\OmPA^{*}(\VFM_{mn})$ using integrals.
In Section~\ref{sec:proof-formality}, we first show that certain integrals used in the definition of $\vgraphs_{mn}$ can be simplified.
We construct the map into $H^{*}(\VFM_{mn})$ and we show that all our maps are quasi-isomorphisms.
In Appendix~\ref{sec:twist-relat-cooper}, we define twisting of relative cooperads.
In Appendix~\ref{sec:semi-algebr-comp}, we sketch a proof that $\VFM_{mn}(U,V)$ is an SA manifold and that canonical projections are SA bundles.

\paragraph{Acknowledgments}

The author would like to thank Benoit Fresse, Matteo Felder, Muriel Livernet, and Thomas Willwacher for fruitful discussions and helpful comments.
The author would also like to thank the anonymous referees for detailed and thoughtful reports and many suggestions for improvements.
The author acknowledges support from ERC StG 678156--GRAPHCPX.\@

\section{Prerequisites}
\label{sec:prerequisites}

We work with cohomologically graded modules over the base field $\R$ (except in Section~\ref{sec.comp-cohom} which is over $\Z$).
The cohomology (resp.\ homology) of a space is concentrated in nonnegative (resp.\ nonpositive) degrees.
For a graded vector space $V$, the free graded symmetric algebra $S(V)$ is $\R[V^{\mathrm{even}}] \otimes \Lambda(V^{\mathrm{odd}})$, i.e.\ the tensor product of the polynomial algebra on even elements by the exterior algebra on odd elements.
Homogeneous elements $x,y$ satisfy $yx = (-1)^{(\deg x)(\deg y)} xy$.

\subsection{Operads}
\label{sec:operads}

We work extensively with operads and cooperads and we assume basic proficiency with the theory.
General references include~\cite{LodayVallette2012} and~\cite[Part~I(a)]{Fresse2017}.
We will usually label inputs by elements of arbitrary finite sets rather than numbers.
Briefly, let $\Bij$  be the category of finite sets and bijections.
A symmetric collection is a functor $\Bij \to \mathcal{C}$ where $\mathcal{C}$ is a symmetric monoidal category.
For $k \ge 0$, we let $\underline{k} = \{ 1, \dots, k\}$.
A symmetric collection $\MM$ can equivalently be seen as a symmetric sequence $\{ \MM(n) \coloneqq \MM(\underline{n}) \}_{n \ge 0}$ with actions of $\Sigma_{n} = \operatorname{Aut}_{\Bij}(\underline{n})$ on $\MM(n)$.

For a pair of finite sets $W \subset U$, we define the quotient $U/W \coloneqq (U \setminus W) \sqcup \{*\}$ (note that $U/\varnothing = U \sqcup \{*\}$).
For $u \in U$, we let $[u] \in U/W$ be its class in the quotient.
An operad $\PP$ is a symmetric collection equipped with composition maps $\circ_{W} : \PP(U/W) \otimes \PP(W) \to \PP(U)$, for each pair $W \subset U$, satisfying the usual axioms.
A cooperad $\CC$ is a symmetric collection equipped with cocomposition maps $\circ_{W}^{\vee} : \CC(U) \to \CC(U/W) \otimes \CC(W)$.
A Hopf cooperad~\cite{Fresse2017} a cooperad in the category of commutative differential-graded algebras (CDGAs).

We also deal with some special particular bicolored operads called ``relative operads''~\cite{Voronov1999} or ``Swiss-Cheese type operads''~\cite{Willwacher2016}.
Given an operad $\PP$, a relative $\PP$-operad is an operad in the category of right $\PP$-modules.
Equivalently, a relative $\PP$-operad is a bisymmetric collection, i.e.\ a functor $\QQ : \Bij \times \Bij \to \mathcal{C}$, equipped with composition maps:
\begin{align*}
  \circ_{T} : \QQ(U, V/T) \otimes \PP(T) & \to \QQ(U, V) & \text{for } V \subset T, \\
  \circ_{W,T} : \QQ(U/W, V) \otimes \QQ(W, T) & \to \QQ(U, V \sqcup T) & \text{for } W \subset U.
\end{align*}
We will often write ``the operad $\QQ$'' when $\PP$ is implicit from the context.
Relative cooperad are defined dually.
We also apply the adjective ``Hopf'' to refer to such cooperads in the category of CDGAs.

\subsection{Little disks and variants}
\label{sec:little-disks-fulton}

Fix some $n \ge 0$ and
let $D^{n} = \{ x \in \R^{n} \mid \|x\| \le 1 \}$ be the closed disk.
The space $\DD_{n}(r)$ is the space of maps $c : (D^{n})^{\sqcup r} \to D^{n}$ such that:
\begin{enumerate*}[label=(\roman*)]
\item each $c_{i} : D^{n} \to D^{n}$ is an embedding given by the composition of a translation and a positive rescaling;
\item the interiors of two different little disks are disjoint, i.e.\ $c_{i}(\mathring{D}^{n}) \cap c_{j}(\mathring{D}^{n}) = \varnothing$ for $i \neq j$.
\end{enumerate*}
Using composition of embeddings, the collection $\DD_{n} = \{ \DD_{n}(r) \}_{r \ge 0}$ forms a topological operad called the little $n$-disks operad.

Let us now fix notations for the (Axelrod--Singer--)Fulton--MacPherson operad~\cite{FultonMacPherson1994,AxelrodSinger1994} (see also~\cite{Sinha2004} and~\cite[Chapter~5]{LambrechtsVolic2014}).
Given some space $X$, and a finite set $U$, we define the configuration space:
\begin{equation}
  \Conf_{X}(U) \coloneqq \{ (x_{u})_{u \in U} \in X^{U} \mid \forall u \neq v, \, x_{u} \neq x_{v} \}.
\end{equation}
Consider the quotient $\Conft_{n}(U) \coloneqq \Conf_{\R^{n}}(U) / \R^{n} \rtimes \R_{>0}$ by the action of translations and positive rescalings.
This space embeds in $(S^{n-1})^{\Conf_{U}(2)} \times [0,+\infty]^{\Conf_{U}(3)}$ using the maps $\theta_{ij}$ and $\delta_{ijk}$ from~\eqref{eq:theta}, \eqref{eq:delta}.
The Fulton--MacPherson compactification $\FM_{n}(r)$ is the closure of the image of this embedding.
The collection $\FM_{n}$ forms an operad~\cite[Section~5.2]{LambrechtsVolic2014} with the same homotopy type as $\DD_{n}$, see~\cite{Markl1999} and~\cite[Proposition~4.3]{Salvatore2001}.

\begin{remark}\label{rmk:diskn}
  There is an operad $\Disk_{n}^{\fr}$ related to locally constant framed factorization algebras on $\mathring{D}^{n}$ and which has the same homotopy type as $\DD_{n}$, see~\cite[Definition~5.4.2.10]{Lurie2016} or~\cite[Notation~2.8]{AyalaFrancisTanaka2017a}.
  This operad $\Disk_{n}^{\fr}$ is associated, in some sense, to the trivial stratification of $\R^{n}$.
  The present article is devoted to the operad $\Disk_{m \subset n}^{\fr}$ associated to the stratification $\{\R^{m} \subset \R^{n}\}$~\cite[Section~4.3]{AyalaFrancisTanaka2017a}.
\end{remark}

\subsection{Semi-algebraic sets and PA forms}
\label{sec:pa-forms}

For technical reasons, we will use the technology of semi-algebraic (SA) sets and piecewise semi-algebraic (PA) forms. We use~\cite{HardtLambrechtsTurchinVolic2011} as a general reference.
Recall in particular that the CDGA $\OmPA^{*}(X)$ of all PA forms on a compact SA set $X$ is a model for the real homotopy type of $X$~\cite[Theorem~6.1]{HardtLambrechtsTurchinVolic2011}.

If $\PP$ is an operad in compact SA sets, then the symmetric sequence $\OmPA^{*}(\PP)$ is not a Hopf cooperad: the Künneth quasi-isomorphisms go in the wrong direction.
However, for a Hopf cooperad $\CC$, we can define a ``morphism'' $\CC \to \OmPA^{*}(\PP)$ as a collection of maps $\CC(U) \to \OmPA^{*}(\PP(U))$ making the obvious diagrams commute~\cite[Chapter~3]{LambrechtsVolic2014}.
We call it a quasi-isomorphism if it is a quasi-isomorphism in each arity.
For simplicity, we will treat $\OmPA^{*}(\PP)$ as if it were an actual Hopf cooperad; when we write a morphism into $\OmPA^{*}(\PP)$, we actually have a ``morphism'' as defined above.
Results of Fresse~\cite[Discussion after Proposition~4.4]{Fresse2018} can be adapted to $\OmPA^{*}$: if $\PP$ is cofibrant in the category of operads satisfying $\PP(0) = \{*\}$ (e.g.\ $\PP = \FM_{n}$), then any Hopf cooperad quasi-isomorphic to $\OmPA^{*}(\PP)$ encodes the real homotopy type of $\PP$.
Briefly, there is an operadic upgrade of the functor $\OmPA^{*}$ which turns operads into Hopf cooperads, and this operadic upgrade is part of a Quillen adjunction.
We will thus say that an operad $\PP$ in compact SA sets is formal if there exists a zigzag of quasi-isomorphisms of Hopf cooperads $H^{*}(\PP;\R) \gets \cdot \to \OmPA^{*}(\PP)$.
These constructions can be extended to relative operads over a given base, as they can be seen as operads in a given symmetric monoidal category (of right operadic modules).

\section{Definition of \texorpdfstring{$\VFM_{mn}$}{VFM\_mn} and comparison}
\label{sec.defin-comp-with}

From now on and until the end of the article, we fix integers $n - 2 \ge m \ge 1$.
(In some tangential remarks, we will consider $n = m +1$.)
We identify $\R^{m}$ as the subspace of $\R^{n}$ given by $\R^{m} \times \{0\}^{n-m}$.

\subsection{The compactification and its boundary}
\label{sec.comp-its-bound}

Let $U$ and $V$ be finite sets.
We define the colored configuration spaces by:
\begin{equation}
  \Conf_{mn}(U,V)
  \coloneqq \Conf_{\R^{m}}(U) \times \Conf_{\R^{n} \setminus \R^{m}}(V) \subset \Conf_{\R^{n}}(U \sqcup V).
\end{equation}
Roughly speaking, $\Conf_{mn}(U,V)$ is the set of configurations of bicolored of points in $\R^{n}$: $U$ ``terrestrial'' points in $\R^{m}$, and $V$ ``aerial'' points in $\R^{n} \setminus \R^{m}$.
We will reuse the terminology ``aerial/terrestrial'' throughout the document.

The group $\R^{m} \rtimes \R_{>0}$ of translations and positive rescalings acts on $\Conf_{mn}(U,V)$.
Let $\Conft_{mn}(U,V)$ be the quotient.
Elements of $\Conft_{mn}(U,V)$ can be seen as configurations of radius $1$ (i.e.\ $\max(\|x_{i}\|)_{i \in U \sqcup V} = 1$) with a barycenter in $\{0\}^{m} \times \R^{n-m} \subset \R^{n}$.
Since $\R^{m} \rtimes \R_{> 0}$ is contractible, the quotient map is a weak homotopy equivalence.
If $\#U + 2 \# V \ge 2$, then the action is free, smooth, and proper, thus $\Conft_{mn}(U,V)$ is a manifold of dimension $m \#U + n \#V - m - 1$.
However, if $\#U \le 1$ and $\#V = 0$, then $\Conft_{mn}(U,V)$ is merely a point.
We embed $\Conft_{mn}(U,V)$ into $(S^{n-1})^{\Conf_{U \sqcup V}(2)} \times [0, +\infty]^{\Conf_{U \sqcup V}(3)} \times (S^{n-m-1})^{V} \times [0,+\infty]^{\Conf_{V}(2)} \times [0,+\infty]^{\Conf_{U \sqcup V}(2) \times V}$ using the maps:
\begingroup
\allowdisplaybreaks
\begin{align}
  \label{eq:theta}
  \theta_{ij}(x) & \coloneqq (x_{i} - x_{j})/\|x_{i} - x_{j}\|, & \text{for } i \neq j \in U \sqcup V; \\
  \label{eq:delta}
  \delta_{ijk}(x) & \coloneqq \|x_{i} - x_{j}\| / \|x_{i} - x_{k}\|, & \text{for } i \neq j \neq k \neq i \in U \sqcup V; \\
  \alpha_{v}(x) & \coloneqq p_{(\R^{m})^{\perp}}(x_{v}) / \|p_{(\R^{m})^{\perp}}(x_{v})\|, & \text{for } v \in V; \\
  \label{eq:rho}
  \rho_{vv'}(x) & \coloneqq \| p_{(\R^{m})^{\perp}}(x_{v}) \| / \| p_{(\R^{m})^{\perp}}(x_{v'}) \|, & \text{for } v,v' \in V;
  \\
  \label{eq:sigma}
  \sigma_{ijv}(x) & \coloneqq \|x_{i}-x_{j}\| / \| p_{(\R^{m})^{\perp}}(x_{v}) \|, & \text{for } i,j \in U \sqcup V, \, v \in V;
\end{align}
\endgroup
where $p_{(\R^{m})^{\perp}}$ is the orthogonal projection on $(\R^{m})^{\perp} = \{0\}^{m} \times \R^{n-m} \subset \R^{n}$.

\begin{definition}
  The space $\VFM_{mn}(U,V)$ is the closure of the image of the embedding $(\theta_{ij}, \delta_{ijk}, \alpha_{v}, \rho_{vv'}, \sigma_{ijv})$.
\end{definition}

\begin{example}
  We have $\VFM_{mn}(U,\varnothing) = \FM_{m}(U)$ and $\VFM_{mn}(0,1) = S^{n-m-1}$.
\end{example}

\begin{proposition}[Appendix~\ref{sec:semi-algebr-comp}]\label{prop:vfm-sa}
  The space $\VFM_{mn}(U,V)$ is a compact semi-algebraic manifold and a smooth manifold with corners.
  Its dimension is $m \#U + n\#V - m - 1$ if $\#U + 2\#V \ge 2$, and zero otherwise.
  The projections $p_{U,V} : \VFM_{mn}(U \sqcup I, V \sqcup J) \to \VFM_{mn}(U,V)$ are SA bundles.
\end{proposition}

\begin{proposition}\label{prop:vfm-operad}
  The collection $\VFM_{mn}$ forms a relative $\FM_{n}$-operad.
\end{proposition}
\begin{proof}
  On the coordinates $\theta_{ij}$ and $\delta_{ijk}$, the formulas are identical to $\FM_{n}$, see~\cite[Section~5.2]{LambrechtsVolic2014}.
  For $v \in V$, one simply define $\alpha_{v}(x \circ_{T} y) = \alpha_{[v]}(x)$; and $\alpha_{v}(x' \circ_{W,T} y') = \alpha_{v}(x')$ if $v \notin T$, or $\alpha_{v}(y')$ if $v \in T$.
  For $v \neq v' \in V$, then $\rho_{vv'}(x \circ_{T} y) = 1$ if $v,v' \in T$, or $\rho_{[v][v']}(x)$ otherwise; and $\rho_{vv'}(x' \circ_{W,T} y') = \rho_{vv'}(x')$ if $v,v' \notin T$, or $\rho_{vv'}(y')$ if $v,v' \in T$, or $0$ if $v \in T$, $v' \notin T$, or $+\infty$ if $v' \in T$, $v \notin T$.
  Finally, $\sigma_{ijv}(x \circ_{T} y) = 0$ if $i,j \in T$, or $\sigma_{[i][j][v]}(x)$ otherwise; and $\sigma_{ijv}(x' \circ_{W,T} y') = \sigma_{ijv}(y')$ if $i,j,v \in W \cup T$, or $0$ if $i,j \in W \cup T$, $v \notin T$, or $\sigma_{[i][j]v}(x')$ if $v \notin T$ and $\#(\{i,j\} \cap (W \cup T)) = 1$, or $\infty$ otherwise.
\end{proof}

\begin{remark}
  \label{rmk:difference}
  The operad $\VFM_{mn}$ is not homotopy equivalent to the operad $\ESC_{mn}$ considered by Willwacher~\cite{Willwacher2017a}.
  Recall that $\ESC_{mn}(U,V) \coloneqq \DD_{n}(U \sqcup V) \times_{\DD_{n}(U)} \DD_{m}(U)$, where $\DD_{n}(U \sqcup V) \to \DD_{n}(U)$ is the projection that forgets disks and $\DD_{m}(U) \to \DD_{n}(U)$ is the usual embedding.
  The difference is that we do not allow ``aerial'' points to be on $\R^{m}$, so e.g.\ $\VFM_{mn}(0,1) = S^{n-m-1} \not\simeq \ESC_{mn}(0,1) \simeq *$, and $\VFM_{mn}(1,1) \simeq S^{n-m-1} \not\simeq \ESC_{mn}(1,1) \simeq S^{n-1}$.
\end{remark}

\begin{proposition}
  \label{prop:decomp-boundary}
  We have a decomposition in terms of faces:
  \begin{equation*}
    \partial \VFM_{mn}(U,V) = \bigcup_{T \in \BF'(V)} (\im {\circ}_{T}) \cup \bigcup_{(W,T) \in \BF''(U;V)} (\im {\circ}_{W,T}),
  \end{equation*}
  where the subsets $\BF'(V) \subset V$ and $\BF''(U;V) \subset U \times V$ are respectively defined by the conditions $\# T \ge 2$ and by $(W \cup T \subsetneq U \cup V \text{ and } 2 \cdot \# T + \# W \ge 2)$.
  Each of these boundary faces is a compact SA subset of the boundary, and the intersection of two distinct faces is of positive codimension in $\partial \VFM_{mn}(U,V)$.
\end{proposition}
\begin{proof}
  Adapting the proofs of~\cite[Proposition~5.4.1]{LambrechtsVolic2014} is straightforward.
\end{proof}

If $p : E \to B$ is an SA bundle of rank, then its fiberwise boundary $p^{\partial} : E^{\partial} \to B$ is an SA bundle of rank $k-1$, where $E^{\partial} = \bigcup_{x \in B} \partial p^{-1}(x)$, see~\cite[Definition~8.1]{HardtLambrechtsTurchinVolic2011}.
The fiberwise Stokes formula reads $d(p_{*}\alpha) = p_{*}(d\alpha) \pm p^{\partial}_{*}\alpha$~\cite[Proposition~8.12]{HardtLambrechtsTurchinVolic2011}, where $p_{*} : \Omega_{\mathrm{min}}^{*}(E) \to \OmPA^{*-k}(B)$ is integration along fibers.

\begin{proposition}
  The fiberwise boundary of the projection $p_{U,V} : \VFM_{mn}(U \sqcup I, V \sqcup J) \to \VFM_{mn}(U,V)$ is the subset of $\VFM_{mn}(U \sqcup I, V \sqcup J)$ given by:
  \begin{equation*}
    \VFM_{mn}^{\partial}(U \sqcup I, V \sqcup J) = \bigcup_{T \in \BF'(V,J)} (\im {\circ}_{T}) \cup \bigcup_{(W,T) \in \BF''(U,I;V,J)} (\im {\circ}_{W,T}),
  \end{equation*}
  where the subsets $\BF'(V,J) \subset \BF'(V \sqcup J)$ and $\BF''(U,I;V,J) \subset \BF''(U \sqcup I, V \sqcup J)$ are respectively defined by the conditions $\# (T \cap J) \le 1$ and by $(( U \subset W \text{ and } V \subset T ) \text{ or } (V \cap T = \varnothing \text{ and } \#(U \cap W) \leq 1))$.
\end{proposition}
\begin{proof}
  We can adapt the proof of~\cite[Proposition~5.7.1]{LambrechtsVolic2014} immediately.
  We simply use the decomposition of Proposition~\ref{prop:decomp-boundary} and we check easily that $\BF'(V,J)$ and $\BF''(U,I;V,J)$ index the faces which are sent to the interior of $\Conft_{mn}(U,V)$ under the projection $p_{U,V}$.
\end{proof}

\subsection{Comparison with $\VSC_{mn}$}
\label{sec.comp-with-disks}

In this section, we compare $\VFM_{mn}$ with
$\VSC_{mn}$, the complementarily constrained (little) disks operad.
Let $D^{m} = D^{n} \cap \R^{m}$ for convenience.

\begin{definition}
  \label{def:vsc}
  The space $\VSC_{mn}(U,V)$ is the space of maps $c : (D^{n})^{U \sqcup V} \hookrightarrow D^{n}$ satisfying:
  \begin{enumerate*}[label=(\roman*)]
  \item for all $i$, $c_{i} : D^{n} \hookrightarrow D^{n}$ is an embedding obtained by composing a translation and a positive rescaling;
  \item for $i \neq j$, we have $c_{i}(\mathring{D}^{n}) \cap c_{j}(\mathring{D}^{n}) = \varnothing$;
  \item for $u \in U$, we have $c_{u}(D^{m}) \subset D^{m}$;
  \item for $v \in V$, we have $c_{v}(D^{n}) \cap D^{m} = \varnothing$.
  \end{enumerate*}
  Using composition of embeddings, $\VSC_{mn}$ is a relative $\DD_{n}$-operad, called the complementarily constrained disks operad.
\end{definition}

\begin{remark}
  The usual Swiss-Cheese operad $\SC_{n}$ is the suboperad of $\VSC_{(n-1)n}$ formed by the components where all the aerial disks are in the upper half-disk.
\end{remark}

\begin{proposition}
  \label{prop:zigzag-vsc-vfm}
  There exists a zigzag of weak homotopy equivalences of operads $(\VSC_{mn}, \DD_{n}) \simeq (\VFM_{mn}, \FM_{n})$.
\end{proposition}
\begin{proof}
  We can adapt the proof of Salvatore~\cite[Proposition~3.9]{Salvatore2001} directly.
  Briefly, we use the Boardman--Vogt resolution $W \VSC_{mn} \xrightarrow{\sim} \VSC_{mn}$.
  Elements of $W\VSC_{mn}(U,V)$ are rooted trees with $(U,V)$ leaves, bicolored edges, internal vertices labeled by $\VSC_{mn}$, and internal edges labeled by time parameters $t \in [0,1]$.
  Operadic composition is grafting (new edges are decorated by $1$).
  If $t=0$ then the corresponding edge is collapsed and the decorations are composed.
  The map $W\VSC_{mn} \to \VFM_{mn}$ is defined on trees with edge decorations $<1$ by rescaling the disks by one minus edge decoration, composing in $\VSC_{mn}$, and keeping the centers of the remaining disks.
  This extends continuously to composite trees and defines a weak equivalence of operads.
\end{proof}

Let us also give examples of $\VSC_{mn}$-algebras.
Recall first that for a pointed space $* \in X$, the iterated loop space $\Omega^{n}X$ is the space of maps $\gamma : D^{n} \to X$ such that $\gamma(\partial D^{n}) = *$.
The space $\Omega^{n}X$ is an algebra over $\DD_{n}$.
Conversely, the recognition principle states that any ``group-like'' $\DD_{n}$-algebra is weakly equivalent to an iterated loop space~\cite{BoardmanVogt1968,May1972}.

\begin{wrapfigure}[4]{r}{2cm}
  \vspace{-5mm}
  \begin{tikzpicture}[every label/.style={font=\scriptsize}, every node/.style={outer sep = 0, inner sep = .5mm}, baseline=5mm]
    \fill[pattern=north west lines] (-1,0) -- (1,0) arc [start angle=0, end angle=180, radius=1];
    \draw [red] (-1,0) to (1,0) node[midway, label={-90:$\mapsto A$}] {};
    \draw [blue] (1,0) arc [start angle=0, end angle=180, radius=1] node[midway, label={90:$\mapsto *$}] {};
    \node[fill=white, fill opacity = .8, ellipse, inner sep = 0] at (0,.5) {$\mapsto X$};
  \end{tikzpicture}
\end{wrapfigure}

For a pair of pointed topological spaces $* \in A \subset X$, the iterated loop space $\Omega^{n}(X,A)$ is $\operatorname{hofib}(\Omega^{n-1}A \to \Omega^{n-1}X)$.
Concretely, let $D^{n}_{h} \coloneqq D^{n} \cap \mathbb{H}^{n}$ be the upper half-disk.
Its boundary $\partial D^{n}_{h}$ is the union of the disk $\partial_{-} D^{n}_{h} \coloneqq D^{n} \cap \partial \mathbb{H}^{n} \cong D^{n-1}$ and the upper hemisphere $\partial_{+} D^{n}_{h} \coloneqq \partial D^{n} \cap \mathbb{H}^{n} \cong D^{n-1}$ along the equator $\partial D^{n} \cap \partial \mathbb{H}^{n} \cong S^{n-2}$.
The relative iterated loop space $\Omega^{n}(X,A)$ is the space of maps $\gamma : D^{n}_{h} \to X$ such that $\gamma(\partial_{-}D^{n}_{h}) \subset A$ and $\gamma(\partial_{+}D^{n}_{h}) = *$.
For example, $\Omega^{1}(X,A) = \{ \gamma : [0,1] \to X \mid \gamma(0) \in A,\, \gamma(1) = *\}$.
A sketch for $n=2$ is on the side.
The pair $(\Omega^{n}(X,A), \Omega^{n}X)$ is an algebra over the operad $\SC_{n}$.
By the relative recognition principle, any $\SC_{n}$-algebra satisfying appropriate properties is weakly equivalent to such a pair~\cite{Ducoulombier2014,Quesney2015,HoefelLivernetStasheff2016,Vieira2018}.

By analogy, we define the $(n,m)$-relative iterated loop space:
\begin{equation}
  \Omega^{n,m}(X,A) \coloneqq \{ \gamma : D^{n} \to X \mid \gamma(D^{m}) \subset A \text{ and } \gamma(\partial D^{n}) = * \}.
\end{equation}
The pair $(\Omega^{n,m}(X,A), \Omega^{n}X)$ is an algebra over the operad $\VSC_{mn}$.
We conjecture that an analogous relative recognition principle holds: any $\VSC_{mn}$-algebra satisfying appropriate conditions should be weakly equivalent to such a pair.

\section{(Co)homology of \texorpdfstring{$\VSC_{mn}$}{VSC\_mn}}
\label{sec:cohom-sfm}

In this section, we compute the integral cohomology of $\VSC_{mn}$ (Definition~\ref{def:vsc}).
We then give a presentation of the operad $H_{*}(\VSC_{mn})$ by generators and relations.
Unless specified, the ring of coefficients is $\Z$ in this section.

\subsection{The cohomology as a Hopf cooperad}
\label{sec.comp-cohom}

We will first compute the cohomology of $\Conf_{W}(l)$ with $W \coloneqq \R^{n} \setminus \R^{m}$ for $n - m \ge 2$.
The computation is inspired by the methods of~\cite{Sinha2013}.
We prove that it is free as an abelian group, thus we will be able to apply Künneth's formula to get the cohomology of $\VSC_{mn}(k,l) \simeq \Conf_{\R^{m}}(k) \times \Conf_{W}(l)$ as a tensor product.
Then we study the maps induced on cohomology by the operad structure of $\VSC_{mn}$.

\begin{definition}
  The Poincaré polynomial of $X$ is $\mathscr{P}(X) \coloneqq \sum_{i \ge 0} (\operatorname{rk} H^{i}(X)) \cdot t^{i}$.
  For $P,\, Q \in \mathbb{N}[[t]]$, we say that $P \preceq Q$ if the coefficients of $Q-P$ are nonnegative.
\end{definition}

\begin{proposition}\label{prop:poinc-poly}
  For $n-m\ge2$, the Poincaré polynomial of $\Conf_{W}(l)$ satisfies:
  \begin{equation}
    \label{eq.poinc-poly}
    \mathscr{P}(\Conf_{W}(l)) \preceq \prod\nolimits_{i=0}^{l-1} (1 + t^{n-m-1} + i t^{n-1}).
  \end{equation}
  Moreover, if the equality is reached and the homology of $\Conf_{W}(l-1)$ is free as a $\Z$-module, then the homology of the total space $\Conf_{W}(l)$ is free too.
\end{proposition}

\begin{proof}
  We use the Serre spectral sequence of the Fadell--Neuwirth fibrations:
  \begin{equation}
    \label{eq:fibration}
    \begin{tikzcd}
      S^{n-m-1} \vee (S^{n-1})^{\vee (l-1)} \ar[r, hook] & \Conf_{W}(l) \ar[r, two heads, "\pi"] & \Conf_{W}(l-1).
    \end{tikzcd}
  \end{equation}
  For $n - m = 2$, the base $\Conf_{W}(l-1)$ is not simply connected.
  However, we can adapt the arguments of~\cite[Lemma~6.3]{Cohen1976} to show that the coefficient system is trivial.
  Let $c_{1}, c_{2} \in \Conf_{W}(l-1)$ be two configurations and $F_{1}, F_{2} = \pi^{-1}(c_{1}), \pi^{-1}(c_{2}) \subset \Conf_{W}(l)$ the fibers ($\simeq S^{1} \vee (S^{n-1})^{\vee (l-1)}$).
  Any path $\gamma \in \Omega_{c_{1},c_{2}}\Conf_{W}(l-1)$ lifts to a path in $\Conf_{W}(l)$ by putting the $l$th point far from all the others (e.g.\ outside a ball $B$ enclosing the compact subset $\operatorname{im}(\gamma)$).
  Let us show that the induced isomorphism $h_{\gamma} : H_{*}(F_{1}) \to H_{*}(F_{2})$ is the identity.

  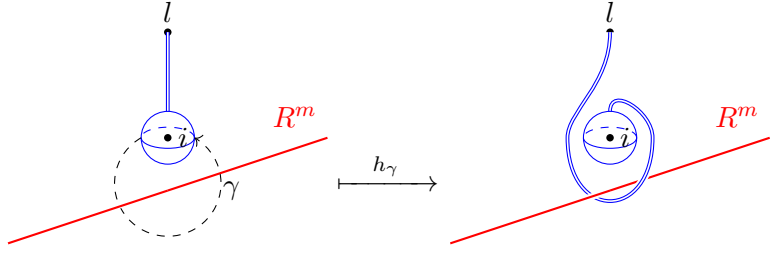
\begin{figure}[htbp]
    \centering
    \begin{equation*}
      \begin{tikzpicture}[scale=.5, baseline=0]
        \fill (0,3) circle[radius=2pt,blue] node[above] (l) {$l$};
        \draw[blue,thick] (l) -- (0,1.5);

        \draw (0,1) arc[radius=1, start angle=90, end angle=270];
        \draw[draw=white,ultra thick,double=red] (-3,-1) -- (3,1) node[above left] {\color{red} $R^{m}$};
        \node at (1.2,0) {$\gamma$};
        \draw[->,white,ultra thick,postaction={draw=black,thin,->}] (0,-1) arc[radius=1, start angle=270, end angle=420];

        \path[draw=blue,fill=white] (0,1) circle[radius=.5];
        \draw[blue] (-.5,1) arc[x radius=.5, y radius=.2, start angle=180, end angle=360];
        \draw[blue, dashed] (.5,1) arc[x radius=.5, y radius=.2, start angle=0, end angle=180];
        \fill (0,1) circle[fill,radius=2pt] node[right] (i) {$i$};
      \end{tikzpicture}
      \xmapsto{\quad h_{\gamma} \quad}
      \begin{tikzpicture}[scale=.5, baseline=0]
        \draw[red,thick] (-3,-1) -- (3,1) node[above left] {\color{red} $R^{m}$};
        \fill (0,3) circle[radius=2pt,blue] node[above] (l) {$l$};

        \draw[white, line width=3pt] (l)
        .. controls (0,2) and (-.8,1.5)
        .. (-.8,1)
        .. controls (-.8,.2) and (-.4,-.2)
        .. (0,-.2)
        .. controls (.4,-.2) and (.8,.2)
        .. (.8,1)
        .. controls (.8,1.5) and (0,2)
        .. (0,1.5);

        \draw[blue,thick] (l)
        .. controls (0,2) and (-.8,1.5)
        .. (-.8,1)
        .. controls (-.8,.2) and (-.4,-.2)
        .. (0,-.2)
        .. controls (.4,-.2) and (.8,.2)
        .. (.8,1)
        .. controls (.8,1.5) and (0,2)
        .. (0,1.5);

        \draw[white, line width=2.5pt] (-3,-1) -- (0,0);
        \draw[red,thick] (-3,-1) -- (0,0);

        \draw[blue] (0,1) circle[radius=.5];
        \draw[blue] (-.5,1) arc[x radius=.5, y radius=.2, start angle=180, end angle=360];
        \draw[blue, dashed] (.5,1) arc[x radius=.5, y radius=.2, start angle=0, end angle=180];
        \fill (0,1) circle[fill,radius=2pt] node[right] (i) {$i$};
      \end{tikzpicture}
    \end{equation*}
    \caption{The effect of $h_\gamma$ on the fundamental class (in blue) of $S^{n-1}$}
    \label{fig:effect-h-gamma}
  \end{figure}

  It is clear that $h_{\gamma}$ does not affect the fundamental class of $S^{n-m-1} = S^{1}$, as we can choose a representative with the $l$th point rotating around the axis $R^{m}$ outside the ball $B$.
  The class of the $i$th $S^{n-1}$ in the fiber corresponds to the $l$th point rotating around the $i$th point.
  This can be represented by concatenating a path $\eta_{ij}$ from $l$ to $i$ with a small sphere $\sigma_{i}$ around $i$.
  Consider the path $\gamma_{i}$ given by the $i$th coordinate of the path $\gamma$.
  Then $(h_{\gamma})_{*}[S^{n-1}]$ can be represented by $\eta_{ij} \cdot \gamma_{i} \cdot \sigma_{i}$ (see Figure~\ref{fig:effect-h-gamma} for an example).
  But this $(n-1)$st homology class is homologous to $\eta_{ij} \cdot \sigma_{i}$ for any path $\gamma$.
\end{proof}

\begin{remark}
  \label{rmk:cohen-proof}
  If $n = m + 1$, then $W = \R^{n} \setminus \R^{n-1} \cong \R^{n} \sqcup \R^{n}$ is not even connected.
  However, we then have $\Conf_{W}(l) = \bigsqcup_{l = l'+l''} \Conf_{\R^{n}}(l') \times \Conf_{\R^{n}}(l'') \times_{\Sigma_{l'} \times \Sigma_{l''}} \Sigma_{l}$.
  Its Poincaré polynomial is $\sum_{l = l' + l''} \frac{l!}{(l')! (l'')!} \prod_{i=0}^{l'-1} (1+it^{n-1}) \prod_{j=0}^{l''-1} (1+jt^{n-1})$, which is the RHS of~\eqref{eq.poinc-poly} (here, $\prod_{i=0}^{l-1}(2+t^{n-1})$) by induction.
\end{remark}

The classical presentation of $\eV_{n}(l) \coloneqq H^{*}(\Conf_{\R^{n}}(l))$~\cite{Arnold1969,Cohen1976} is ($\deg \omega_{ij} = n-1$):
\begin{equation}
  \eV_{n}(l) = S(\omega_{ij})_{1 \leq i \neq j \leq l} \bigm/ (\omega_{ji} - (-1)^{n} \omega_{ij}, \, \omega_{ij}^{2},\, \omega_{ij} \omega_{jk} + \omega_{jk} \omega_{ki} + \omega_{ki} \omega_{ij}).
\end{equation}

\begin{definition}\label{def:vscV}
  We define an algebra, with generators $\eta_{i}$ of degree $n-m-1$:
  \begin{equation}
    \vscV_{mn}(0,l) \coloneqq \eV_{n}(l) \otimes S(\eta_{i})_{1 \leq i \le l} \bigm/ (\eta_{i}^{2}, \eta_{i} \omega_{ij} - \eta_{j} \omega_{ij}).
  \end{equation}
\end{definition}

Elements of $\eV_{n}(l)$ can be seen as linear combinations of simple oriented graphs on $l$ vertices, modulo orientation signs, double edges, and a local three-term Arnold relation.
In $\vscV_{mn}(0,l)$, each connected component is decorated by $1$ or $\eta$ (formally, vertices can be decorated and decorations can move along edges; two $\eta$ classes on the same vertex give $0$).

\begin{remark}
  This algebra is very similar to the Lambrechts--Stanley model $\mathsf{G}_{A}$~\cite{Idrissi2018b} applied to $A = H^{*}(D^{n} \setminus D^{m})$ with vanishing diagonal class.
\end{remark}

\begin{remark}\label{rmk:asked-by-referee2}
  The space $\Conf_{W}(l) = (\R^{n})^{l} \setminus \bigl( \bigcup_{i \neq j} \{x_{i}=x_{j}\} \cup \bigcup_{i=1}^{l} p_{i}^{-1}(\R^{m}) \bigr)$ is a real subspace arrangement.
  Its Betti numbers can be computed by~\cite[Theorem~A]{GoreskyMacPherson1988}.
  For even $n$ and $m$, this is also a complex arrangement and the ring structure of the cohomology can also be computed using~\cite[Section~5]{ConciniProcesi1995}.
\end{remark}

\begin{lemma}\label{lem:equal-poly}
  The Poincaré polynomial of $\vscV_{mn}(0,l)$ is the RHS of~\eqref{eq.poinc-poly}.
\end{lemma}

\begin{proof}
  Let $V_{d}$ be the graded vector space with basis $\langle 1, \omega_{1d}, \dots, \omega_{(d-1)d}, \eta_{d}\rangle$.
  Its Poincaré polynomial is $1 + t^{n-m-1} + (d-1)t^{n-1}$.
  It thus suffices to show that $\vscV_{mn}(0,l) \cong V_{1} \otimes \dots \otimes V_{l}$,
  i.e.\ that $\vscV_{mn}(0,l)$ has the basis $\omega_{i_{1}j_{1}} \dots \omega_{i_{r}j_{r}} \eta_{k_{1}} \dots \eta_{k_{s}}$, where $j_{1} < \dots < j_{r}$, $i_{a} < j_{a}$ and $k_{b} \neq j_{a}$ for all $a,b$ (that is, graphs where edges increase and are ordered by their target, no two edges have the same target, and no target is decorated by $\eta$).

  We use the theory of commutative PBW bases~\cite[Section~4.8]{PolishchukPositselski2005}.
  Generators of $\vscV_{mn}(0,l)$ are $\omega_{ij}$ for $1 \le i < j \le l$ and $\eta_{i}$ for $1 \le i \le l$.
  We order them by $\eta_{1} > \dots > \eta_{l}$, $\omega_{ij} < \omega_{kl}$ if $i < k$ or $i = k \wedge j < l$, and e.g.\ $\eta_{k} < \omega_{ij}$ for all $i,j,k$.
  Monomials are ordered lexicographically.
  Quadratic relations, written as ``rewriting rules'' (higher monomial = lower monomials), are $\omega_{jk}\omega_{ik} = (-1)^{n+1}\omega_{ij} \omega_{jk} - \omega_{ik}\omega_{ij}$ for all $i < j < k$, $\eta_{j} \omega_{ij} = \eta_{i} \omega_{ij}$ and $\omega_{ij}^{2} = 0$ for all $i < j$, and $\eta_{i}^{2} = 0$ for all $i$.
  The claimed basis is exactly the PBW generating set.

  We need to check that overlapping rewriting rules are confluent, i.e.\ we get to the same element using only the rewriting rules.
  The critical monomials are $\omega_{lk}\omega_{jk}\omega_{ik}$, $\omega_{ij}^{3}$, $\eta_{i}^{3}$, $\omega_{jk}^{2} \omega_{ik}$, $\eta_{k}\omega_{jk}\omega_{ik}$, $\eta_{j}\omega_{ij}^{2}$, and $\eta_{j}^{2}\omega_{ij}$ (where $i < j < k < l$).
  These are all easily seen to be confluent.
  For example, $\eta_{k}\omega_{jk}\omega_{ik}$ can be rewritten in two ways.
  The first way, $(\eta_{k}\omega_{jk})\omega_{ik} = \eta_{j}\omega_{jk}\omega_{ik} = \eta_{j}((-1)^{n+1}\omega_{ij} \omega_{jk} - \omega_{ik}\omega_{ij}) = (-1)^{n+1}\eta_{i}\omega_{ij}\omega_{jk} - \eta_{i}\omega_{ik}\omega_{ij}$.
  The second way, $\eta_{k}(\omega_{jk}\omega_{ik}) = \eta_{k}((-1)^{n+1}\omega_{ij} \omega_{jk} - \omega_{ik}\omega_{ij}) = (-1)^{n+1} \eta_{j}\omega_{ij}\omega_{jk}-\eta_{i}\omega_{ik}\omega_{ij} = (-1)^{n+1} \eta_{i}\omega_{ij}\omega_{jk}-\eta_{i}\omega_{ik}\omega_{ij}$.
\end{proof}

\begin{proposition}
  \label{prop:well-def}
  For $n-2 \ge m \ge 1$, we have a well-defined algebra map $\vscV_{mn}(0,l) \to H^{*}(\Conf_{W}(l))$ given by $\omega_{ij} \mapsto \theta^{*}_{ij}(\vol_{n-1})$ and $\eta_{i} \mapsto \alpha^{*}_{i}(\vol_{n-m-1})$.
\end{proposition}

\begin{proof}
  Clearly $\alpha^{*}_{i}(\vol_{n-m-1})^{2} = 0$, so we just need to check is that $\eta_{i} \omega_{ij} = \eta_{j} \omega_{ij}$.
  It is sufficient to check this on $\Conf_{W}(2)$.
  The product $S^{n-1} \times S^{n-m-1}$ maps into $\Conf_{W}(2)$: the $S^{n-1}$ describes the rotation of point $2$ around point $1$, while the $S^{n-m-1}$ describes the rotation of the pair around $\R^{m}$.
  By~\eqref{eq.poinc-poly}, $\dim H^{n-m-1+n-1}(\Conf_{2}(W)) \le 1$.
  Since $\eta_{1} \omega_{12}$ and $\eta_{2} \omega_{12}$ pair to $1$ with the pushforward of the fundamental class of $S^{n-1} \times S^{n-m-1}$, we get $\eta_{1} \omega_{12} = \eta_{2} \omega_{12}$.
\end{proof}

\begin{proposition}\label{prop:iso-vscv-h}
  For $n-2 \ge m \ge 1$, the map $\vscV_{mn}(0,l) \to H^{*}(\Conf_{W}(l))$ is an isomorphism.
\end{proposition}
\begin{proof}
  Our proof is inspired by the proof of~\cite[Theorem~4.9]{Sinha2013} (see e.g.~\cite{SinhaWalter2011} for other uses of the pairing defined below).
  Given Inequality~\eqref{eq.poinc-poly}, the universal coefficients theorem, and Proposition~\ref{prop:well-def}, it is sufficient to show that the map is injective over any field $\K \in \{ \Q, \mathbb{F}_{p} \}$.

  Let us define classes in $H_{*}(\Conf_{W}(l);\K)$ using
  ``solar system''~\cite[Section~2]{Sinha2013}.
  With this point of view, any forest consisting of binary trees whose roots are possibly decorated by a loop induces a class in $H_{*}(\Conf_{W}(l); \K)$.
  The difference with classical solar systems is that orbital centers are chosen far enough from $\R^{m}$, and that if a root is decorated by a loop then the orbital center itself orbits around $\R^{m}$ in the shape of $S^{n-m-1}$.
  For example, the embedding $S^{n-1} \times S^{n-m-1} \to \Conf_{W}(2)$ of Proposition~\ref{prop:well-def} corresponds to the decorated binary tree with two leaves.
  
  The homology classes induced by undecorated trees satisfy the Jacobi relation, as these homology classes come from the subspace $\Conf_{\R^{n}}(l)$ (with $\R^{n}$ being e.g.\ the upper half-space).
  The homology/cohomology pairing is given by the pairing between graphs (using the description after Definition~\ref{def:vscV}) and trees, which is nondegenerate by~\cite[Theorem~4.7]{Sinha2013}, tensored with the pairing between the loops and the classes $\eta$ (which is clearly nondegenerate).
  We therefore get that $\vscV_{mn}(l) \to H^{*}(\Conf_{W}(l))$ is injective, establishing the result.
\end{proof}

\begin{definition}
  We define, for integers $k,l \geq 0$, $\vscV_{mn}(k,l) \coloneqq \eV_{m}(k) \otimes \vscV_{mn}(0,l)$.
\end{definition}

For cosmetic reasons, for $m \ge 2$ we will write $\tilde{\omega}_{ij}$ for the generators of $\eV_{m}(k)$, to distinguish them from the generators of $\vscV_{mn}(0,l)$.
(Recall $\eV_{1}(k)$ is simply the algebra of functions on $\Sigma_{k}$.)
The CDGA $\vscV_{mn}(k,l)$ is equipped with the obvious action of $\Sigma_{k} \times \Sigma_{l}$ and we can therefore view $\vscV_{mn}$ as a bisymmetric collection.

\begin{proposition}
  There is an isomorphism of algebras $\vscV_{mn}(k,l) \cong H^{*}(\VSC_{mn}(k,l))$.
\end{proposition}
\begin{proof}
  This follows directly from the Künneth formula and the homotopy equivalence $\VSC_{mn}(k,l) \simeq \Conf_{\R^{m}}(k) \times \Conf_{\R^{n} \setminus \R^{m}}(l)$.
\end{proof}

We now turn to the cooperad structure of $\vscV_{mn}$.
In $\eV_{n}$, we have $\circ_{W}^{\vee}(\omega_{uv}) = 1 \otimes \omega_{uv}$ if $u,v \in W$, and $\circ_{W}^{\vee}(\omega_{uv}) = \omega_{[u][v]} \otimes 1$ otherwise.
Moreover, $\eV_{1}(U) = \Ass^{\vee}(U) = \R[\Sigma_{U}]^{\vee}$ is the algebra of functions on $\Sigma_{U} \coloneqq \Bij(U, \{1,\dots,\#U\})$, the set of linear orders on $U$.
The cocomposition of $\eV_{1}$ is dual to block composition:
\begin{equation}
  \label{eq:block}
  \Sigma_{U/W} \times \Sigma_{W} \to \Sigma_{U}, \quad (\sigma, \tau) \mapsto \sigma \circ_{W} \tau,
\end{equation}
where $\sigma \circ_{W} \tau$ is the linear order on $U$ obtained by inserting the order on $W$ at the position $*$ in the order $\sigma$,
e.g.\ $(a < * < b) \circ_{\{x,y\}} (x < y) = (a < x < y < b)$.

\begin{proposition}
  \label{prop:coop-G}
  For $m \ge 2$, the isomorphism of Proposition~\ref{prop:iso-vscv-h} induces cooperadic structure maps $\circ_{T}^{\vee} : \vscV_{mn}(U,V) \to \vscV_{mn}(U, V/T) \otimes \eV_{n}(T)$ and $\circ_{W,T}^{\vee} : \vscV_{mn}(U,V \sqcup T) \to \vscV_{mn}(U/W,V) \otimes \vscV_{mn}(W,T)$ given by:
  \begin{equation*}
    \circ_{T}^{\vee}(\tilde{\omega}_{uu'}) = \tilde{\omega}_{uu'} \otimes 1.
    \qquad
    \circ_{W,T}^{\vee}(\tilde{\omega}_{uu'}) = \begin{cases} 1 \otimes \tilde{\omega}_{uu'}, & \text{if } u,u' \in W; \\ \tilde{\omega}_{[u][u']} \otimes 1, & \text{otherwise.} \end{cases}
  \end{equation*}
  \begin{equation*}
    \circ_{T}^{\vee}(\omega_{vv'}) = \begin{cases} 1 \otimes \omega_{vv'}, & \text{if } v,v' \in T; \\ \omega_{[v][v']} \otimes 1 & \text{otherwise.} \end{cases}
    \qquad
    \circ_{W,T}^{\vee}(\omega_{vv'}) = \begin{cases} 1 \otimes \omega_{vv'}, & \text{if } v,v' \in T; \\ \omega_{vv'} \otimes 1 & \text{if } v,v' \in V; \\ 0 & \text{otherwise.} \end{cases}
  \end{equation*}
  \begin{equation*}
    \circ_{T}^{\vee}(\eta_{v}) = \eta_{[v]} \otimes 1.
    \qquad
    \circ_{W,T}^{\vee}(\eta_{v}) = \begin{cases} \eta_{v} \otimes 1 & \text{if } v \in V; \\ 1 \otimes \eta_{v} & \text{if } v \in T. \end{cases}
  \end{equation*}

  For $m = 1$, the maps $\circ_{T}^{\vee}$ and $\circ_{W,T}^{\vee}$ have the same behavior as above on the generators $\omega_{vv'}$ and $\eta_{v}$.
  On $\eV_{1}(U) = \Ass^{\vee}(U) = \R[\Sigma_{U}^{\vee}]$, the map $\circ_{T}^{\vee}$ is the identity, and $\circ_{W,T}^{\vee}$ is the dual of the block composition of Equation~\eqref{eq:block}.
\end{proposition}

\begin{proof}
  Our proof is similar to~\cite[Theorem~6.3]{Sinha2013}.
  Recall the maps $\theta_{uu'} : \VSC_{mn}(U,V) \to S^{m-1}$, $\theta_{vv'} : \VSC_{mn}(U,V) \to S^{n-1}$, and $\alpha_{v} : \VSC_{mn}(U,V) \to S^{n-m-1}$ from Section~\ref{sec.defin-comp-with}.
  Let us first check that, when they are composed with the insertion maps $\circ_{T} : \VSC_{mn}(U,V/T) \times \DD_{n}(T) \to \VSC_{mn}(U,V)$, we obtain the behavior indicated above on generators.
  For $\theta_{uu'}$ and $\theta_{vv'}$, this is identical to the proof that $\eV_{n}$ (resp.\ $\scV_{m}$) is the cohomology of $\DD_{n}$ (resp.\ $\SC_{m}$).
  Now, let us consider the composite $\alpha_{v}(- \circ_{T} -) : \VSC_{mn}(U,V/T) \times \DD_{n}(T) \to S^{n-m-1}$ for some $v \in V$.
  If $v \not\in T$, then the composite map $\alpha_{v}(- \circ_{T} -)$ is equal to $\alpha_{v} \circ \operatorname{proj}_{1}$,
  therefore:
  \begin{equation}
    \circ_{T}^{\vee}(\eta_{v}) = (\alpha_{v}(- \circ_{T} -))^{*}(\vol_{n-m-1}) = \operatorname{proj}_{1}^{*}(\alpha_{v}^{*}(\vol_{n-m-1})) = \eta_{v} \otimes 1.
  \end{equation}
  If $v \in T$, consider the homotopy which precomposes the embedding indexed by $[v] \in V/T$ with $x \mapsto tx$.
  At the limit $t = 0$, we find the map $\alpha_{*} \circ \operatorname{proj}_{1}$.
  Since the homotopy class of the map is constant as $t$ varies, we also get $\circ_{T}^{\vee}(\eta_{v}) = \eta_{[v]} \otimes 1$.

  Now consider $\circ_{W,T} : \VSC_{mn}(U/W,V) \times \VSC_{mn}(W,T) \to \DD_{mn}(U,V \sqcup T)$ and let us check that we get the maps from the proposition.
  For $\theta_{uu'}$, this is again identical to the computation for $\eV_{n}$.
  For $\eta_{v}$, it is easy to see that $\eta_{v}(- \circ_{W,T} -)$ factors by the projection on one of the two factors in the product.
  Similarly, for $\theta_{vv'}$, if $v,v'$ are both in $V$ or both in $T$, then $\theta_{vv'}(- \circ_{W,T} -)$ also factors by one of the projections.
  Otherwise, WLOG assume that $v \in V$ and $v' \in T$.
  If we contract the appropriate disks by a homotopy which linearly decreases the radius, then the limit is a constant map, thus $\circ_{W,T}^{\vee}(\omega_{vv'}) = 0$.
\end{proof}

\begin{remark}\label{rmk:vsc-esc}
  We can compare $\vscV_{mn}$ with the cohomology of $\ESC_{mn}$ (see Remark~\ref{rmk:difference}).
  We have $H^{*}(\ESC_{mn}(U,V)) = \eV_{m}(U) \otimes_{\eV_{n}(U)} \eV_{n}(U \sqcup V)$, where $\eV_{n}(U) \to \eV_{n}(U \sqcup V)$ is the obvious inclusion and $\eV_{n}(U) \to \eV_{m}(U)$ sends all the generators $\omega_{uu'}$ to zero~\cite[Proposition~4.1]{Willwacher2017a}.
  Thus $H^{*}(\ESC_{mn}(U,V)) = \eV_{m}(U) \otimes \eV_{n}(U \sqcup V) / \eV_{n}(U)$.
  The cooperadic structure maps are defined  by formulas similar to Proposition~\ref{prop:coop-G} (forgetting the $\eta_{v}$).
  The inclusion $\VSC_{mn} \subset \ESC_{mn}$ induces on cohomology the composite $\eV_{m}(U) \otimes \eV_{n}(U \sqcup V)/\eV_{n}(U) \twoheadrightarrow \eV_{m}(U) \otimes \eV_{n}(V) \hookrightarrow \vscV_{mn}(U,V)$.
  An interesting question would be whether this inclusion is formal.
\end{remark}

\begin{remark}\label{rmk:why-not}
  The equality $\VSC_{mn}(-, \varnothing) = \DD_{m}$ induces a left $\DD_{m}$-module structure on $\VSC_{mn}(\varnothing,-)$, using the operad structure of $\VSC_{mn}$.
  This structure dualizes to the map $\Delta : \vscV_{mn}(\varnothing, \bigsqcup_{u \in U} V_{u}) \to \vscV_{mn}(U, \varnothing) \otimes \bigotimes_{u \in U} \vscV_{mn}(\varnothing, V_{u})$ given by $\Delta(\eta_{v}) = 1 \otimes \eta_{v}$ (put in the corresponding $\vscV_{mn}(\varnothing, V_{u})$), $\Delta(\omega_{vv'}) = 1 \otimes \omega_{vv'}$ if $v$ and $v'$ are in the same $V_{u}$ and $\Delta(\omega_{vv'}) = 0$ otherwise.
\end{remark}

\subsection{Generators and relations for the homology}
\label{sec:pres-homol}

We take the opportunity to describe a presentation of $\vsc_{mn} \coloneqq H_{*}(\VFM_{mn})$ by generators and relations.
Let us first recall the presentation of $\ee_{n} \coloneqq H_{*}(\FM_{n})$, due to Cohen~\cite{Cohen1976}.
An algebra over $\ee_{1}$ is a unital associative algebra.
For $n \ge 2$, an algebra over $\ee_{n}$ is a unital Poisson $n$-algebra, i.e.\ a unital commutative algebra equipped with a Lie bracket of (cohomological) degree $1-n$ which is a biderivation for the product.

\begin{proposition}
  \label{prop:pres-homology}
  For $n-2 \ge m \ge 1$, an algebra over $\vsc_{mn}$ is the data $(A,B,f,\delta)$ consisting of an $\ee_{m}$-algebra $A$, an $\ee_{n}$-algebra $B$, a central morphism of algebras $f : B \to A$, and its central derivation $\delta : B[n-m-1] \to A$.
\end{proposition}

Central means that $f$ and $\delta$ land in the center $Z(A) = \{ a \in A \mid \forall b \in A, \, [a,b] = 0\}$, where the bracket is the graded commutator ($m = 1$) or the shifted Lie bracket ($m \ge 2$).
The map $\delta$ is a derivation with respect to $f$: $\delta(xy) = \delta(x) f(y) \pm f(x) \delta(y)$.

\begin{proof}
  An algebra is a pair $(A,B)$ where $B$ is an $\ee_{n}$-algebra and $A$ is a $\vsc_{mn}$-algebra.
  Since $\vsc_{mn}(-,\varnothing) = \ee_{m}$, we know that $A$ is an $\ee_{m}$-algebra.
  Given that $\vsc_{mn}(0,1) = H_{*}(S^{n-m-1})$, we obtain two classes $f : B \to A$ of degree $0$ and $\delta : B[n-m-1] \to A$ of degree $1+m-n$.
  The two possible compositions $\vsc_{mn}(2,0) \times \vsc_{mn}(0,1) \to \vsc_{mn}(1,1)$ in either input of the product are homotopic, so $f$ and $\delta$ are central.
  Similarly, $f$ is a morphism of algebras.
  The proof that $\delta$ is a derivation is the same as the proof that the Lie bracket in $\ee_{n}$ is a biderivation.

  We thus get a suboperad of $\vsc_{mn}$ that contains exactly the data of the proposition.
  Elements of this suboperad can be described as forests of rooted binary trees (like elements of $\ee_{n}$) with roots possibly decorated by a loop of degree $1+m-n$ (the element $\delta$).
  These are exactly the trees that appeared in the proof of Proposition~\ref{prop:iso-vscv-h} and we proved there that they generate the whole homology.
\end{proof}

We can rephrase Proposition~\ref{prop:pres-homology} more compactly.
Let $A$ be an $\ee_{m}$-algebra and consider $A[\varepsilon] \coloneqq A \otimes \R[\varepsilon] / (\varepsilon^{2})$ where $\deg \varepsilon = n-m-1$.
If $m \ge 2$, then there is a Poisson bracket on $A[\varepsilon]$ given by $[x + \varepsilon y, x' + \varepsilon y'] = [x,x'] + \varepsilon ([x,y'] \pm [x',y])$.
A $\vsc_{mn}$-algebra is the data of an $\ee_{m}$-algebra $A$, an $\ee_{n}$-algebra $B$, and a central morphism $f + \varepsilon \delta : B \to A[\varepsilon]$.

Compare this result with the $\infty$-categorical counterparts from~\cite[Section~4.3]{AyalaFrancisTanaka2017a}.
An algebra over the $\infty$-categorical version $\mathscr{D}isk^{\fr}_{m \subset n}$ consist of a $\mathscr{D}isk^{\fr}_{m}$-algebra $A$, a $\mathscr{D}isk^{\fr}_{n}$-algebra $B$, and a morphism of $\mathscr{D}isk^{\fr}_{m+1}$-algebras $\alpha : \int_{S^{n-m-1}} B \to \operatorname{HC}^{*}_{\DD_{m}}(A)$, where $\int_{S^{n-m-1}} B$ is the ``factorization homology'' of $S^{n-m-1}$ with coefficients in $B$, and $\operatorname{HC}^{*}_{\DD_{m}}$ refers to Hochschild cochains.
We view this as a ``up to homotopy'' version of an $\vsc_{mn}$-algebra, the morphism $f + \varepsilon \delta$ being the part $\int_{S^{n-m-1}}B \to \operatorname{HC}^{0}_{\DD_{m}}(A)$ and the higher terms being homotopies.
It would be interesting to make this observation precise.

\begin{remark}
  \label{rmk:livernet}
  An algebra over the homology $\ssc_{n} \coloneqq H_{*}(\SC_{n})$ of the Swiss-Cheese operad is the data of $(A,B,f)$ as in the proposition, see~\cite[Section~4.3]{Livernet2015}.
  However, our computation for $H^{*}(\VSC_{(n-1)n})$ in Section~\ref{sec.comp-cohom} above does not apply (e.g.\ the class $\omega_{12} \in H^{*}(\VSC_{(n-1)n}(0,2))$ vanishes on some connected components).
  Instead, we get that $\vsc_{(n-1)n}$-algebras are given by quadruples $(A,B,f,g)$ where $(A,B,f)$ are as above and $g : B \to A$ is another central morphism (i.e.\ $A$ is a unitary $B$-bimodule).
  There is an embedding $\SC_{n} \subset \VSC_{(n-1)n}$.
  On homology, an $\vsc_{(n-1)n}$-algebra $(A,B,f,g)$ viewed as an $\ssc_{n}$-algebra is simply $(A,B,f)$, i.e.\ we forget the right action.
  Livernet~\cite[Theorem~4.3]{Livernet2015} proved that $\SC_{n}$ is not formal by exhibiting a nontrivial operadic Massey product $\langle \mu_{B}, f, \lambda_{A} \rangle$, where $\mu_{B}$ represents the product of $B$ and $\lambda_{A}$ represents the Lie bracket of $A$.
  This shows that the operad of chains $C_{*}(\VFM_{(n-1)n};\Q)$ cannot be formal either, because we obtain a nontrivial Massey product there too.
\end{remark}

\section{Graph complexes}
\label{sec.graph-complexes}

In this section, we define a bicolored Hopf cooperad, whose operations in the second color are given by Kontsevich's cooperad $\graphs_{n}$~\cite{Kontsevich1999}, and whose operations in the first color will be called $\vgraphs_{mn}$.
Our tool of choice to define $\vgraphs_{mn}$ will be ``operadic twisting''~\cite{Willwacher2014,DolgushevWillwacher2015}, just like in~\cite{Willwacher2015a}.
To give an idea of how $\vgraphs_{mn}$ is built, we first recall the steps in the definition of $\graphs_{n}$.
We also refer to~\cite[Sections~1.5--1.6]{Idrissi2018b} for more details with matching notation.

\begin{remark}
  We use the notation $\graphs_{n}$ for the cooperad rather than its dual operad.
  Its linear dual $\graphs_{n}^{\vee}$ is an operad which is quasi-isomorphic to $H_{*}(\DD_{n})$.
  We make this choice because we never deal with $\graphs^{\vee}_{n}$.
  Kontsevich's graph complex $\GC_{n}^{\vee}$ where the differential splits vertices will also be denoted with a dual sign.
  See e.g.~\cite{Idrissi2018b,CamposIdrissiLambrechtsWillwacher2018} for matching notations.
\end{remark}

\subsection{Recollections: the cooperad \texorpdfstring{$\graphs_n$}{graphs\_n}}
\label{sec.recoll-coop-graphs_n}

\paragraph{Untwisted}

The first step is to define an untwisted Hopf cooperad $\Gra_{n}$, given in each arity by the following CDGA, with generators $e_{uv}$ of degree $n-1$:
\begin{equation}
  \label{eq:4}
  \Gra_{n}(U) \coloneqq S(e_{uv})_{u, v \in U} \bigm/ (e_{vu} = (-1)^{n} e_{uv}).
\end{equation}
Graphically, $\Gra_{n}(U)$ is spanned by graphs on the set of vertices $U$.
The monomial $e_{u_{1}v_{1}} \dots e_{u_{r}v_{r}}$ corresponds to the graph with edges $\overrightarrow{u_{1}v_{1}}$, \dots, $\overrightarrow{u_{r}v_{r}}$.
The identification $e_{vu} = \pm e_{uv}$ allows us to view the graphs as undirected, although we need directions to define the signs precisely for odd $n$.
If $n$ is even then $\deg e_{uv}$ is odd, so we need to order edges to get precise signs.
We explicitly allow tadpoles $(e_{uu})$ and double edges $(e_{uv}^{2})$.
However, for even $n$, $e_{uv}^{2} = 0$ because $\deg e_{uv}$ is odd; and for odd $n$, $e_{uu} = (-1)^{n} e_{uu} = -e_{uu}$ thus $e_{uu} = 0$.
The product is given by gluing graphs along their vertices.
The cooperadic structure is given by subgraph contraction.
Explicitly, the map $\circ_{W}^{\vee} : \Gra_{n}(U) \to \Gra_{n}(U/W) \otimes \Gra_{n}(W)$ is given by $\circ_{W}^{\vee}(e_{uv}) = e_{**} \otimes 1 + 1 \otimes e_{uv}$  if $u,\,v \in W$, and by $\circ_{W}^{\vee}(e_{uv}) = e_{[u][v]} \otimes 1$ otherwise.

\begin{remark}\label{rmk:asked-by-referee}
  The quotient of $\Gra_{1}$ by disconnected graph forms an operad denoted by $\mathcal{G}r$ in~\cite{SinhaWalter2011}.
  It was used to present Harrison cohomology and find bases for cofree Lie coalgebras.
\end{remark}

One may then produce a first zigzag of Hopf cooperads:
\begin{equation}
  \label{eq:3}
  H^{*}(\FM_{n}) = \eV_{n} \gets \Gra_{n} \xrightarrow{\omega'} \OmPA^{*}(\FM_{n}),
  \quad
  \omega_{uv} \mapsfrom e_{uv} \mapsto p_{uv}^{*}(\varphi_{n}),
\end{equation}
where $p_{uv} : \FM_{n}(U) \to  \FM_{n}(2)$ is the projection and $\varphi_{n}$ is the ``propagator'':
\begin{equation}
  \label{eq:propagator1}
  \varphi_{n} \coloneqq \vol_{n-1} = \mathrm{cst} \cdot \sum_{i = 1}^{n} \pm x_{i} dx_{1} \wedge \ldots \widehat{dx}_{i} \ldots \wedge dx_{n} \in \OmPA^{n-1}(\FM_{n}(2)) = \OmPA^{n-1}(S^{n-1}).
\end{equation}

Given a graph $\Gamma \in \Gra_{n}(U)$, one may define its \emph{coefficient} $\mu(\Gamma)$ by $\mu(\Gamma) = 0$ if $\# U \le 1$ and $\mu(\Gamma) = \int_{\FM_{n}(U)} \omega'(\Gamma)$ otherwise.
This element $\mu$ has a simple description: it vanishes on all graphs, except for the one with exactly two vertices and an edge between them~\cite[Lemma~9.4.3]{LambrechtsVolic2014}.
In other words, in the dual basis:
\begin{equation}
  \label{eq:mu}
  \mu =
  \begin{tikzpicture}[scale=.3, baseline=(a.base)]
    \node[exta] (a) {1};
    \node[exta, right=1cm of a] (b) {2};
    \draw (a) -- (b);
  \end{tikzpicture}
  \; \in \Gra_{n}^{\vee}(2) \subset \prod\nolimits_{i \ge 0} \Gra_{n}^{\vee}(i).
\end{equation}

\paragraph{Twisting}

The second step is to twist the cooperad $\Gra_{n}$.
Briefly, the idea of twisting is to turn an operad $\PP$ that encodes algebras that are also Lie algebras (i.e.\ there is a morphism $\Lie \to \PP$) into an operad that encode $\PP$-algebras twisted by a Maurer--Cartan element, see~\cite[Appendix~I]{Willwacher2014}.
In the case of $\Gra_{n}$, this turns out to produce the Arnold relations, which are themselves the algebraic shadow of the stratification of $\FM_{n}$.
See Appendix~\ref{sec:twisting-cooperads} for concepts and notations.

The element $\mu$ of~\eqref{eq:mu} defines a morphism from $\Lie_{n}$ to $\Gra_{n}^{\vee}$: it sends the generating bracket $\lambda_{2} \in \Lie_{n}(2)$ to the graph appearing in~\eqref{eq:mu}.
By composing with the canonical resolution $\hoLie_{n} \to \Lie_{n}$, we thus obtain a morphism $\hoLie_{n} \to \Gra_{n}^{\vee}$, which sends the binary bracket to $\mu$ and the $k$-ary brackets to $0$ for $k \ge 3$.
One can check immediately that this graph satisfies the Jacobi relation.
Alternatively, this can be deduced from the definition as an integral and the decomposition of the boundary of $\FM_{n}(U)$, cf.\ Proposition~\ref{prop:c-mor-holie}.

Cooperadic twisting produces a dg-cooperad $\Tw\Gra_{n}$ from $\Gra_{n}$ and $\mu$.
Let us now describe it.
As a graded module,
\begin{equation}
  \label{eq:8}
  \Tw \Gra_{n}(r) \coloneqq \biggl( \bigoplus_{i \geq 0} \bigl( \Gra_{n}(r+i) \otimes \K[n]^{\otimes i} \bigr)_{\Sigma_{i}}, d_{\mu} \biggr).
\end{equation}
The CDGA structure is defined using the fact that $\Gra_{n}(\varnothing) = \R$ (see~\cite[Lemma~9]{Idrissi2018b}) and the fact that $\mu$ vanishes on disconnected graphs as well as graphs admitting a cut point (i.e.\ if a point that disconnects the graph when removed).

Let us now give a graphical interpretation of this definition.
The CDGA $\Tw\Gra_{n}(U)$ is spanned by graphs with two kinds of vertices: external vertices, which are in bijection with $U$, and internal vertices, which are indistinguishable (in pictures they will be drawn in black).
Given a graph $\Gamma$, its differential $d\Gamma = \sum_{e} \pm \Gamma/e$ is obtained as a sum, over all the edges $e \in E_{\Gamma}$ connected to at least one  internal vertex, of the graphs obtained by contracting these edges.
Note that in this differential, edges connected to univalent vertices are not contracted, roughly speaking because the contraction appears twice in $d\Gamma$ and cancels out; but if both endpoints of the edge are univalent, then the edge is contracted with a minus sign, see~\cite[Appendix~I.3]{Willwacher2014} and Description~\ref{descr:diff-gcn}.
The product of two graphs is the graph obtained by gluing them along their external vertices.
The cooperadic structure is given by subgraph contraction (summing over all choices of whether internal vertices and edges are in the subgraph).

One checks that the zigzag of Equation~\eqref{eq:3} extends to a zigzag:
\begin{equation}
  \label{eq:9}
  \eV_{n} \gets \Tw\Gra_{n} \xrightarrow{\omega} \OmPA^{*}(\FM_{n}).
\end{equation}
This follows from the results of~\cite[Section~9--10]{LambrechtsVolic2014} and~\cite[Section~3]{Willwacher2014}.
The left-pointing map sends all graphs with internal vertices to zero.
The right-pointing map is given by the following integrals.
Given a graph $\Gamma \in \Gra_{n}(U \sqcup I) \subset \Tw\Gra_{n}(U)$, the form $\omega(\Gamma)$ is the integral of $\omega'(\Gamma)$ along the fiber of the PA bundle $p_{U} : \FM_{n}(U \sqcup I) \to \FM_{n}(U)$ which forgets points in the configuration:
\begin{equation}
  \label{eq:10}
  \omega(\Gamma) \coloneqq (p_{U})_{*}(\omega'(\Gamma)) = \int_{\FM_{n}(U \sqcup I) \to \FM_{n}(U)} \omega'(\Gamma).
\end{equation}

\paragraph{Graph complex $\fGC_{n}$}

We record the following definition for future use.

\begin{definition}
  \label{def:fgcn}
  The full graph complex $\fGC_{n}$ is defined to be $\Tw\Gra_{n}(\varnothing)[-n]$.
  It is spanned by graphs with only internal vertices.
  The degree of $\gamma \in \fGC_{n}$ is $\deg \gamma = (n-1) \#E_{\gamma} - n \#V_{\gamma} + n$.
  We describe its differential below.
\end{definition}

\begin{remark}\label{rmk:degree-shift}
  This degree shift by $n$ comes from the fact that we mod out by $\R^{n} \rtimes \R_{>0}$ in the definition of $\FM_{n}$.
  To get a Maurer--Cartan element (of degree $1$) out of $\omega$ in the Lie algebra $\GC_{n}^{\vee}$ defined below, a shift by $n$ is thus used.
\end{remark}

The module $\fGC_{n}$ is a shifted algebra, with a product given by disjoint union of graphs.
If $\GC_{n} \subset \fGC_{n}$ is the subcomplex of connected graphs, then there is an isomorphism of shifted algebras $\fGC_{n} = S(\GC_{n}[n])[-n]$.
The module $\GC_{n}$ is a (pre)Lie coalgebra.
Its cobracket $\Delta$ is given by subgraph contraction (i.e.\ it is inherited from the cooperad structure on $\Gra_{n}$).
Dually, $\GC_{n}^{\vee}$ is a (pre)Lie algebra, with a bracket given by graph insertion.

\begin{mydescription}
  \label{descr:diff-gcn}
  Let us now describe the differential of $\GC_{n}$.
  It is given by $(-\mu \otimes 1 + 1 \otimes \mu)\Delta$, where $\mu \in \GC_{n}^{\vee}$ is the Maurer--Cartan element defined in Equation~\eqref{eq:mu}.
  The summand $(1 \otimes \mu) \Delta$ is given by the sum of contractions of edges, while the summand $(\mu \otimes 1) \Delta$ is given by the sum of contractions of edges attached to a univalent vertex.
  The contraction of an edge attached to exactly one univalent vertex appears twice in the sum with opposite signs, which thus cancels out.
  However, if both endpoints of the edge are univalent, then the contractions of this edge appears once with a plus and twice with a minus, which leaves one term with a minus.
  (See~\cite[Appendix~I.3]{Willwacher2014} for more details)
\end{mydescription}

Dually, the differential on $\GC_{n}^{\vee}$ is $[\mu, -]$.
It is roughly speaking given by vertex splitting (with the caveat about univalent vertices explained above).
The space $\Tw\Gra_{n}(U)$ is a module over the shifted CDGA $\fGC_{n}$, by taking disjoint unions.

\paragraph{Reduction}

The next step is to mod out graphs with \emph{internal components}, i.e.\ connected components with only internal vertices, to obtain a new Hopf cooperad $\Graphs_{n}$.
Formally, we consider the tensor product
\begin{equation}
  \label{eq:17}
  \Graphs_{n}(U) \coloneqq \Tw\Gra_{n}(U) \otimes_{\fGC_{n}} \R,
\end{equation}
where $\fGC_{n}$ acts trivially on $\R$.
In other words, in $\Graphs_{n}$, a graph with a connected component consisting entirely of internal vertices is set equal to zero.
The map $\omega$ of Equation~\eqref{eq:10} factors through the quotient defining $\Graphs_{n}$~\cite[Lemma~9.3.7]{LambrechtsVolic2014}, and the quotient map $\Tw\Gra_{n} \to \eV_{n}$ clearly does.

We can moreover reduce further the operad.
The quotient $\graphs_{n}$ is given by modding out graphs containing: internal vertices that are univalent or bivalent; double edges; or tadpoles.
It follows again from the lemmas of~\cite[Section~9.3]{LambrechtsVolic2014} that $\omega$ factors through this quotient (and the map $\Graphs_{n} \to \eV_{n}$ clearly does).

Recall that even though $\OmPA^{*}(\FM_{n})$ is not a Hopf cooperad, we take the convention that a morphism into it is a morphism in the sense defined in Section~\ref{sec:pa-forms}.

\begin{theorem}[{\cite{Kontsevich1999}, \cite[Chapter~10]{LambrechtsVolic2014}}]
  \label{thm:fml-kont}
  This defines quasi-isomorphisms of Hopf cooperads $\eV_{n} \xleftarrow{\sim} \graphs_{n} \xrightarrow{\sim} \OmPA^{*}(\FM_{n})$, thus $\FM_{n}$ is formal over $\R$.
\end{theorem}

\subsection{The cooperad \texorpdfstring{$\VGraphs_{mn}$}{VGraphs\_mn}}
\label{sec.cooperad-sgraphs_mn}

Let us now define $\VGraphs_{mn}$, using the same methodology that was used to define $\Graphs_{n}$.
Note that we must take special care of the case $m = 1$.
We define the further reduced cooperad $\vgraphs_{mn}$ in the next section.

\paragraph{Untwisted}

The first step is the definition of the untwisted graph cooperad.

\begin{definition}
  For $n-2 \ge m \ge 2$, the untwisted graph cooperad is a relative $\VGra_{mn}$-cooperad given in each arity by ($\deg \tilde{e}_{uu'} = m-1$ and $\deg e_{ij} = n-1$):
  \[ \VGra_{mn}(U,V) \coloneqq S(\tilde{e}_{uu'})_{u,u' \in U} \otimes S(e_{ij})_{i,j \in U \sqcup V} \bigm/ (e_{ji} = (-1)^{n} e_{ij}, \tilde{e}_{uu'} = (-1)^{m} \tilde{e}_{u'u}). \]
\end{definition}

\begin{definition}
  \label{def:vgra-1}
  For $n \ge 3$, we let (where $\Sigma_{U} = \Ass(U) = \Bij(U, \{1, \dots, \#U\})$):
  \[ \VGra_{1n}(U,V) \coloneqq S(e_{ij})_{i,j \in U \sqcup V} \otimes \R[\Sigma_{U}]^{\vee} \bigm/ (e_{ji} = (-1)^{n} e_{ij}). \]
\end{definition}

Let us now give a graphical interpretation of $\VGra_{mn}(U,V)$.
We will concentrate first on the case $m \ge 2$.
As a vector space, it is spanned by graphs with two kinds of vertices: terrestrial (in bijection with $U$) and aerial (in bijection with $V$).
We will draw the aerial vertices as circles, and the terrestrial vertices as semicircles, below the aerial ones.
There are also two kind of edges: full edges (corresponding to the $e_{ij}$) between any two vertices, and dashed edges (corresponding to the $\tilde{e}_{uu'}$) between two terrestrial vertices.
Note that we allow tadpoles and double edges.
See Figure~\ref{fig:exa-sgra} for an example.

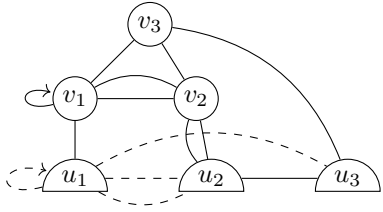
\begin{figure}[htbp]
  \centering
  \begin{tikzpicture}
    \node[exta] (a1) {$v_{1}$};
    \node[exta, right = 1cm of a1] (a2) {$v_{2}$};
    \node[exta, above right = .8cm of a1] (a3) {$v_{3}$};
    \node[extt, below = .5cm of a1] (t1) {$u_{1}$};
    \node[extt, right = 1cm of t1] (t2) {$u_{2}$};
    \node[extt, right = 1cm of t2] (t3) {$u_{3}$};
    \draw
    (a1) edge (a2) edge[bend left] (a2) edge (a3) edge[loop left] () edge (t1)
    (a2) edge (a3) edge (t2) edge[bend right] (t2)
    (t3) edge (t2) edge[bend right] (a3);
    \draw[dashed] (t1) edge (t2) edge [bend right] (t2) (t1) edge[loop left] () (t1) to[bend left] (t3);
  \end{tikzpicture}
  \caption{Example of a graph in $\VGra_{mn}$ (for $m \ge 2$)}
  \label{fig:exa-sgra}
\end{figure}

If $m = 1$, the interpretation is similar, except that there are no dashed edges, and in addition the set of terrestrial vertices is ordered.
Note that we are implicitly using the dual of the canonical basis of $\R[\Sigma_{U}]$ (i.e.\ we consider the basis elements of $\R[\Sigma_{U}]^{\vee}$ which vanish on all linear orders except one).

The product glues graphs along their vertices, and the differential is zero.
These CDGAs assemble to form a relative Hopf $\Gra_{n}$-cooperad.
The cooperadic structure maps act on the generators by the following formulas (compare with Proposition~\ref{prop:coop-G}), with the convention $u,u' \in U$, $v,v' \in V$:
\begin{align*}
  \circ_{T}^{\vee}(\tilde{e}_{uu'}) & = \tilde{e}_{uu'} \otimes 1.
  & \circ_{W,T}^{\vee}(\tilde{e}_{uu'}) & = \begin{cases} \tilde{e}_{**} \otimes 1 + 1 \otimes \tilde{e}_{uu'}, & \text{if } u,u' \in W; \\ \tilde{e}_{[u][u']} \otimes 1, & \text{otherwise.} \end{cases}
\end{align*}
\[
  \circ_{T}^{\vee}(e_{vv'}) =
  \begin{cases}
    e_{**} \otimes 1 + 1 \otimes e_{vv'}, \text{if } v,v' \in T; \\
    \circ_{T}^{\vee}(e_{vv'}) = e_{[v][v']} \otimes 1, \text{ otherwise.}
  \end{cases}
  \quad
  \circ_{W,T}^{\vee}(e_{vv'}) = \begin{cases} 1 \otimes e_{vv'}, & \text{if } v,v' \in T; \\ e_{vv'} \otimes 1 & \text{if } v,v' \in V; \\ 0 & \text{otherwise.} \end{cases} 
\]
\begin{align*}
  \circ_{T}^{\vee}(e_{uu'}) & = e_{uu'} \otimes 1,
  & \circ_{W,T}^{\vee}(e_{uu'}) & = \begin{cases} e_{**} \otimes 1 + 1 \otimes e_{uu'}, & \text{if } u,u' \in W, \\ e_{[u][u']} \otimes 1, & \text{otherwise}, \end{cases} \\
  \circ_{T}^{\vee}(e_{uv}) & = e_{u[v]} \otimes 1,
  & \circ_{W,T}^{\vee}(e_{uv}) & = \begin{cases} e_{**} \otimes 1 + 1 \otimes e_{uv}, & \text{if } u \in W, v \in T \\ e_{u*} \otimes 1, & \text{if } u \not\in W, v \in T, \\ e_{[u]v} \otimes 1 & \text{otherwise}. \end{cases}
\end{align*}
Moreover, if $m = 1$, then $\circ_{T}^{\vee}$ is the identity on $\R[\Sigma_{U}]^{\vee}$, and $\circ_{W,T}^{\vee}$ is the dual of block composition from Equation~\eqref{eq:block}.
Graphically, the cooperadic structure map $\circ_{W,T}$ is given by subgraph contraction, which we now explain.

\begin{definition}[Quotient graph, case $m \ge 2$]
  \label{def:quotient-graph}
  Let $\Gamma \in \VGra_{mn}(U,V)$ be a graph and $\Gamma' \subset \Gamma$ be a subgraph, not necessarily full.
  We define the quotient graph $\Gamma / \Gamma'$ as follows.
  The set of vertices of $\Gamma/\Gamma'$ is the quotient set $V_{\Gamma} / V_{\Gamma'}$, identifying all the vertices of $\Gamma'$ to produce a new vertex $[\Gamma']$, always terrestrial even if $\Gamma'$ contains no terrestrial vertices.
  The set of edges of $\Gamma/\Gamma'$ is the difference $E_{\Gamma} \setminus E_{\Gamma'}$ (we informally view the edges of $\Gamma'$ as being ``contracted'').
  If $e \in E_{\Gamma} \setminus E_{\Gamma'}$ has an endpoint in $\Gamma'$, then the corresponding endpoint of $e \in E_{\Gamma/\Gamma'}$ is the new terrestrial vertex $[\Gamma']$.
  In particular, if an edge is not in $\Gamma'$ but both of its endpoints are, then this edge becomes a tadpole on the vertex $[\Gamma']$ in $\Gamma/\Gamma'$.
\end{definition}

\begin{example}
  The quotient $\Gamma / \varnothing$ is $\Gamma$ with a new isolated terrestrial vertex.
\end{example}

\begin{definition}[Quotient graph, case $m = 1$]
  \label{def:quotient-graph-1}
  For $m = 1$, the definition of the quotient graph $\Gamma / \Gamma'$ is slightly modified to deal with the linear order on terrestrial vertices.
  Suppose that $\Gamma' \subset \Gamma$ is a subgraph.
  The linear order of the terrestrial vertices of $\Gamma/\Gamma'$ is dual to the block composition of linear orders from Equation~\eqref{eq:block}.
  More concretely, $\Gamma / \Gamma'$ is a linear combination of graphs:
  \begin{itemize}[nosep]
  \item If $\Gamma'$ has no terrestrial vertices, then $\Gamma / \Gamma'$ is a sum of several terms.
    Each term is a quotient graph as in Definition~\ref{def:quotient-graph}.
    The sum ranges over all possible positions for the new vertex: if $f \in \ee_{1}^{\vee}(U)$ is the decoration of $\Gamma$, then the decoration of $\Gamma/\Gamma'$ is $(\sigma \in \Sigma_{U/\varnothing}) \mapsto f(\sigma|_{U})$.
  \item If the terrestrial vertices of $\Gamma'$ are consecutive, then $\Gamma / \Gamma'$ is defined as in Definition~\ref{def:quotient-graph}, and the linear order on the terrestrial vertices of $\Gamma / \Gamma'$ is inherited from $\Gamma$, with the new terrestrial vertex $[\Gamma']$ in the position of the vertices of $\Gamma'$.
  \item If the terrestrial vertices of $\Gamma'$ are not consecutive, then $\Gamma / \Gamma' = 0$.
  \end{itemize}
\end{definition}

We can interpret $\circ_{W,T}(\Gamma)$ as the sum of the $\Gamma/\Gamma' \otimes \Gamma'$ for all (not necessarily full) subgraphs $\Gamma'$ with set of vertices $(W,T)$.

Let us now define a zigzag of Hopf cooperad maps
\begin{equation}
  H^{*}(\VFM_{mn}) = \vscV_{mn} \gets \VGra_{mn} \to \OmPA^{*}(\VFM_{mn}).
\end{equation}
The left-pointing map is defined by $\tilde{e}_{uu} \mapsto 0$, $\tilde{e}_{uu'} \mapsto \tilde{\omega}_{uu'}$ for $u \neq u' \in U$, $e_{vv} \mapsto 0$, $e_{vv'} \mapsto \omega_{vv'}$ for $v \neq v' \in V$, and $e_{ij} \mapsto 0$ if $i \in U$ or $j \in U$.
The right-pointing map is defined using the following three ``propagators'':
\begin{itemize}[nosep]
\item Using the identification $\VFM_{mn}(2,0) = \FM_{m}(2) \cong S^{m-1}$, we can use the propagator $\varphi_{m} \coloneqq \vol_{m-1} \in \OmPA^{m-1}(\FM_{m}(2))$ of Equation~\eqref{eq:propagator1}.
\item Recall the map $\theta_{12} : \VFM_{mn}(1, 1) \to S^{n-1}$ which records the direction from the aerial point to the terrestrial point.
  The propagator $\psi_{mn}$ is the pullback of the volume form of $S^{n-1}$ along $\theta_{12}$:
  \begin{equation}
    \label{eq:psi-del}
    \psi_{mn}^{\partial} \coloneqq \theta_{12}^{*}(\vol_{n-1}) \in \OmPA^{n-1}(\VFM_{mn}(1,1)).
  \end{equation}
\item Similarly, there is another map $\vartheta_{12} : \VFM_{mn}(0,2) \to S^{n-1}$ which records the direction from the second point to the first point.
  The propagator is:
  \begin{equation}
    \psi_{mn} \coloneqq \vartheta_{12}^{*}(\vol_{n-1}) \in \OmPA^{n-1}(\VFM_{mn}(0,2)).
  \end{equation}
\end{itemize}

\begin{remark}
  By construction, these propagators are all minimal forms on $\VFM_{mn}$, because the volume form on $S^{d}$ is $\vol_{d} = \mathrm{cst} \cdot \sum_{i} (-1)^{i} x_{i} dx_{1} \wedge \dots \widehat{dx}_{i} \dots \wedge dx_{d}$.
  Hence they can be pushed forward (once) along PA bundles~\cite{HardtLambrechtsTurchinVolic2011}.
\end{remark}

\begin{remark}\label{rmk:not-surj}
  The map $\VGra_{mn} \to \vscV_{mn}$ is not surjective as its image does not contain the classes $\eta_{i}$.
  These classes are only reached once $\VGra_{mn}$ is twisted: it is the image of a graph with internal vertices (see Remark~\ref{rmk:graph-eta} and Section~\ref{sec:conn-graphs-cohom}).
\end{remark}

We may then define a morphism
\begin{equation}
  \omega' : \VGra_{mn}(U,V) \to \OmPA^{*}(\VFM_{mn}(U,V))
\end{equation}
as follows, using the convention that $u,\,u' \in U$ and $v,\,v' \in V$:
\begin{equation}
  \label{eq:def-omega}
  \begin{aligned}
    \omega'(\tilde{e}_{uu'}) & \coloneqq p_{uu'}^{*}(\varphi_{m});
    & \omega'(e_{vv'}) & \coloneqq p_{vv'}^{*}(\psi_{mn}); \\
    \omega'(e_{uu'}) & \coloneqq 0;
    & \omega'(e_{vu}) & \coloneqq p_{vu}^{*}(\psi_{mn}^{\partial}).
  \end{aligned}
\end{equation}
Here, $p_{ij}(x) = (p_{i}(x), p_{j}(x))$ (including if $i = j$).
The following lemma is immediate:

\begin{lemma}\label{lem:mor-vgra-vsc-vfm}
  This defines a zigzag $\vscV_{mn} \gets \VGra_{mn} \to \OmPA^{*}(\VFM_{mn})$.
  \qed{}
\end{lemma}

Given a graph $\Gamma \in \VGra_{mn}(U,V)$, we define $V_{\Gamma} = V^{t}_{\Gamma} \cup V^{a}_{\Gamma} = U \cup V$ to be its set of vertices, partitioned into terrestrial and aerial ones.
Similarly, $E_{\Gamma} = E^{f}_{\Gamma} \cup E^{d}_{\Gamma}$ is its set of edges, split into full edges and dashed edges.
The graph $\Gamma$ induces:
\begin{equation}
  \label{eq:20}
  \Phi_{\Gamma} : \VFM_{mn}(U,V) \to (S^{m-1})^{E^{d}_{\Gamma}} \times (S^{n-1})^{E^{f}_{\Gamma}},
\end{equation}
obtained using the maps $\theta_{ij}$ from Section~\ref{sec.defin-comp-with}.
We also define
\begin{equation}
  \vol_{\Gamma} \in \OmPA^{\deg \Gamma}((S^{m-1})^{E^{d}_{\Gamma}} \times (S^{n-1})^{E^{f}_{\Gamma}})
\end{equation}
to be the product of the volume forms. Then by definition, $\omega'(\Gamma) = \Phi_{\Gamma}^{*}(\vol_{\Gamma})$.

Given a graph $\Gamma \in \VGra_{mn}(U,V)$ with $\# U + 2 \#V \ge 2$, we may define its \emph{coefficient} $c(\Gamma)$ by
\begin{equation}
  \label{eq:def-c}
  c(\Gamma) \coloneqq \int_{\VFM_{mn}(U,V)} \omega'(\Gamma).
\end{equation}
If $\# U + 2 \# V \le 1$ then we just set $c(\Gamma) = 0$ (this is due to the special case in the definition of $\BF''$, see Proposition~\ref{prop:decomp-boundary}).

\begin{remark}\label{rmk:konts-analogous}
  In the case $(n,m) = (2,1)$, these are analogous to the coefficients in Kontsevich's universal formality morphism $T_{\mathrm{poly}} \to D_{\mathrm{poly}}$ from~\cite{Kontsevich2003}.
  For $(n,m) = (3,1)$, these are related to local versions of Kontsevich integrals~\cite{Kontsevich1994}, which are conjectured to equal Bott--Taubes integrals~\cite{Poirier2002}.
\end{remark}

\paragraph{Twisting}

In this section, we are going to twist the cooperad $\VGra_{mn}$ using the theory developed in Appendix~\ref{sec:twist-relat-coop}.
The general motivation of twisting is to encode Maurer--Cartan equations.
Here, since we are going to twist with respect to the coefficients $c$, which come from the stratification of $\VFM_{mn}$, we get a differential which matches the way the boundary facets of $\VFM_{mn}$ interact.

Recall the relative $\hoLie_{n}$-operad $\hoLieto_{mn}$ from Appendix~\ref{sec:twist-relat-coop}.
Recall also the coefficient $\mu : \hoLie_{n} \to \Gra_{n}^{\vee}$ from Equation~\eqref{eq:mu}.

\begin{proposition}\label{prop:c-mor-holie}
  Together with $\mu$, the collection of coefficients $c$ defines a morphism of colored operads $(\hoLie_{n}, \hoLieto_{mn}) \to (\Gra^{\vee}_{n}, \VGra^{\vee}_{mn})$.
  The generating bracket $\lambda_{U,V} \in \hoLieto_{mn}(U,V)$ is sent to the element $c(\lambda_{U,V}) \in \VGra_{mn}^{\vee}(U,V)$ given by $\Gamma \mapsto c(\Gamma)$, which we can write as the possibly infinite sum $c = \sum_{\Gamma} c(\Gamma) \Gamma$.
\end{proposition}
\begin{proof}
  The proof is similar to Kontsevich's proof that $\mu$ defines a morphism $\hoLie_{n} \to \Gra_{n}^{\vee}$.
  Let $C_{*}^{SA}$ be the functor of semi-algebraic chains~\cite[Section~3]{HardtLambrechtsTurchinVolic2011}.
  It is lax-monoidal~\cite[Proposition~3.8]{HardtLambrechtsTurchinVolic2011} so $C^{SA}_{*}(\VFM_{mn})$ is a dg-operad.
  We can dualize Lemma~\ref{lem:mor-vgra-vsc-vfm} to get a morphism $I : C_{*}^{SA}(\VFM_{mn}) \to \OmPA^{*}(\VFM_{mn})^{\vee} \to \VGra_{mn}^{\vee}, \sigma \mapsto \langle \omega'(-), \sigma \rangle$.
  The claimed morphism $c : \hoLieto_{mn} \to \VGra_{mn}^{\vee}$ is the composition of $I$ with:
  \[ \llbracket \VFM_{mn} \rrbracket : \hoLieto_{mn} \to C_{*}^{SA}(\VFM_{mn}) \]
  which maps the generating bracket $\lambda_{U,V} \in \hoLieto_{mn}(U,V)$ to the fundamental chain $\llbracket \VFM_{mn}(U,V) \rrbracket$.
  So we just need to check that $\llbracket \VFM_{mn} \rrbracket$ is a morphism.

  Since $\hoLieto_{mn}$ is quasi-free, we just need to check that our prescription on generators is compatible with the differential.
  Recall~\cite[Theorem~3.5]{HardtLambrechtsTurchinVolic2011} that for a compact oriented SA manifold $M$, $d \llbracket M \rrbracket = \llbracket \partial M \rrbracket$.
  If we use the description of $\partial \VFM_{mn}(U,V)$ from Proposition~\ref{prop:decomp-boundary}, we find that the differential in $C_{*}^{SA}(\VFM_{mn})$ matches the cobar differential.
  The faces of type $\BF'$ correspond to cocompositions $\Comto_{mn}^{\vee}(U,V) \to \Comto_{mn}^{\vee}(U,V/T) \otimes \Com_{n}^{\vee}(T)$, and the faces of type $\BF''$ correspond to cocompositions $\Comto_{mn}^{\vee}(U,V) \to \Comto_{mn}^{\vee}(U/W, V \setminus T) \otimes \Comto_{mn}^{\vee}(W,T)$.
  The ``missing'' faces (when $\# T < 2$, $(W,T) = (U,V)$, or $2 \#T + \#W < 2$) come from the fact that the cobar construction uses the coaugmentation quotient of the cooperad and from $\Com_{n}^{\vee}(\varnothing) = \Comto_{mn}^{\vee}(\varnothing,\varnothing) = 0$.
\end{proof}

\begin{definition}
  \label{def:tw-sgra}
  The twisted graph cooperad $\Tw\VGra_{mn}$ is the relative $(\Tw\Gra_{n})$-cooperad obtained by twisting $\VGra_{mn}$ with respect to $\mu$ and $c$.
\end{definition}

This twisted cooperad has a graphical description.
The CDGA $\Tw\VGra_{mn}(U,V)$ is spanned by graphs with four types of vertices:
external terrestrial vertices, in bijection with $U$ (drawn as semicircles);
external aerial vertices, in bijection with $V$ (drawn as circles);
internal terrestrial vertices, indistinguishable among themselves, of degree $-m$ (drawn as black semicircles);
internal aerial vertices, indistinguishable among themselves, of degree $-n$ (drawn as black circles).

There are two kinds of edges: full edges of degree $n-1$, and dashed edges of degree $m-1$.
The product glues graphs along external vertices.
If $m = 1$, then there are no dashed edges, and the whole set of terrestrial vertices is ordered (if we view $\ee_{1}^{\vee}(U \sqcup I)$ as spanned by the dual basis of $\ee_{1}(U \sqcup I)$).
In this basis, the product of two graphs $\Gamma \cdot \Gamma'$ vanishes if the external vertices of $\Gamma$ and $\Gamma'$ are ordered differently.
Otherwise $\Gamma \cdot \Gamma'$ is obtained by gluing the two graphs along external vertices, and summing all possible ways of ordering the union (along external vertices) of the terrestrial vertices of $\Gamma$ and $\Gamma'$ in a way that the new order restricts to the given orders in $\Gamma$ and $\Gamma'$.

\begin{mydescription}
  \label{descr:diff}
  The differential is given by the twisting procedure (Appendix~\ref{sec:twist-relat-coop}).
  Recall the subgraph contractions defined in Definitions~\ref{def:quotient-graph},~\ref{def:quotient-graph-1}.
  Given a graph $\Gamma$, its differential $d\Gamma$ is given as a sum (with signs in Section~\ref{sec:twist-relat-coop}, the terms corresponding respectively to $(-) \cdot \mu$, $(-) \cdot (c-c_{1})$, and $c_{1} \cdot (-)$):
  \begin{enumerate}[nosep]
  \item contractions of full edges between an aerial internal vertex and an aerial vertex of any kind, including edges connected to a univalent internal vertex (this uses the simple description of $\mu$ in Equation~\eqref{eq:mu});
  \item contractions of subgraphs $\Gamma'$ containing at most one external (necessarily terrestrial) vertex, with result $c(\Gamma') \cdot \Gamma/\Gamma'$;
  \item forgetting of all vertices outside a subgraph $\Gamma''$ which contains all the external vertices, with result $c(\Gamma/\Gamma'') \cdot \Gamma''$.
  \end{enumerate}
  Note that the differential is not purely combinatorial: it depends on $c$, which is defined by integrals.
\end{mydescription}

Since $\VGra_{mn}$ is a Hopf cooperad satisfying $\VGra_{mn}(\varnothing, \varnothing) = \Bbbk$, and since $c$ vanishes on disconnected graphs (see Lemma~\ref{lem:c-discon}) and graphs with a terrestrial cut point (by a simple degree argument: $\dim \VFM_{mn}(i+i'+1,j+j') > \dim \VFM_{mn}(i+1,j) \times \VFM_{mn}(i'+1,j')$), we find again that $\Tw\VGra_{mn}$ inherits a Hopf cooperad structure, similarly to the uncolored case.
For $m = 1$, this is also true: the differential preserves (up to quotient) the order of the terrestrial vertices, and if two consecutive terrestrial vertices in $\Gamma \cdot \Gamma'$ come respectively from $\Gamma$ and $\Gamma'$ (which could potentially produce an unwanted term in $d(\Gamma \cdot \Gamma')$ due to the exception in Lemma~\ref{lem:c-discon}), then they also appear with the opposite order in the product and the two terms in the differential cancel out.

\begin{lemma}
  \label{lem:def-omega}
  For $n-2 \ge m \ge 1$, the morphism $\omega'$ extends to a morphism of Hopf cooperads
  $\omega : \Tw\VGra_{mn} \to \OmPA^{*}(\VFM_{mn})$
  by setting, for $\Gamma \in \VGra_{mn}(U \sqcup I, V \sqcup J) \subset \Tw\VGra_{mn}(U,V)$ and $\#U + 2 \#V \geq 2$:
  \[ \omega(\Gamma) \coloneqq (p_{U,V})_{*}(\omega'(\Gamma)) = \int_{\VFM_{mn}(U \sqcup I, V \sqcup J) \to \VFM_{mn}(U,V)} \omega'(\Gamma). \]
  If $\#U + 2\#V \leq 1$, then $\omega(\Gamma) = 1$ if $\Gamma$ has no edges and no internal vertices, and $\omega(\Gamma) = 0$ otherwise.
\end{lemma}
\begin{proof}
  We use the double pushforward formula~\cite[Proposition~8.13]{HardtLambrechtsTurchinVolic2011}, Stokes' formula~\cite[Proposition~8.12]{HardtLambrechtsTurchinVolic2011}  and the decomposition of the fiberwise boundary $\VFM_{mn}(U,V)$ from Section~\ref{sec.comp-its-bound} (compare with~\cite[Section~9.4]{LambrechtsVolic2014}).

  We deal separately with the case $\#U+2\#V \leq 1$ (cf.~\cite[Section~9.1]{LambrechtsVolic2014}).
  Indeed, the fiber of $p_{U,V} : \VFM_{mn}(U \sqcup I, V \sqcup J) \to \VFM_{mn}(U,V)$ has dimension $m \#I + n \#J$ in general, but $\dim \VFM_{mn}(U,V) = 0$ for $\#U + 2\#V \le 1$, so the dimension of the fiber of $p_{U,V}$ is either $m \#I + n\#J - 1$ or $m \#I + n \#J - m - 1$ and $\omega'$ would not preserve degrees.
  In these cases $\VFM_{mn}(U,V)$ is merely a singleton, so we just have to check that any cocycle is mapped to zero.
  This is clear if $U = V = \varnothing$, as the differential of a nonempty graph is nonempty.
  For $\Gamma \in \Tw \VGra_{mn}(1,0)$ with at least one edge, the graph with no edges appears exactly twice with opposite signs in $d\Gamma$ (see Description~\ref{descr:diff}): for $\Gamma' = \Gamma$ in $-\Gamma \cdot c_{1} = - c(\Gamma') \Gamma/\Gamma' + \dots$, and for $\Gamma''$ containing only the external vertex in $c_{1} \cdot \Gamma = c(\Gamma/\Gamma'') \Gamma'' + \dots$.
\end{proof}

\paragraph{Graph complex $\fVGC_{mn}$}

We now mimic Definition~\ref{def:fgcn}:

\begin{definition}\label{def:fvgcmn}
  For $n-2 \ge m \ge 1$, we define $\fVGC_{mn} \coloneqq \Tw\VGra_{mn}(\varnothing,\varnothing)[-m].$

  This shifted CDGA is free and generated by its submodule of connected graphs.
  We still denote by $c \in \fVGC^{\vee}_{mn}$ the restriction of the morphism $c$ from~\eqref{eq:def-c}.
\end{definition}

\begin{lemma}
  \label{lem:c-discon}
  Let $n -2 \ge m \ge 1$.
  Given a disconnected graph $\gamma \in \fVGC_{mn}$, the coefficient $c(\gamma)$ vanishes, unless $m=1$ and $\gamma = \bigl( \tikz{ \node[intt] (i) {}; \node[intt, right=.5cm of i] () {};} \bigr)$.
\end{lemma}
\begin{proof}
  Let us first assume that $\gamma$ has no isolated terrestrial vertices.
  Then $\gamma$ factors as a product $\gamma = \gamma_{1} \cdot \gamma_{2}$, with corresponding sets of vertices $(I_{1}, J_{1})$ and $(I_{2}, J_{2})$ both satisfying $\#I_{\bullet} + 2 \#J_{\bullet} \geq 2$ for $\bullet \in \{1,2\}$.
  We need $\deg \vol_{\gamma} = \dim \VFM_{mn}(I,J)$, otherwise $c(\gamma) = \int_{\VFM_{mn}(I,J)} \Phi^{*}_{\gamma}(\vol_{\gamma})$ obviously vanishes.
  But thanks to the hypothesis $\Phi_{\gamma}$ factors through $\VFM_{mn}(I_{1}, J_{1}) \times \VFM_{mn}(I_{2}, J_{2})$, which is of strictly lower dimension.
  Hence
  $\Phi^{*}_{\gamma}(\vol_{\gamma}) = 0 \implies c(\gamma) = 0$.

  Now let $\gamma$ have an isolated terrestrial vertex $i \in I$.
  Then $\Phi_{\gamma}$ factors through $\VFM_{mn}(I \setminus \{i\}, J)$.
  If $(\#I - 1) + 2 \#J \geq 2$, then the codimension with $\VFM_{mn}(I,J)$ is $m > 0$, and so similarly $c(\gamma) = 0$.
  However, if $\gamma$ is the graph with two terrestrial vertices and no aerial ones (i.e.\ $(\#I, \#J) = (2,0)$), then the codimension is $m-1$, because $\dim \VFM_{mn}(1,0) = 0$.
  Thus we only get $c(\gamma) = 0$ if $m > 1$.
\end{proof}

\begin{remark}
  \label{rmk:m-egal-un}
  In the case $m = 1$, we have that $\VFM_{1n}(2,0) \cong \FM_{1}(2) \cong S^{0}$.
  We find that for $\gamma = \bigl( \tikz{ \node[intt] (i) {}; \node[intt, right=.5cm of i] () {};} \bigr) \in \fVGC_{1n}$, we have $c(\gamma) = 1$.
\end{remark}

\begin{definition}\label{def:vgcmn}
  For $n - 2 \ge m \ge 2$, the connected complementarily constrained graph complex $\VGC_{mn}$ as the quotient of $\fVGC_{mn}$ by disconnected graphs.
\end{definition}

The case $m = 1$ is special.
Mimicking~\cite[Section~4]{Willwacher2015a}, we define an alternate basis of $\Tw \Gra_{1n}(U,V)$.
Recall that $\Gra_{1n}(U, \varnothing)$ contains $\Ass^{\vee}(U)$, where $\Ass$ encodes associative algebras.
The symmetric sequence $\Ass$ is isomorphic to $\Pois = \Com \circ \Lie$~\cite[Section~13.3]{LodayVallette2012}.
We can thus decree that terrestrial internal vertices are in the same ``Lie component'' if they are connected by brackets.

\begin{definition}[{Cf.~\cite[Section~4]{Willwacher2015a}}]
  A Lie-decorated graph is a graph $\Gamma$ defined similarly as an element of $\Tw \Gra_{1n}(U,V)$ (Definition~\ref{def:tw-sgra}), but instead of giving a function on linear orders of terrestrial vertices, we give a function on commutative products of Lie words of terrestrial vertices.
  More precisely, we use the dual of the PBW isomorphism to get an isomorphism and we obtain a new basis of $\Tw \VGra_{1n}(U,V)$ through the twisting procedure:
  \[ \VGra_{1n}(U,V) \cong S(e_{ij})_{i,j \in U \sqcup V} \otimes \Pois^{\vee}(U) \bigm/ (e_{ji} = (-1)^{n} e_{ij}). \]
\end{definition}

One possible way to describe a Lie-decorated graph $\Gamma \in \Tw \VGra_{1n}(U,V)$ (with internal vertices $I$) is to take a usual basis of $\Pois(U \sqcup I) = (\Com \circ \Lie)(U \sqcup I)$ and consider its dual basis in $\Pois^{\vee}(U \sqcup I)$, i.e.\ functions on products of Lie words that vanish on all such elements except one.
See Figure~\ref{fig:lie-dec-graph} for an example.

\begin{figure}[htbp]
  \centering
  \begin{tikzpicture}
    \node[exta] (a1) {$v_{1}$};
    \node[inta, right = 1cm of a1] (a2) {};
    \node[exta, above right = .5cm of a1] (a3) {$v_{2}$};
    \node[extt, below = .4cm of a1] (t1) {$u_{1}$};
    \node[intt, right = 1cm of t1] (t2) {};
    \node[extt, right = 1cm of t2] (t3) {$u_{2}$};
    \node[intt, right = 1cm of t3] (t4) {};
    \node[left = .2cm of t1] {$\lbrack$};
    \node[right = .3cm of t1] {$,$};
    \node[right = .2cm of t2] {$\rbrack$};
    \draw
    (a1) edge (a2) edge (a3) edge (t1)
    (a2) edge (a3) edge (t2) edge[bend right] (t2)
    edge[bend right] (a3)
    (t2) edge[bend right] (t3);
  \end{tikzpicture}
  \caption{Lie-decorated graph; the depicted product of Lie words $[x_{1}, x_{2}] \wedge x_{3} \wedge x_{4}$ is actually the element in the dual basis of $\Pois^{\vee}(4)$}
  \label{fig:lie-dec-graph}
\end{figure}

It is actually easier to describe the differential in of a Lie-decorated graph using coLie words.
Recall that a coLie word $U$ is an element in the cofree coalgebra $T^{\ge1}(U)^{\vee}$ which vanishes on the image of $S^{\ge 2}(\Lie(U))$ under the PBW isomorphism.
These words are the image of the first Eulerian idempotent $e^{(1)} = \log_{\star}(\id)$, where $\star$ is the convolution product of the bialgebra $T(U)^{\vee}$ (see e.g.~\cite[Appendix~4.2]{Fresse2006}).
CoLie words themselves are described by graphs (see~\cite{SinhaWalter2011}), one can also thus think of Lie-decorated graphs as ``graph-decorated graphs''.
With this basis, the differential merges coLie words (with a description similar to the one of $\Tw\VGra_{mn}$ above).

\begin{lemma}\label{lem:iso-pbw}
  The vector space spanned by Lie-decorated graphs with external vertices $(U,V)$ is isomorphic to $\Tw \Gra_{1n}(U,V)$.
\end{lemma}
\begin{proof}
  We use the dual of the PBW isomorphism $\Pois^{\vee}(U \sqcup I) \cong \Ass^{\vee}(U \sqcup I)$ to identify linear combinations of Lie-decorated graphs with elements of $\Tw \Gra_{1n}(U,V)$.
  For the signs, terrestrial vertices have degree $-1$.
\end{proof}

\begin{definition}
  \label{def:lie-connected}
  Two vertices in a Lie-decorated graph are in the same Lie component if they are connected by full edges or if they are both terrestrial and in the same Lie word.
  A Lie-decorated graph is Lie-connected if all its vertices are in the same Lie component, and Lie-disconnected otherwise.
\end{definition}

\begin{lemma}
  The coefficient $c$ vanishes on the image of a Lie-disconnected graph under the isomorphism of Lemma~\ref{lem:iso-pbw}.
\end{lemma}
\begin{proof}
  This follows from Lemma~\ref{lem:c-discon}.
\end{proof}

\begin{definition}
  For $n \ge 3$, the Lie-connected complementarily constrained graph complex $\VGC_{1n}$ as the quotient of $\fVGC_{1n}$ by the image of the Lie-disconnected graph under the isomorphism of Lemma~\ref{lem:iso-pbw}.
\end{definition}

\begin{lemma}
  For $n - 2 \ge m \ge 1$, the coefficient $c : \fVGC_{mn} \to \R$ factors through the quotient defining $\VGC_{mn}$.
  Abusing notation, we denote by $c$ the induced map $\VGC_{mn} \to \R$.
\end{lemma}
\begin{proof}
  This also follows from Lemma~\ref{lem:c-discon}, in both cases $m = 1$ and $m \ge 2$.
\end{proof}

\begin{remark}\label{rmk:c-maurer-cartan}
  The dual $\VGC_{mn}^{\vee}$ is a (pre-)Lie algebra and the Lie algebra $\GC_{n}^{\vee}$ acts on it by derivations, in both cases using insertion of graphs (respectively at terrestrial and aerial vertices).
  The differential of $\VGC_{mn}^{\vee}$ is given by $[\mu + c,-]$ where $c$ represents the coefficients from Equation~\eqref{eq:def-c}.
  The elements $\mu$ and $c$ satisfy the Maurer--Cartan equation, i.e.\ $[\mu,\mu] = 0$ and $[\mu,c] + \frac{1}{2} [c,c] = 0$.
\end{remark}

\paragraph{Reduction}

We now mod out internal components.

\begin{definition}
  \label{def:vgraphs}
  For $n - 2 \ge m \geq 2$, we define the graph cooperad $\VGraphs_{mn}$ to be the relative $\Graphs_{n}$-cooperad given by quotient of $\Tw\VGra_{mn}$ by the Hopf cooperadic ideal of graphs with internal components.
  For $n - 2 \ge m = 1$, we define $\VGraphs_{1n}$ to be the quotient of $\Tw\VGra_{1n}$ by the Hopf cooperadic ideal of Lie-disconnected graphs (Definition~\ref{def:lie-connected}).
\end{definition}

\begin{remark}
  Compare this with Willwacher's model $\SGraphs_{n}$ for the Swiss-Cheese operad~\cite{Willwacher2015a}.
  In $\SGraphs_{n}$, there are no dashed edges, the full edges are oriented, and their source is always aerial.
\end{remark}

\begin{proposition}\label{prop:omega-factors}
  The morphism $\omega$ factors through the quotient and defines a morphism $\omega : \VGraphs_{mn} \to \OmPA^{*}(\VFM_{mn})$.
\end{proposition}

\begin{proof}
  The proof is identical to the proof of~\cite[Lemma~9.3.7]{LambrechtsVolic2014}.
  Briefly,
  let $\gamma \in \Tw\VGra_{mn}(U,V)$ be a graph whose edges all are between internal vertices.
  Let $(I,J)$ be the sets of internal vertices of $\gamma$ and assume $\#U+2\#V \ge 2$ otherwise the claim is obvious (see Lemma~\ref{lem:def-omega}).
  We can fill the diagram, where $\rho$ is the product of two projections:
  \begin{equation*}
    \begin{tikzcd}[column sep = tiny, row sep = small]
      \VFM_{mn}(U \sqcup I, V \sqcup J) \ar[rr, "\Phi_{\gamma}"] \ar[dr, "\rho"] \ar[d, "p_{U,V}"]
      && (S^{m-1})^{E^{t}_{\Gamma}} \times (S^{n-1})^{E^{a}_{\Gamma}}
      \\
      \VFM_{mn}(U,V) & \VFM_{mn}(U,V) \times \VFM_{mn}(I,J) \ar[ur, dashed, "\exists \Phi'_{\gamma}" swap] \ar[l, "\operatorname{pr}_{1}"]
    \end{tikzcd}
    .
  \end{equation*}
  We therefore find that $\omega(\Gamma) = (p_{U,V})_{*}(\omega'(\Gamma)) = (p_{U,V})_{*}(\rho^{*}(\Phi_{\gamma}^{'*}\vol_{\Gamma}))$.
  The dimension of the fibers of $p_{U,V}$ and $\operatorname{pr}_{1}$ are respectively $m\#I+n\#J > m \#I+n\#J-m$.
  Therefore $\omega(\Gamma) = 0$ by~\cite[Proposition~5.1.2]{HardtLambrechtsTurchinVolic2011}.
  For $m = 1$, we note that
  a Lie-disconnected graph is in particular disconnected.
\end{proof}

\subsection{Vanishing lemmas and \texorpdfstring{$\vgraphs_{mn}$}{vgraphs\_mn}}
\label{sec:vanishing-lemmas}

We now prove some vanishing lemmas about $\omega$ and $c$, which allows us to define the further reduced cooperad $\vgraphs_{mn}$.
This will be useful to prove that the quotient map $\VGra_{mn} \to \vscV_{mn}$ extends to $\vgraphs_{mn}$ in Section~\ref{sec:proof-formality}.

\begin{lemma}
  \label{lem:vanish-loop-double}
  The morphism $\omega'$ vanishes on graphs with loops or double edges.
\end{lemma}
\begin{proof}
  This follows from simple dimension arguments, cf.~\cite[Section~9.3]{LambrechtsVolic2014}.
\end{proof}

\begin{lemma}[{Cf.~\cite[Lemma~9.3.9]{LambrechtsVolic2014}}]
  \label{lem:vanish-unival}
  The morphism $\omega$ vanishes on graphs containing univalent aerial internal vertices, or univalent terrestrial internal  vertices connected to another terrestrial vertex.
\end{lemma}
\begin{proof}
  The lemma is trivial if $\#U + 2 \#V \le 1$ (see Lemma~\ref{lem:def-omega}), so let us assume that we are not in this case.

  Let us first deal with univalent aerial internal vertices connected to another aerial vertex.
  The graph $\Gamma_{\text{uni}} = \tikz[scale=.5, baseline=(a.base)]{\node[exta] (a) {1}; \node[inta, right=1cm of a] (i) {}; \draw (a) -- (i);}$ is of negative degree, thus $\omega(\Gamma_{\text{uni}}) = 0$.
  For a general graph $\Gamma$ containing $\Gamma_{\mathrm{uni}}$, the argument is the same as~\cite[Corollary~44]{Idrissi2018b} or~\cite[Lemma~9.3.9]{LambrechtsVolic2014}.
  The map $\Phi_{\Gamma}$ of~\eqref{eq:20} factors through $\VFM_{mn}(\varnothing, \{v,j\}) \times \VFM_{mn}(U \sqcup I, V \sqcup J \setminus \{j\})$, and so does the canonical projection.
  The two factorizations make the obvious diagram (like in Proposition~\ref{prop:omega-factors}) commute.
  Using~\cite[Propositions~8.10, 8.13]{HardtLambrechtsTurchinVolic2011} and $\omega(\Gamma_{\text{uni}}) = 0$, we get $\omega(\Gamma) = 0$.

  Let us now assume that $\Gamma$ is a graph with a univalent internal vertex connected $i$ to a terrestrial vertex $u$.
  We can assume that $i$ is aerial or that the incident edge is dashed, as otherwise $\omega(\Gamma) = 0$ by definition~\eqref{eq:def-omega}.
  Let $(U \sqcup I, V \sqcup J)$ be the sets of vertices of $\Gamma$.
  Then both $\Phi_{\Gamma}$ and $p_{U,V}$ factor through either $X = \VFM_{mn}(\{u\}, \{i\}) \times \VFM_{mn}(U \sqcup I, V \sqcup J \setminus \{i\})$ (if $i$ is aerial) or $X = \VFM_{mn}(\{u,i\}, \varnothing) \times \VFM_{mn}(U \sqcup I \setminus \{i\}, V \sqcup J)$ if $i$ is terrestrial, and the factorizations make the obvious diagram commute.
  The dimension of the fiber of $p_{U,V}$ is $m\#I+n\#J$, while the dimension of the fiber of the projection $X \to \VFM_{mn}(U,V)$ is $m\#I+n\#J-1$ in both cases.
  Hence, by~\cite[Proposition~8.14]{HardtLambrechtsTurchinVolic2011}, $\omega(\Gamma) = 0$.
\end{proof}

\begin{remark}
  \label{rmk:graph-eta}
  The argument fails for graphs containing a univalent internal terrestrial vertex connected to an aerial vertex.
  If we consider the following graph, then we find that $\omega(\Gamma) \in \OmPA^{n-m-1}(\VFM_{mn}(0,1))$ represents $\eta$ from Section~\ref{sec.comp-cohom}:
  \begin{equation}\label{eq:graph-eta}
    \Gamma =
    \begin{tikzpicture}[scale=.3, baseline=(v.base)]
      \node[exta] (v) {1};
      \node[intt, right = .5cm of v] (i) {};
      \draw (v) -- (i);
    \end{tikzpicture}
    \in \vgraphs_{mn}(0,1),
  \end{equation}
  Indeed, by the definition of $\psi^{\partial}_{mn}$ in~\eqref{eq:psi-del} and by~\cite[Proposition~8.10]{HardtLambrechtsTurchinVolic2011}, we find that $\int_{\VFM_{mn}(1,1)} \psi^{\partial}_{mn} = \int_{S^{n-1}} \vol_{n-1} = 1$.
  Therefore $\int_{\VFM_{mn}(0,1)} \omega(\Gamma) = 1$ by~\cite[Proposition~8.13]{HardtLambrechtsTurchinVolic2011}.
  This also implies that if $\gamma$ is obtained from the $\Gamma$ above by making the external vertex internal, then we get $c(\gamma) = 1$.
\end{remark}

\begin{lemma}
  \label{lem:c-vanish-unival}
  The coefficient $c$ vanishes on graphs with more than two vertices and which either contain a univalent aerial vertex, or which contain a univalent terrestrial vertex connected to another terrestrial vertex.
  It also vanishes on the graph with exactly two vertices, both aerial, and a (full) edge between the two.
\end{lemma}
\begin{proof}
  We can reuse the proof of Lemma~\ref{lem:vanish-unival} almost verbatim.
  There is one caveat: the graph $\gamma$ with the univalent vertex removed must still satisfy the hypothesis $\#I+2\#J \geq 2$.
  Otherwise, a degree shift occurs, to deal with the fact that $\dim \VFM_{mn}(I,J) = 0$ and not $-1$ or $-m-1$ as the general formula would give.
  This is the case unless $\gamma$ has exactly two vertices with at most one aerial.
  If $\gamma$ has exactly two aerial vertices and one edge, then $c(\gamma)$ vanishes for degree reasons ($\dim \VFM_{mn}(0,2) = 2n-m-1 > n-1$).
\end{proof}

\begin{remark}
  \label{rmk:leading-c}
  The restriction about the number of vertices is necessary.
  Indeed, for \(
  \gamma =
    \begin{tikzpicture}[scale=.3, baseline = (i.south)]
      \node[intt] (i) {};
      \node[intt, right = 1cm of i] (j) {};
      \draw[dashed] (i) -- (j);
    \end{tikzpicture}
    \in \VGC_{mn}.
  \) we find $c(\gamma) = \int_{\VFM_{mn}(2,0)} \varphi_{m} = \int_{S^{m-1}} \vol_{S^{m-1}} = 1$.
\end{remark}

\begin{lemma}
  \label{lem:vanish-bival}
  The morphism $\omega$ and the coefficient $c$ vanish on graphs with bivalent internal terrestrial vertices connected to two terrestrial vertices.
\end{lemma}
\begin{proof}
  Let us first consider the case of the graph
  \begin{equation}
    \Gamma_{\mathrm{biv}} =
    \begin{tikzpicture}[baseline=(a.base)]
      \node[intt] (i) {};
      \node[extt, right=1cm of i] (a) {v};
      \node[extt, left=1cm of i] (b) {u};
      \draw[dashed] (a) -- (i) -- (b);
    \end{tikzpicture}
    .
  \end{equation}
  Using $\VFM_{mn}(U,\varnothing) \cong \FM_{m}(U)$ and the fact that the terrestrial propagator is the same as the one used in the definition of the map $\Graphs_{m} \to \OmPA^{*}(\FM_{m})$, we deduce that $\omega(\Gamma_{\mathrm{biv}}) = 0$ from~\cite[Lemma~9.3.9]{LambrechtsVolic2014} (see also~\cite[Lemma~2.1]{Kontsevich1994} for the origin of this result: there is an involution on $\FM_{m}(U)$ which leaves $\omega'(\Gamma_{\mathrm{biv}})$ invariant but reverses the orientation, thus $\omega(\Gamma_{\mathrm{biv}}) = - \omega(\Gamma_{\mathrm{biv}})$).
  If $\Gamma$ contains $\Gamma_{\mathrm{biv}}$ as a subgraph, then we use the same reasoning as Lemma~\ref{lem:vanish-unival} to show that $\omega(\Gamma) = 0$.
  Finally, for the other cases, simply note that $\omega$ and $c$ vanish by definitions on graphs with full edges between terrestrial vertices.
\end{proof}

\begin{definition}\label{def:iuv}
  Let $I(U,V) \subset \VGraphs_{mn}(U,V)$ be the vector space spanned by graphs containing loops, double edges (of any type), univalent aerial internal vertices, univalent internal terrestrial vertices connected to another terrestrial vertex, or bivalent internal terrestrial vertices connected to two terrestrial vertices by dashed edges.
  For $m = 1$, we take graphs containing loops, double edges, and univalent aerial internal vertices.
\end{definition}

\begin{lemma}
    \label{lem:diff-ideal}
  The subspace $I(U,V)$ define a CDGA ideal in $\VGraphs_{mn}(U,V)$, and the subcollection $I \subset \VGraphs_{mn}$ define a cooperadic coideal.
\end{lemma}

\begin{proof}
  It is clear that $I(U,V)$ is an algebra ideal.
  Let us show that it is a differential ideal.
  The fact that $I$ defines a cooperadic coideal can also be checked easily case-by-case.
  We deal with $m \ge 2$, and the case $m = 1$ is also checked by a similar case-by-case argument.
  Let $\Gamma \in I(U,V)$ be a graph.
  Let us show that $d\Gamma \in I(U,V)$.
  The three summands in Description~\ref{descr:diff} are called $d_{1} = (-) \cdot (\mu - \mu_{1})$, $d_{2} = (-) \cdot (c - c_{1})$, and $d_{3} = (c_{1}) \cdot (-)$, where the coefficients $c$ are defined analytically by integrals.
  \begin{itemize}[nosep]
  \item If $\Gamma$ contains a loop, then all the summands in $d\Gamma$ contain a loop, as both $\mu$ and $c$ vanish on loops.
  \item Suppose $\Gamma$ contains a univalent aerial internal vertex $i$, with incident edge $e$.
    In $d_{2} \Gamma = \sum_{\Gamma'} c(\Gamma') \Gamma/\Gamma'$, if $\Gamma'$ contains $e$ then $c(\Gamma')$ vanishes by Lemma~\ref{lem:c-vanish-unival}, and otherwise $\Gamma/\Gamma'$ contains either a loop (if $i \in \Gamma'$) or a univalent aerial internal vertex (if $i \not\in \Gamma'$).
    Similarly, all terms in $d_{3}\Gamma = \sum_{\Gamma''} c(\Gamma/\Gamma'') \Gamma''$ either vanish, contain a loop, or contain a univalent aerial internal vertex.
    The only problem in $d_{3} \Gamma$ is for $\Gamma'' = \Gamma \setminus e$: then $c(\Gamma/\Gamma'') = 1$ (see Equation~\eqref{eq:c-zero}) and $\Gamma''$ does not have a univalent vertex anymore.
    But this term cancels with the contraction of $e$ in $d_{1}\Gamma$, and all the other summands of $d_{1}\Gamma$ contain a univalent aerial internal vertex.
  \item If $\Gamma$ contains a univalent terrestrial vertex connected to another terrestrial vertex, then Lemma~\ref{lem:c-vanish-unival} and an argument similar to the previous one show that all summands of $d\Gamma$ either vanish, contain a loop, or contain a univalent terrestrial vertex connected to another terrestrial vertex.
  \item Suppose that $\Gamma$ contains a bivalent internal terrestrial vertices connected to two terrestrial vertices by dashed edges.
    Let $i$ be the internal terrestrial vertex, and $e$, $e'$ its incident edges.
    All terms in $d_{1} \Gamma$ contain a similar subgraph.
    In $d_{2}\Gamma = \sum_{\Gamma'} c(\Gamma') \Gamma/\Gamma'$:
    \begin{itemize}[nosep]
    \item If $i \not\in \Gamma'$, then $\Gamma/\Gamma'$ still contains a bivalent internal terrestrial vertices connected to two terrestrial vertices.
    \item If $i$ is in $\Gamma'$ but $e,e'$ are not, then $c(\Gamma') = 0$ by Lemma~\ref{lem:c-vanish-unival}.
    \item The term with $\Gamma' = e$ is cancelled with the terms with $\Gamma' = e'$.
    \item If $e \in \Gamma'$ but $e' \not\in \Gamma'$ and $\Gamma' \neq e$, then $c(\Gamma') = 0$ by Lemma~\ref{lem:c-vanish-unival}.
      The case $e' \in \Gamma'$ and $e \not\in \Gamma'$ is symmetric.
    \item Finally if both $e,e' \in \Gamma'$, then $c(\Gamma') = 0$ by Lemma~\ref{lem:vanish-bival}.
    \end{itemize}
    The fact that $d_{3}\Gamma \in I(U,V)$ follows by similar arguments.
    \qedhere
  \end{itemize}
\end{proof}

\begin{definition}\label{def:very-small-graph}
  The reduced graph cooperad $\vgraphs_{mn}$ is the relative $\graphs_{n}$-cooperad given in each arity by the quotient of $\VGraphs_{mn}(U,V)$ by $I(U,V)$.
\end{definition}

\begin{proposition}
  \label{prop:omega-reduced}
  The CDGA $\vgraphs_{mn}(U,V)$ is well-defined, and $\omega$ factors through the quotient to define $\vgraphs_{mn}(U,V) \xrightarrow{\omega} \OmPA^{*}(\VFM_{mn}(U,V))$.
\end{proposition}

\begin{proof}
  This follows from Lemmas~\ref{lem:vanish-loop-double},~\ref{lem:vanish-unival},~\ref{lem:c-vanish-unival}, and~\ref{lem:vanish-bival}.
\end{proof}

\section{Proof of the formality}
\label{sec:proof-formality}

In this section we complete the proof of the formality of the operad $\VFM_{mn}$.
We first show that, up to homotopy, the differential of $\vgraphs_{mn}$ can be simplified: the Maurer--Cartan element $c$ is gauge equivalent to a much simpler one $c_{0}$.
To prove this, we compute a terrestrially-bound version of the twisted graph complex $(\VGC_{mn}^{\vee}, d+[c_{0},-])$ (denoted with a $\downarrow$ superscript).
The difference $c-c_{0}$ belongs to this terrestrially-bound complex.
We prove that $\VGC_{mn}^{\vee,c_{0},\downarrow}$ quasi-isomorphic to the classical hairy graph complex $\HGC_{mn}^{\vee}$, whose cohomology is known; in particular, it vanishes in the right degree.
We can then apply classical deformation-theorerical theorems to prove that $c \simeq c_{0}$.
This enables us to define a simpler version $\vgraphs_{mn}^{0}$ (replacing $c$ with $c_{0}$ in the definition) which is quasi-isomorphic to $\vgraphs_{mn}$.
Then we construct a map from $\vgraphs_{mn}^{0}$ to $H^{*}(\VFM_{mn})$ by explicit formulas.
We prove by combinatorial arguments that this map is a quasi-isomorphism.
Since $\omega$ is clearly surjective in cohomology, the theorem will follow.

\subsection{Change of Maurer--Cartan element and \texorpdfstring{$\vgraphs_{mn}^{0}$}{vgraphs0\_mn}}
\label{sec:change-maurer-cartan}

We would like to define a morphism $\vgraphs_{mn} \to \vscV_{mn}$.
However, $\vgraphs_{mn}$ depends on the Maurer--Cartan element $c \in \VGC_{mn}^{\vee}$ from Equation~\eqref{eq:def-c}, and we do not know the precise form of $c$.
We just know its leading terms:

\begin{proposition}\label{prop:c-leading}
  For $n - 2 \ge m \ge 1$, we have $c = c_{0} + (\dots)$, where $(\dots)$ denotes terms where $\#\{\text{terr.\ vert.}\} + 2 \#\{\text{aer.\ vert.}\} > 3$, and:
  \begin{equation}
    \label{eq:c-zero}
    c_{0} \coloneqq
    \begin{cases}
      \begin{tikzpicture}[baseline = (i.south)]
        \node[intt] (i) {};
        \node[intt, right = 1cm of i] (j) {};
        \draw[dashed] (i) -- (j);
      \end{tikzpicture}
      +
      \;
      \begin{tikzpicture}[baseline=(i.south)]
        \node[intt] (i) {};
        \node[inta, right = .5 of i] (j) {};
        \draw (i) -- (j);
      \end{tikzpicture}
      \in \VGC_{mn}^{\vee},
      & \text{ if } m \ge 2; \\
            \begin{tikzpicture}[baseline = (i.south)]
        \node[intt] (i) {};
        \node[intt, right = 1cm of i] (j) {};
      \end{tikzpicture}
      +
      \;
      \begin{tikzpicture}[baseline=(i.south)]
        \node[intt] (i) {};
        \node[inta, right=.5 of i] (j) {};
        \draw (i) -- (j);
      \end{tikzpicture}
      \in \VGC_{1n}^{\vee},
      & \text{ if } m = 1.
    \end{cases}
  \end{equation}
\end{proposition}
\begin{proof}
  This follows from Remark~\ref{rmk:leading-c}, Remark~\ref{rmk:graph-eta}, and Remark~\ref{rmk:m-egal-un} for $m = 1$.
  Note that $c$ vanishes by definition on the graph with just one internal terrestrial vertex (see after Equation~\eqref{eq:def-c}).
  There are no other (Lie-)connected graphs that satisfy $\#\{\text{terr.\ vert.}\} + 2 \#\{\text{aer.\ vert.}\} \le 3$, with no loops, double edges, full edges between terrestrial vertices, or bivalent terrestrial vertices connected by dashed edges to other terrestrial vertices (see Lemmas~\ref{lem:vanish-loop-double} and~\ref{lem:vanish-bival}).
\end{proof}

If we knew that $c = c_{0}$, then we would be able to build a map $\vgraphs_{mn} \to \vscV_{mn}$ easily (see Section~\ref{sec:conn-graphs-cohom}).
In this section, we show that $c$ is gauge equivalent to $c_{0}$.
For this we show that the cohomology of $\VGC_{mn}^{\vee}$ twisted by $c_{0}$ vanishes in the right degree.
Obstruction theory then shows that $c$ is gauge equivalent to $c_{0}$.
As we will see, the cohomology of the graph complex
\begin{equation}
  \VGC_{mn}^{\vee,c_{0}} \coloneqq (\VGC_{mn}^{\vee}, d + [c_{0},-])
\end{equation}
is related to the cohomology of $\GC_{n}^{\vee}$ and the cohomology of the ``hairy graph complex'' $\HGC_{mn}^{\vee}$ that will be defined below (see e.g.\ \cite[Section~2.2.6]{FresseWillwacher2015} or~\cite{AroneTurchin2015}).

\begin{proposition}[{\cite[Proposition~2.2.3]{FresseWillwacher2015}}]
  \label{prop:vani-gcn}
  The cohomology of $\GC_{n}^{\vee}$ splits as $H^{*}(\GC_{n}^{\vee,\ge 3}) \oplus \bigoplus_{l \equiv 2n+1 \pmod{4}} \K \gamma_{l}$, where $\deg \gamma_{l} = l - n$ and $\GC_{n}^{\vee,\ge 3}$ is the Lie subalgebra of graphs whose vertices are all at least trivalent.
  The class $\gamma_{l}$ is represented by the loop with $l$ vertices.
  Moreover $H^{> -n}(\GC_{n}^{\vee, \ge 3}) = 0$ for $n \ge 3$.
\end{proposition}

\begin{definition}\label{def:hgcmn}
  For $k \ge 0$, let $\Graphs'_{n}(k)$ be the quotient of $\Graphs_{n}(k)$ by the ideal spanned by graphs which are disconnected or whose external vertices are not univalent.
  The hairy graph complex is (with differential from $\Graphs'_{n}$):
  \begin{equation}
    \HGC^{\vee}_{mn} \coloneqq \prod_{k \ge 1} \bigl( \Graphs'_{n}(k)^{\vee} \otimes (\R[m])^{\otimes k} \bigr)^{\Sigma_{k}}[-m].
  \end{equation}
\end{definition}

The full hairy graph complex is the dual of the shifted CDGA $S(\HGC_{mn}[m])[-m]$.
The complex $\HGC_{mn}^{\vee}$ is spanned by (infinite sums of) graphs whose external vertices are exactly univalent and indistinguishable.
The differential is given by vertex splitting.
Each external vertex, together with its only incident edge, can be seen as a ``hair'', which justifies the terminology.
This hairy graph complex is of great topological interest, as it can e.g.\ be used to compute spaces of higher-codimensional long knots~\cite{AroneTurchin2015} or the mapping space $\operatorname{Map}(\DD_{m}, \DD_{n})$~\cite{FresseTurchinWillwacher2017}.

\begin{proposition}[{\cite[Proposition~2.2.7]{FresseWillwacher2015}}]
  \label{prop:vani-hgc}
  When $n - m \ge 2$, the cohomology of the hairy graph complex $\HGC_{mn}^{\vee}$ vanishes in degrees $> -1$.
\end{proposition}
\begin{proof}
  Note that our definition of the hairy graph complex (denoted by $\HGC_{mn}$ without the dual in~\cite{FresseWillwacher2015}) is slightly different, as we allow bivalent and univalent internal vertices.
  However, we can reuse their arguments to show that the inclusion of their complex into ours is a quasi-isomorphism (see also~\cite[Proposition~3.4]{Willwacher2014} for a similar argument).
  Briefly, we can filter both complexes by the number of internal vertices of valence $\ge 3$.
  Both spectral sequences collapse starting on page $E^{2}$, and the inclusion induces an isomorphism on this page.
  We can then use~\cite[Proposition~2.2.7]{FresseWillwacher2015} to show the vanishing of the homology in degrees $> -1$ (note that in the reference, homologically graded complexes are used, so we just use the natural correspondence that reverse degrees).
\end{proof}

For $m \ge 2$, there is a natural preLie product on $\HGC_{mn}^{\vee}$, induced by the operad structure of $\Graphs_{n}^{\vee}$.
Roughly speaking, $\Gamma \circ \Gamma'$ is obtained by inserting $\Gamma'$ in an external vertex of $\Gamma$ and reconnecting the incident edge to a (non-hair) vertex of $\Gamma'$, in all possible ways.
Moreover, there is a natural action of the Lie algebra $\GC_{n}^{\vee}$ (see Definition~\ref{def:fgcn}) on $\HGC_{mn}^{\vee}$.
Given $\Gamma \in \HGC_{mn}^{\vee}$ and $\gamma \in \GC_{n}^{\vee}$, the action $\Gamma \cdot \gamma$ is given by inserting $\gamma$ at a vertex of $\Gamma$ in all possible ways.

When $m = 1$, this simple Lie algebra structure is not right.
There is an $L_{\infty}$-structure on $\HGC_{1n}^{\vee}$, called the Shoikhet structure~\cite{Willwacher2015}.
It is defined by a certain Maurer--Cartan element $m_{\mathrm{trans}}$~\cite{Shoikhet2018} (in an oriented version of the graph complex $\GC_{n}^{\vee}$).
The hairy graph complex with this $L_{\infty}$-structure is denoted by $\HGC_{1n}^{'\vee}$ and encodes the deformation complex of the map $\ee_{1} = \Ass \to \ee_{n}$ rather than $\Pois \to \ee_{n}$ (we refer to~\cite[Theorem~7.16]{FresseTurchinWillwacher2017} where this connection is spelled out in detail).
See~\cite[Section~7]{Willwacher2015} for examples of this $L_{\infty}$ structure.

\begin{definition}\label{def:fvgcterr}
  Let $\fVGC_{mn}^{\vee,c_{0},\downarrow}$ be the submodule of $\fVGC_{mn}^{\vee,c_{0}}$ spanned by graphs whose connected components all have at least one terrestrial vertex and at least one full edge.
  Let $\VGC_{mn}^{\vee,c_{0},\downarrow} \subset \fVGC_{mn}^{\vee,c_{0},\downarrow}$ be the submodule of connected graphs.
  Then clearly:
\end{definition}

\begin{lemma}
  The submodule $\VGC_{mn}^{\vee,c_{0},\downarrow} \subset \VGC_{mn}^{\vee,c_{0}}$ is a dg-Lie subalgebra and a Lie $\GC_{n}^{\vee}$-submodule.
  \qed{}
\end{lemma}

\begin{lemma}
  \label{lem:incl-sgc-hgc}
  There is an inclusion of dg-modules $\fHGC_{mn}^{\vee} \subset \fVGC_{mn}^{\vee,c_{0},\downarrow}$ obtained by considering all external vertices as terrestrial, with no dashed edges.
  On the connected parts, for $m \ge 2$ this inclusion is compatible with the Lie structure and the action of the Lie algebra $\GC_{n}^{\vee}$ on both sides.
  For $m = 1$, it can be extended to an $L_{\infty}$-morphism whose linear part is the inclusion.
\end{lemma}

\begin{proof}
  Simple inspection shows that the inclusion is well-defined, and that it is compatible with the differential.
  The compatibility with the Lie bracket is clear for $m \ge 2$.
  The case $m = 1$ follows by adapting the proof of~\cite[Proposition~5.1]{Willwacher2015}, replacing the Hochschild complex of $\Graphs^{\vee}_{n}$ with $\VGC_{1n}^{\vee,c_{0},\downarrow}$.
\end{proof}

\begin{proposition}
  \label{prop:qiso-incl}
  The inclusion $\fHGC_{mn}^{\vee} \subset \fVGC_{mn}^{\vee,c_{0},\downarrow}$ is a quasi-isomorphism.
\end{proposition}

\begin{proof}
  The proof is similar to~\cite[Lemma~4.4]{Willwacher2014}.
  Indeed, $\fVGC_{mn}^{\vee,c_{0},\downarrow}$ is very close to the deformation complex of the morphism $\Graphs_{m}^{\vee} \to \Graphs_{n}^{\vee}$ (denoted by $\Def(\hoe_{m} \to \Graphs_{n})$ in~\cite{Willwacher2014}).

  We first filter both complexes by the number of full edges, which is the only kind of edges in $\fHGC_{mn}^{\vee}$.
  The differential of $\fHGC_{mn}^{\vee}$ always increases this number strictly by $1$.
  Let us write $c_{0} = c'_{0} + c''_{0}$, where $c'_{0}$ is the part with two terrestrial vertices, and $c''_{0}$ with one vertex of each kind (see Equation~\eqref{eq:c-zero}).
  The differential of $\fVGC_{mn}^{\vee,c_{0},\downarrow}$ increases the filtration number by $1$ (for the action of $\mu + c''_{0}$) or keeps it constant (for the action of $c'_{0}$).
  Hence on the associated spectral sequences, the differential of $E^{0}\fHGC_{mn}^{\vee}$ vanishes, while the differential of $E^{0}\fVGC_{mn}^{\vee,c_{0},\downarrow}$ is just the bracket $[c'_{0}, -]$.

  We now check that the inclusion induces a quasi-isomorphism on these $E^{0}$ pages, from which the proposition follows.
  Let us first assume that $m \ge 2$.
  Given $\Gamma \in \fVGC_{mn}^{\vee,c_{0},\downarrow}$, define its character $[\Gamma] \in \fHGC_{mn}^{\vee}$ as follow: remove all terrestrial vertices and dashed edges, and call the full edges that used to be connected to terrestrial vertices ``dangling'', then make the dangling edges into hairs (see~\cite[Lemma~4.4]{Willwacher2014} for an analogous definition).
  The differential $[c'_{0},-]$ does not change the character of a graph.
  Hence $E^{0}\fVGC_{mn}^{\vee,c_{0},\downarrow}$ splits:
  \begin{equation}
    E^{0} \fVGC_{mn}^{\vee,c_{0},\downarrow} = \prod_{\gamma \in \fHGC_{mn}^{\vee}} \underbrace{\{ \Gamma \in E^{0}\fVGC_{mn}^{\vee,c_{0},\downarrow} \mid [\Gamma] = \gamma\}}_{\eqqcolon C_{\gamma}}.
  \end{equation}
  Let $\gamma \in \fHGC_{mn}^{\vee}$ be a graph with hairs $\{ h_{1}, \dots, h_{k} \}$.
  Let $G$ be the group of permutations of hairs.
  Then $C_{\gamma}$ is isomorphic to $C'_{\gamma} = (\Graphs_{m}^{\vee} \circ \bar{S}^{c}(H_{1}, \dots, H_{k})^{(1,\dots,1)})^{G}$, where $\bar{S}^{c}(H_{1},\dots,H_{k})$ is the (non counital) cofree cocommutative coalgebra on variables $H_{i}$ of degree $-m$, $\Graphs_{m}^{\vee} \circ -$ is the free $\Graphs_{m}^{\vee}$-algebra functor, and $(-)^{(1,\dots,1)}$ is the subcomplex where each $H_{i}$ appears exactly once.
  Indeed, we can view $\Xi \in C'_{\gamma}$ as a linear combination of graphs from $\Graphs_{m}^{\vee}(r)$ with each external vertex decorated by one or more $H_{i}$, with each $H_{i}$ appearing once.
  We can identify $\Xi$ with an element of $C_{\gamma}$ by making its edges dashed, its vertices terrestrial, and we glue $\gamma$ to the graph obtained, connecting the hair $h_{i}$ to the vertex decorated by $H_{i}$.
  The hairs are indistinguishable, but $\Xi$ is invariant under $G$ so this is well-defined.
  This is illustrated by (with $\Xi$ at the bottom):
  \begin{equation}
    \begin{tikzpicture}[baseline=.5cm]
      \node[exta, label={\tiny $H_{2}$}] (a) {1};
      \node[exta, right=.5cm of a, label={\tiny $H_{1},H_{3}$}] (b) {2};
      \node[inta, right=.5cm of b] (i) {};
      \draw (a) -- (b) -- (i);

      \node[minimum width = 1cm, above = 1cm of b, draw, densely dotted] (g) {$\gamma$};
      \draw[->] (g) to +(-1,-.5) node[left] {\tiny $h_{2}$};
      \draw[->] (g) to +(0,-.7) node[left] {\tiny $h_{1}$};
      \draw[->] (g) to +(.4,-.7) node[right] {\tiny $h_{3}$};
    \end{tikzpicture}
    \; \longmapsto \;
    \begin{tikzpicture}[baseline = .5cm]
      \node[intt] (a) {};
      \node[intt, right=1cm of a] (b) {};
      \node[intt, right=1cm of b] (i) {};
      \node[minimum width = 1cm, above = 1cm of b, draw, densely dotted] (g) {$\gamma$};
      \draw[dashed] (a) -- (b) -- (i);
      \draw (g) edge (a) edge (b) edge [bend left] (b);
    \end{tikzpicture}
    \in C_{\gamma}.
  \end{equation}
  For example, full edges between terrestrial vertices may be obtained when $\gamma$ contains a copy of the ``line graph'', i.e.\ the only connected hairy graph with no internal vertices.
  The differential $[c'_{0},-]$ replicates the differential of $\Graphs_{m}^{\vee}(k)$ (i.e.\ vertex splitting), thus this is an isomorphism of dg-modules.

  The homology of $\Graphs_{m}^{\vee}$ is the $m$-Poisson operad (Theorem~\ref{thm:fml-kont}).
  Checking the degrees and the induced differential $[\mu + c''_{0},-]$, we can identify the page $E^{1}\fVGC_{mn}^{\vee,c_{0},\downarrow}$ with (a shift of) the deformation complex $\Def(\hoe_{m} \xrightarrow{*} \Graphs_{n}^{\vee})$ considered in~\cite{Willwacher2014}.
  Note that there, the case $n = m$ is considered and so the map $\hoe_{n} \to \Graphs_{n}^{\vee}$ sends the Lie bracket to a nonzero element; however, in~\cite{Willwacher2014}, the part of the differential induced by this element is discarded, so the complex considered is $\Def(\hoe_{m} \xrightarrow{*} \Graphs_{n}^{\vee})$ up to shifts.
  Compare also with the results of~\cite[Section~5]{AroneTurchin2015}, where the full hairy graph complex is called $HH^{m,n}$.

  The differential of $\fHGC_{mn}^{\vee}$ raises the number of edges by $1$, so the page $E^{1}\fHGC_{mn}^{\vee}$ is just $\fHGC_{mn}^{\vee}$.
  Thanks to~\cite[Lemma~4.4]{Willwacher2014}, the inclusion $\fHGC_{mn}^{\vee} \to \Def(\hoe_{n} \xrightarrow{*} \Graphs_{n}^{\vee})$ is a quasi-isomorphism, thus our morphism induces an isomorphism on the $E^{2}$ page of the spectral sequence and so it is a quasi-isomorphism itself.

  For $m = 1$, the proof is similar, but Lie clusters replace dashed edges.
  We get that $E^{1}\fVGC_{1n}^{\vee,c_{0},\downarrow}$ is the (chains) deformation complex of $\hoe_{1} \to \Graphs^{\vee}_{n}(k)$, whose homology is the full hairy graph complex $\fHGC_{1n}$.
  The induced morphism on the $E^{2}$ pages is the identity, from which the result follows.
\end{proof}

\begin{corollary}
  The inclusion $\HGC_{mn}^{\vee} \subset \VGC_{mn}^{\vee,c_{0},\downarrow}$ is a quasi-isomorphism.
\end{corollary}
\begin{proof}
  Both CDGAs $\fHGC_{mn}^{\vee}$ and $\fVGC_{mn}^{\vee,c_{0},\downarrow}$ are free as CDGAs, so they are in particular cofibrant.
  The functor of indecomposables is a left Quillen adjoint~\cite[Section~12.1.3]{LodayVallette2012}.
  It thus preserves quasi-isomorphisms between cofibrant objects.
  Since the indecomposables of the two CDGAs mentioned above are respectively $\HGC_{mn}^{\vee}$ and $\VGC_{mn}^{\vee,c_{0},\downarrow}$, we conclude by Lemma~\ref{lem:incl-sgc-hgc} and Proposition~\ref{prop:qiso-incl}.
\end{proof}

\begin{corollary}\label{cor:gauge-equivalence}
  The Maurer--Cartan element $c-c_{0} \in \VGC_{mn}^{\vee,c_{0}}$ is gauge equivalent to zero; equivalently, $c$ and $c_{0}$ are gauge equivalent.
\end{corollary}

\begin{proof}
  Let $C \subset \VGC_{mn}^{\vee,c_{0}}$ be the subalgebra spanned by graphs which are not the loops $\gamma_{l}$ from Proposition~\ref{prop:vani-gcn}.
  Similarly let $C' \subset \GC_{n}^{\vee}$ be the subalgebra spanned by graphs with are not the loops.
  We have a short exact sequence $0 \to \VGC_{mn}^{\vee,c_{0},\downarrow} \to C \to C' \to 0$.

  The coefficient $c - c_{0}$ belongs to the subalgebra $C$.
  Indeed, $c$ vanishes on the loops $\gamma_{l}$ by degree reasons (and so does $c_{0}$).
  Moreover, we can compute $c$ on purely terrestrial/dashed graphs and show that it agree with $c$.
  If $m = 1$, then this follows by immediate degree reasons.
  For $m \ge 2$, the restriction of $c$ to purely terrestrial/dashed graphs is equal to Kontsevich's coefficient $\mu \in \GC_{m}^{\vee}$ ($\VFM_{mn}(U,\varnothing) = \FM_{m}(U)$ and the integral is identical).
  So the fact that $c$ and $c_{0}$ agree on such graphs follows from the explicit description of $\mu$ in Equation~\eqref{eq:mu}.

  We can then combine Propositions~\ref{prop:vani-gcn},~\ref{prop:vani-hgc} and~\ref{prop:qiso-incl} to get that the homology of $C$ vanishes in degrees $> -1$.
  We conclude by applying the Goldman--Millson theorem~\cite{GoldmanMillson1988} (see~\cite{Getzler2009,DolgushevRogers2015} for modern accounts that explicitly deal with MC elements) to the inclusion of the truncation $\tau_{<0} \VGC_{mn}^{\vee,c_{0}} \subset \VGC_{mn}^{\vee,c_{0}}$.
\end{proof}

\begin{definition}
  \label{def:sgraphs-zero}
  Let $\VGraphs_{mn}^{0}$ be the variant of $\VGraphs_{mn}$ where we use $c_{0}$ instead of $c$ to twist the Hopf cooperad $\VGra_{mn}$ in the step of Definition~\ref{def:tw-sgra}.
\end{definition}

\begin{corollary}
  \label{cor:sgraphs-zero-qiso}
  The Hopf cooperads $\VGraphs_{mn}$ and $\VGraphs_{mn}^{0}$ are quasi-isomorphic.
\end{corollary}

\begin{proof}
  This follows from the same general arguments of~\cite[Section~5.4]{CamposIdrissiLambrechtsWillwacher2018}.
  Let us briefly describe them.
  Let $S(t,dt)$ be the algebra of polynomial forms on the interval $[0,1]$, with $\deg t = 0$ and $\deg dt = 1$.
  Let $\VGC_{mn}^{\vee,\sim}$ be the Lie algebra with differential $[\mu,-]$, i.e.\ we are only allowed to split aerial vertices.
  Both $c$ and $c_{0}$ are Maurer--Cartan elements, i.e.\ they satisfy $[\mu,c] + \frac{1}{2} [c,c] = [\mu,c_{0}] + \frac{1}{2} [c_{0},c_{0}] = 0$.
  The Lie algebra $\VGC_{mn}^{\vee}$ is the twist of $\VGC_{mn}^{\vee,\sim}$ with respect to $c$.

  The gauge equivalence between $c$ and $c_{0}$ is a Maurer Cartan element $c_{t} \in \VGC_{mn}^{\vee,\sim} \otimes S(t,dt)$ whose restriction at $t = 1$ (resp.\ $t=0$) is $c$ (resp.\ $c_{0}$).
  This element $c_{t}$ produces a differential on $\VGraphs_{mn} \otimes S(t,dt)$ such that restriction at $t = 1$ (resp.\ $t = 0$) gives $\VGraphs_{mn}$ (resp.\ $\VGraphs_{mn}^{0}$).
  We thus have a zigzag:
  \begin{equation}
    \VGraphs_{mn} \xleftarrow{\operatorname{ev}_{t=1}} \VGraphs_{mn} \otimes S(t,dt) \xrightarrow{\operatorname{ev}_{t=0}} \VGraphs_{mn}^{0}.
  \end{equation}
  The evaluation maps $\operatorname{ev}_{t=0}, \operatorname{ev}_{t=1} : S(t,dt) \to \R$ are quasi-isomorphisms of CDGAs.
  This implies that the two maps above are quasi-isomorphisms.
\end{proof}

\begin{definition}\label{def:small-vgraphs0}
  Let $\vgraphs_{mn}^{0}$ be the quotient of $\VGraphs^{0}_{mn}$ defined similarly to how $\vgraphs_{mn}$ is a quotient of $\VGraphs_{mn}$ (see Definition~\ref{def:very-small-graph}).
\end{definition}

\begin{lemma}
  The quotient $\vgraphs_{mn}^{0}$ is a relative Hopf $\graphs_{n}$-cooperad.
\end{lemma}
\begin{proof}
  See Proposition~\ref{prop:omega-reduced}: $c_{0}$ satisfies the same vanishing lemmas as $c$.
\end{proof}

\begin{proposition}\label{prop:quotient-qiso}
  The quotient map $\VGraphs^{0}_{mn} \to \vgraphs^{0}_{mn}$ is a quasi-isomorphism.
\end{proposition}
\begin{proof}
  This follows from the same arguments as in the proof of~\cite[Proposition~3.8]{Willwacher2014}.
  One can set up spectral sequences (counting bivalent vertices of the appropriate type) to see that univalent vertices and bivalent terrestrial vertices with dashed incident edges are killed up to homotopy.
  Similarly another spectral sequence shows that loops (called tadpoles in~\cite{Willwacher2014}) are killed up to homotopy.
\end{proof}

\subsection{Connecting the graphs to the cohomology}
\label{sec:conn-graphs-cohom}

The goal of this section is to describe a quasi-isomorphism of Hopf cooperads $\pi : \vgraphs_{mn}^{0} \to \vscV_{mn}$, where $\vscV_{mn} = H^{*}(\VFM_{mn})$ was obtained in Section~\ref{sec.comp-cohom}.
We will describe this map on generators.
The CDGA $\vgraphs_{mn}^{0}(U,V)$ is free as an algebra for $m \ge 2$.
Its generators are the ``internally connected'' graphs, i.e.\ the graphs which stay connected when all the external vertices are removed.
For example, if $\Gamma$ has no internal vertices, then it is internally connected iff it has exactly one edge (an empty graph is not connected).

\begin{definition}\label{def:pi}
  If $\Gamma$ is an internally connected graph, then $\pi(\Gamma) \in \vscV_{mn}(U,V)$ is given by:
  \begin{itemize}[nosep]
  \item If $\Gamma = e_{vv'}$ has no internal vertices and one full edge between $v \neq v' \in V$, then $\pi(\Gamma) = \omega_{vv'}$.
  \item If $\Gamma = \tilde{e}_{uu'}$ (for $m \ge 2$) has no internal vertices and one dashed edge between $u \neq u' \in U$, then $\pi(\Gamma) = \tilde{\omega}_{uu'}$.
  \item If $\Gamma$ is the graph of~\eqref{eq:graph-eta},
    then $\pi(\Gamma) = \eta_{u}$.
  \item In all other cases, $\pi(\Gamma) = 0$.
  \end{itemize}
  This is extended to the whole algebra.
  For $m = 1$, a graph $\Gamma \in \vgraphs_{mn}^{0}(U,V)$ is additionally decorated (in the dual basis) with an order on $U \sqcup I$ where $I$ is the set of terrestrial internal vertices.
  The element $\pi(\Gamma) \in \vscV_{1n}(U,V)$ is defined as above.
  It is decorated with the order on $U$ given by the restriction of the order on $U \sqcup I$, multiplied by the number of ways the internal vertices can be reordered while still keeping the same graph.
  (This normalization is due to the canonical isomorphism between invariants and coinvariants.)
\end{definition}

\begin{proposition}
  \label{prop:pi}
  The map $\pi$ defined above is a quasi-isomorphism of Hopf cooperads $\vgraphs_{mn}^{0} \to \vscV_{mn}$.
\end{proposition}
The proof of this proposition is split in a series of lemmas, which occupies the rest of this section (until the conclusion, Theorem~\ref{thm:main}).

\begin{lemma}\label{lem:pi-well-def}
  The map $\pi$ is a well-defined algebra map and is equivariant with the symmetric group actions.
\end{lemma}
\begin{proof}
  For $m \ge 2$, we defined $\pi$ on the generators of a free algebra (forgetting about the differential), so it is well-defined.
  It is moreover clearly equivariant.
  For $m = 1$, we need to check the compatibility with the order on the terrestrial vertices.
  One can directly compute that the coefficients match.
  Let us first illustrate with an example:
  \newlength{\myl}
  \setlength{\myl}{.2cm}
  \begin{align*}
    \biggl( 
    \begin{tikzpicture}[scale=.3, baseline=-.25cm]
      \node[exta] (a1) {\tiny 1};
      \node[exta, right = \myl of a1] (a2) {\tiny 2};
      \node[right = \myl of a2] (p) {};
      \node[exta, right = \myl of p] (a3) {\tiny 3};
      \node[exta, right = \myl of a3] (a4) {\tiny 4};
      \node[exta, right = \myl of a4] (a5) {\tiny 5};
      \node[extt, below = \myl of p] {\tiny 1};
      \node[intt, below = \myl of a1] (i1) {}; \draw (a1) -- (i1);
      \node[intt, below = \myl of a3] (i3) {}; \draw (a3) -- (i3);
      \node[intt, below = \myl of a4] (i4) {}; \draw (a4) -- (i4);
    \end{tikzpicture}
    \biggr)
    \cdot
    \biggl(
    \begin{tikzpicture}[scale=.3, baseline=-.25cm]
      \node[exta] (a1) {\tiny 1};
      \node[exta, right = \myl of a1] (a2) {\tiny 2};
      \node[right = \myl of a2] (p) {};
      \node[exta, right = \myl of p] (a3) {\tiny 3};
      \node[exta, right = \myl of a3] (a4) {\tiny 4};
      \node[exta, right = \myl of a4] (a5) {\tiny 5};
      \node[extt, below = \myl of p] {\tiny 1};
      \node[intt, below = \myl of a2] (i2) {}; \draw (a2) -- (i2);
      \node[intt, below = \myl of a5] (i5) {}; \draw (a5) -- (i5);
    \end{tikzpicture}
    \biggr)
    \mapsto
    \bigl( \frac{1}{2} \eta_{1}\eta_{3}\eta_{4} \bigr) \cdot \eta_{2} \eta_{5},
  \end{align*}
  The product on the LHS is given by:
  \newcommand{\hohoho}[5]{%
    \begin{tikzpicture}[scale=.3, baseline=-.25cm]
      \node[exta] (a1) {\tiny 1};
      \node[exta, right = \myl of a1] (a2) {\tiny 2};
      \node[right = \myl of a2] (p) {};
      \node[exta, right = \myl of p] (a3) {\tiny 3};
      \node[exta, right = \myl of a3] (a4) {\tiny 4};
      \node[exta, right = \myl of a4] (a5) {\tiny 5};
      \node[extt, below = \myl of p] {\tiny 1};
      \node[intt, below = \myl of a1] (i1) {}; \draw (a1) -- (i#1);
      \node[intt, below = \myl of a2] (i2) {}; \draw (a2) -- (i#2);
      \node[intt, below = \myl of a3] (i3) {}; \draw (a3) -- (i#3);
      \node[intt, below = \myl of a4] (i4) {}; \draw (a4) -- (i#4);
      \node[intt, below = \myl of a5] (i5) {}; \draw (a5) -- (i#5);
    \end{tikzpicture}
  }
  \begin{multline*}
    \hohoho{1}{2}{3}{4}{5}
    + \hohoho{2}{1}{3}{4}{5}
    + \hohoho{1}{2}{3}{5}{4}
    \\
    + \hohoho{2}{1}{3}{5}{4}
    + \hohoho{1}{2}{4}{5}{3}
    + \hohoho{2}{1}{4}{5}{3}
  \end{multline*}
  The normalization factor in the formula for $\pi$ is $\frac{1}{2! \cdot 3!} = \frac{1}{12}$ which cancels with the $6$ terms to give the $\frac{1}{2}$ in the RHS.
  More generally, to multiply $\Gamma, \Gamma' \in \vgraphs_{mn}^{0}(U,V)$, if the restriction of the orders on $U$ differ then the result is zero.
  Otherwise let $I,I'$ be the respective sets of internal vertices and let $k = \# U$.
  The linear orders on $U \sqcup I$ and $U \sqcup I'$ split $I$ and $I'$ in $(k + 1)$ blocks consecutive vertices, of respective sizes $i_{0}, \dots, i_{k}$, and $i'_{0}, \dots, i'_{k}$.
  The normalization factor in $\pi(\Gamma)$ (resp.\ $\pi(\Gamma')$) is then $\prod_{j} (i_{j})!^{-1}$ (resp.\ $\prod_{j} (i'_{j})!^{-1}$).
  In $\Gamma \cdot \Gamma'$, for each $0 \le j \le k$, the $i_{j}$ vertices of block $j$ in $I$ are shuffled with the $i'_{j}$ vertices of block $j$ in $I'$, yielding in total $(i_{j} + i'_{j})!(i_{j})!^{-1}(i'_{j})^{-1}$ shuffles.
  The normalization factor in $\pi(\Gamma \cdot \Gamma')$ is $\prod_{j}(i_{j}+i'_{j})!^{-1}$ which is equal to the product of the number of shuffles and the normalization factors of $\Gamma$ and $\Gamma'$.
\end{proof}

\begin{lemma}
  The map $\pi$ commutes with the differentials, i.e.\ $\pi d = 0$.
\end{lemma}

\begin{proof}
  Since $\pi$ is an algebra map and the differential is a derivation, it is sufficient to check this on generators.
  Let $\Gamma$ be an internally connected graph.
  If $\Gamma$ has no internal vertices, then $d\Gamma = 0$ thus $\pi d \Gamma = 0$.
  If $\Gamma$ has one internal vertex, then $\pi d \Gamma = 0$ follows from the Arnold relations and the fact that full edges incident to terrestrial vertices are mapped to zero.

  Assume that $\Gamma$ has at least two internal vertices.
  If a summand in $d\Gamma$ is nonzero, then after contracting one edge, all remaining edges are between external vertices or between an external aerial vertex and a univalent terrestrial internal one.
  There can thus be at most one aerial internal vertex.
  Since contracting an edge cannot reduce the valence of the remaining vertices (contracting dead ends is forbidden), there can only be one internal vertex of valence greater than one, necessarily aerial.
  Using the internal connectedness of $\Gamma$, this special vertex must be connected to all the univalent terrestrial vertices by a full edge.
  In other words, the graph $\Gamma$ must be of this type (plus disconnected external vertices):
  
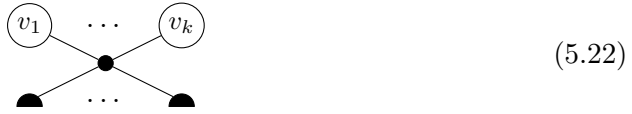
\begin{equation}
    \begin{tikzpicture}[baseline=(i.base), yscale=.3]
      \node[exta] (a1) at (-1,1) {$v_{1}$};
      \node (a2) at (0,1) {\dots};
      \node[exta] (a3) at (1,1) {$v_{k}$};
      \node[inta] (i) at (0,0) {};
      \node[intt] (t1) at (-1,-1) {};
      \node (t2) at (0,-1) {\dots};
      \node[intt] (t3) at (1,-1) {};
      \draw (i) edge (a1) edge (a3) edge (t1) edge (t3);
    \end{tikzpicture}
  \end{equation}
  The Arnold relations in $\eV_{n}$, the symmetry relation $\eta_{v}\omega_{vv'} = \eta_{v'} \omega_{vv'}$, and $\eta_{v}^{2} = 0$ (if there is more than one terrestrial vertex) show that $\pi d \Gamma = 0$.
  The case $m = 1$ is identical except that everything is multiplied by the number of ways of reordering the internal vertices.
\end{proof}

\begin{lemma}
  The map $\pi$ commutes with the cooperad structure maps.
\end{lemma}
\begin{proof}
  It is sufficient to check this on generators, i.e.\ internally connected graphs, which is completely straightforward but tedious.
\end{proof}

\subsection{Proof that \texorpdfstring{$\pi$}{pi} is a quasi-isomorphism}
\label{sec:proof-that-pi}

The last step for Proposition~\ref{prop:pi} is proving that $\pi$ is a quasi-isomorphism.
We split this proof in several sub-lemmas.

\subsubsection{Case $m \ge 2$}

Let us give a rough outline of our strategy.
It is clear that $\pi$ is surjective on cohomology, so we just need to check that $\vgraphs_{mn}^{0}(k,l)$ has the same Betti numbers as $\VFM_{mn}(k,l)$.
We first prove the case $l = 0$, using an inductive argument inspired by~\cite[Theorem~8.1]{LambrechtsVolic2014}.
Then, we reduce to the case of ``split'' graphs, where external aerial vertices and external terrestrial vertices are not in the same connected components.
This mirrors the fact that as a space, $\VFM_{mn}(k,l) \simeq \Conf_{\R^{m}}(k) \times \Conf_{\R^{n} \setminus \R^{m}}(l)$.
Finally, we prove the case $k = 0$, again using an inductive argument.
We conclude using the Künneth formula.

We make an observation that will be useful throughout the proof.
A graph $\Gamma \in \vgraphs_{mn}^{0}(k,l)$ determines a partition of $\{1,\dots,k\} \sqcup \{1,\dots,l\}$, by looking at connected components of $\Gamma$.
We can define the subcomplex of connected graphs:
\begin{equation}
  \vgraphs_{mn}^{0}(k,l)_{\conn} \subset \vgraphs_{mn}^{0}(k,l).
\end{equation}
Then the complex $\vgraphs_{mn}^{0}(k,l)$ splits as a direct sum, over all partitions of $\{1,\dots,k\} \sqcup \{1,\dots,l\}$, of tensor products of complexes $\vgraphs_{mn}^{0}(-,-)_{\conn}$, one for each set in the partition (cf.~\cite[Equation~(8.4)]{LambrechtsVolic2014}).
We define:
\begin{equation}
  \label{eq:betti-conn}
  \beta^{j}(k,l) \coloneqq \dim H^{j}(\vgraphs_{mn}^{0}(k,l)_{\conn}).
\end{equation}
Note that we will focus on the two cases $k = 0$ and $l = 0$, as these will be the relevant ones for the application of the Künneth formula.

\begin{lemma}
  \label{lem:qiso-k-zero}
  The map $\pi : \vgraphs_{mn}^{0}(k,0) \to \vscV_{mn}(k,0) = \eV_{m}(k)$ is a quasi-iso\-mor\-phism for all $k \ge 0$.
\end{lemma}

If $\vgraphs_{mn}^{0}(k,0)$ had no aerial (internal) vertices and no full edges, then it would be equal to $\Graphs_{m}(k)$, which is quasi-isomorphic to $\eV_{m}(k)$ by~\cite[Theorem~8.1]{LambrechtsVolic2014}.
The next proof is the formalization of the intuition that a full edge is killed by:
\begin{equation}\label{eq:exemple-qui-fache}
  \begin{tikzpicture}[baseline=(a.base)]
    \node[extt] (a) {$u$};
    \node[intt, right = 1cm of a] (i) {};
    \node[extt, right = 1cm of i] (b) {$v$};
    \draw[dashed] (a) -- (i);
    \draw (i) -- (b);
  \end{tikzpicture}
  \,\xmapsto{\; d \;}\,
  \begin{tikzpicture}[baseline=(a.base)]
    \node[extt] (a) {$u$};
    \node[extt, right = 1cm of a] (b) {$v$};
    \draw (a) -- (b);
  \end{tikzpicture}
  ,
\end{equation}
and that internal vertices, whether aerial or terrestrial, do not produce any homology  class and are just here to kill the Arnold relations.

\begin{proof}[Proof of Lemma~\ref{lem:qiso-k-zero}]
  Let us work by induction.
  The case $k = 0$ is clear: each component must contain an external vertex, so all graphs are empty and $\pi$ is the identity.
  Assume that $\vgraphs_{mn}^{0}(k,0) \to \eV_{m}(k)$ is a quasi-isomorphism for some $k \ge 0$.
  It is sufficient to prove that $\vgraphs_{mn}^{0}(k+1,0)$ and $\eV_{m}(k+1)$ have the same Betti numbers since $\pi$ is clearly surjective on cohomology.
  Using the splitting in terms of connected components and the Betti numbers of $\Conf_{\R^{m}}$,
  it suffices to prove that $\beta^{i}(k+1,0) = k \cdot \beta^{i-m+1}(k,0)$.
  We split $\vgraphs_{mn}^{0}(k+1,0)_{\conn}$ in three according to the valence of the last external vertex:
  \begin{itemize}[nosep]
  \item $U$: either spanned by the unit graph (if $k = 0$), or the last vertex is univalent and connected to another external vertex by a dashed edge.
  \item $V$: the last external vertex is univalent and connected by a dashed edge to an internal vertex;
  \item $W$: the last external vertex is at least bivalent, or univalent and connected by a full edge to an internal vertex.
  \end{itemize}
  Let $\mathcal{Q} = \vgraphs_{mn}^{0}(k+1,0)_{\conn} / U \cong (V \oplus W, d)$.
  We filter $V$ by the number of edges, and $W$ by the number of edges minus $1$.
  In the $0$th page of the spectral sequence $E^{0}\mathcal{Q}$, the differential maps $V$ isomorphically onto $W$, so $\mathcal{Q}$ is acyclic and $U \simeq \vgraphs_{mn}^{0}(k+1,0)_{\conn}$.
  For $k = 0$, $U = \R = \eV_{m}(1)$ as expected.
  For $k > 0$, $U$ is isomorphic to $\bigoplus_{i=0}^{k} \vgraphs_{mn}^{0}(k,0)_{\conn}[1-m]$ (by removing the vertex $k+1$ and its incident edge).
  Hence $\beta^{j}(k+1,0) = k \cdot \beta^{j-m+1}(k,0)$ as expected too.
\end{proof}

Let us now turn to the second step of the proof.
We prove that we can, in some sense, ``split'' our graph complex in two: external aerial and external terrestrial.

\begin{lemma}
  \label{lem:split}
  Let $k,l \ge 1$ and let $I_{k,l} \subset \vgraphs_{mn}^{0}(k,l)$ be the module spanned by graphs where one of the connected components contains an external aerial vertex and an external terrestrial vertex.
  Then $I$ is an acyclic dg-ideal.
\end{lemma}

\begin{proof}
  The subspace $I_{k,l}$ is a dg-ideal:
  contracting edges does not affect connected components,
  and gluing along external vertices can merge connected components but never split them.
  Let us prove that it is acyclic.
  We only deal with connected graphs (the general case follows by the Künneth formula).
  The proof is similar to Lemma~\ref{lem:qiso-k-zero}: we fix $l \ge 1$ and work by induction on $k \ge 1$.

  For $k = 1$ we check acyclicity directly.
  We split $I_{1,l}$ in two submodules (either the external vertex is univalent with a dashed edge, or not) and we filter like in Lemma~\ref{lem:qiso-k-zero} to get a trivial $E^{1}$ page.
  Let us now assume that the claim is true for a given $k \ge 1$.
  Just like in the proof of Lemma~\ref{lem:qiso-k-zero}, we can split $I_{k+1,l}$ in three summands, depending on whether the last external terrestrial vertex is: univalent, connected by a dashed edge to an external vertex; univalent, connected by a dashed edge to an internal vertex; all other cases.
  The first summand is isomorphic to $\bigoplus_{i=1}^{k} I_{k,l} \simeq 0$.
  The quotient by this summand can be filtered like in Lemma~\ref{lem:qiso-k-zero} and is thus also acyclic.
\end{proof}

\begin{lemma}
  \label{lem:other-lemma}
  The map $\pi : \vgraphs_{mn}^{0}(0,l) \to \vscV_{mn}(0,l)$ is a quasi-isomorphism.
\end{lemma}

\begin{proof}
  This final lemma is also proved by induction.
  Once again $\pi$ is clearly surjective on cohomology, so it suffices to prove that both complexes have the same Betti numbers.
  Using the results of Section~\ref{sec.comp-cohom}, the Poincaré polynomial of $\VFM_{mn}(0,l) \simeq \Conf_{\R^{n} \setminus \R^{m}}(l)$ is $\mathscr{P}(\Conf_{\R^{n} \setminus \R^{m}}(l)) = \prod_{i=0}^{l-1} (1 + t^{n-m-1} + i t^{n-1})$.

  We can again work with the connected part of the graph complex $\vgraphs_{mn}^{0}(0,l)_{\conn}$.
  Note that the case $l = 0$ is covered by Lemma~\ref{lem:qiso-k-zero}.
  The base case that we need to prove  is $\beta^{0}(0,1) = \beta^{n-m-1}(0,1) = 1$, and $\beta^{j}(0,1) = 0$ for other $j$.
  The recurrence relation is $\beta^{j}(0,l+1) = l \cdot \beta^{j-n+1}(0,l)$ for all $j$ and all $l \ge 1$.

  For $l=1$,
  we split $\vgraphs_{mn}^{0}(0,1)_{\conn} = \vgraphs_{mn}^{0}(0,1)$ according to the valence of the only external vertex:
  \begin{itemize}[nosep]
  \item $U$: the external vertex is zero-valent (i.e.\ $\Gamma = 1$) or univalent, connected to a univalent internal terrestrial vertex (i.e.\
    $\Gamma =
    \begin{tikzpicture}[scale=.3, baseline=(v.base)]
      \node[exta] (v) {1};
      \node[intt, right = .5cm of v] (i) {};
      \draw (v) -- (i);
    \end{tikzpicture}
    $).
  \item $V$: the external vertex is at least bivalent.
  \item $V'$: the external vertex is univalent, connected to an aerial internal vertex.
  \item $W$: the external vertex is univalent, connected to a terrestrial internal vertex; this vertex is itself bivalent and its other incident edge is dashed.
  \item $W'$: the external vertex is univalent, connected to a terrestrial internal vertex; this vertex is itself either at least trivalent, or bivalent and both incident edges are full.
  \end{itemize}

  Let $\mathcal{Q} = \vgraphs_{mn}^{0}(0,1) / U$.
  We can set up a spectral sequence just like in Lemma~\ref{lem:qiso-k-zero} so that $E^{0}\mathcal{Q} = \Bigl( V' \xrightarrow[\cong]{d} V \oplus W' \xrightarrow[\cong]{d} W \Bigr)$.
  Thus $\mathcal{Q}$ is acyclic and $U \simeq \vgraphs_{mn}^{0}(0,1)$ is thus a quasi-isomorphism, as we wanted.

  For the induction step, we split $\vgraphs_{mn}^{0}(0,l+1)_{\conn}$ as above.
  We just replace $U$ by $U = \bigoplus_{i=1}^{l} U_{i}$ where $U_{i}$ is spanned by graphs where the external vertex $(l+1)$ is univalent, connected to the external vertex $i$.
  The others are similar but the valence conditions are on the external vertex $(l+1)$ instead.
  The same argument shows that $U \subset \vgraphs_{mn}^{0}(0,l+1)_{\conn}$ is a quasi-isomorphism.
  We have an isomorphism $U_{i} \cong \vgraphs_{mn}^{0}(0,l)[1-m]$ given by removing the last external vertex and its incident edge.
  The Betti numbers thus satisfy the expected recurrence relation: $\beta^{j}(0,l+1) = l \cdot \beta^{j-n+1}(0,l)$.
\end{proof}

\subsubsection{Case $m = 1$}
\label{sec:case-m-=}

We deal separately with the case $m = 1$, because $\ee_{1}$ is the associative operad and not the Poisson operad.
To summarize the differences, recall that: the graphs do not have dashed edges, and the terrestrial vertices are ordered (Definition~\ref{def:vgra-1}); the notion of ``disconnected'' is replaced by ``Lie-disconnected'' (Definition~\ref{def:lie-connected}); the differential $[c_{0},-]$ merges coLie clusters~\eqref{eq:c-zero}.

\begin{proposition}
  \label{prop:case-m-1}
  The map $\pi : \vgraphs_{1n}^{0}(k,l) \to \vscV_{1n}(k,l)$ is a quasi-isomorphism.
\end{proposition}
\begin{proof}
  As in the case $m = 2$, the map $\pi$ is clearly surjective on cohomology, so we just need to check that $\vgraphs_{1n}^{0}(k,l)$ has the correct Betti numbers.
  The proofs of Lemmas~\ref{lem:qiso-k-zero},~\ref{lem:split}, and~\ref{lem:other-lemma} can be adapted in a straightforward manner.
  We can follow the same proofs, replacing $m$ with $1$.
  The crucial difference will be in the splitting of the complex $\vgraphs_{1n}^{0}(k,l)_{\conn}$ or of $I_{k,l}$.
  \begin{itemize}[nosep]
  \item In $\vgraphs_{1n}^{0}(k,0)_{\conn}$ (for Lemma~\ref{lem:qiso-k-zero}), we set $U$ to be the submodule where the $(k+1)$th external vertex is isolated but not adjacent to a terrestrial internal vertex (terrestrial vertices are ordered for $m=1$), $V$ the submodule where the $(k+1)$th external vertex is isolated and adjacent to a terrestrial internal vertex, and $W$ all other kinds of graphs.
  \item In $I_{k,l}$ (for Lemma~\ref{lem:split}), we use the same splitting as for $\vgraphs_{1n}^{0}(k,0)_{\conn}$.
  \item In $\vgraphs_{1n}^{0}(0,l)$ (for Lemma~\ref{lem:other-lemma}), we keep the same $U$, $V$, and $V'$ as in the proof of Lemma~\ref{lem:other-lemma}.
    We change the submodules $W$ and $W'$: in $W$, we require the last external vertex to be connected to a univalent internal terrestrial vertex, while in $W'$ we put all other graphs.
  \end{itemize}
  With this, we obtain the correct recurrence relations on the Betti numbers.
\end{proof}

\subsubsection{Conclusion}

\begin{theorem}
  \label{thm:main}
  The operad $\VFM_{mn}$ is formal over $\R$ for $n - 2 \ge m \ge 1$.
\end{theorem}

\begin{proof}
  We have a zigzag, where $\vscV_{mn}$ is defined in Section~\ref{sec.comp-cohom}, $\VGraphs_{mn}$ in Definition~\ref{def:vgraphs}, $\VGraphs_{mn}^{0}$ in Definition~\ref{def:sgraphs-zero}, $\VGraphs_{mn} \otimes S(t,dt)$ in Corollary~\ref{cor:sgraphs-zero-qiso}, and $\vgraphs_{mn}^{0}$ in Definition~\ref{def:small-vgraphs0}, the map $\pi$ is defined at the beginning of Section~\ref{sec:conn-graphs-cohom}, and the map $\omega$ is defined in Proposition~\ref{prop:omega-reduced}:
  \begin{multline*}
    \vscV_{mn} \xleftarrow{\pi} \vgraphs_{mn}^{0} \gets \VGraphs_{mn}^{0} \gets \\ \gets \VGraphs_{mn} \otimes S(t,dt) \to \VGraphs_{mn} \xrightarrow{\omega} \OmPA^{*}(\VFM_{mn}),
  \end{multline*}

  We proved in Proposition~\ref{prop:coop-G} that $\vscV_{mn} \cong H^{*}(\VFM_{mn})$ as Hopf cooperads.
  We moreover proved in Corollary~\ref{cor:sgraphs-zero-qiso} that the two maps involving the three variants of $\VGraphs_{mn}$ were quasi-isomorphisms of Hopf cooperads.
  We also proved that the quotient $\VGraphs_{mn}^{0} \to \vgraphs_{mn}^{0}$ is a quasi-isomorphism in Proposition~\ref{prop:quotient-qiso}.
  In addition, we proved that $\pi$ was a quasi-isomorphism of Hopf cooperads in Proposition~\ref{prop:pi} (for $m \ge 2$) and Proposition~\ref{prop:case-m-1} (for $m = 1$).
  Therefore it just remains to check that $\omega$ is a quasi-isomorphism of Hopf cooperads to conclude.

  We already know that $\VGraphs_{mn}$ and $\OmPA^{*}(\VFM_{mn})$ have the same cohomology $\vscV_{mn}$.
  Thus we only need $\omega$ to be surjective on cohomology, which is clear ($\eta_{v}$ is obtained by graphs of the type seen in Remark~\ref{rmk:graph-eta}).
\end{proof}

\appendix

\section{Relative cooperadic twisting}
\label{sec:twist-relat-cooper}

Operadic twisting is a tool originally introduced in~\cite[Appendix~I]{Willwacher2014}, studied in further detail in~\cite{DolgushevWillwacher2015}, and generalized to certain types of colored operads in~\cite[Appendix~C]{Willwacher2016}.
In this appendix, we quickly recall operadic twisting for cooperads and right comodules, and we combine both to deal with relative cooperads.

\subsection{Twisting cooperads}
\label{sec:twisting-cooperads}

General references for twisting of plain operads are~\cite[Appendix~I]{Willwacher2014} and~\cite{DolgushevWillwacher2015}.
The dual notion of cooperadic twisting is spelled out in~\cite[Section~1.5]{Idrissi2018b}.
Let $\Lie_{n} = \Lie\{n-1\}$ be the operad governing Lie algebras with a bracket of homological degree $n-1$ (so $\Lie_{n} \subset \ee_{n}$).
Let $\hoLie_{n} = \Omega(\Lie_{n}^{\ashk}) = \Omega(\Com^{\vee}\{n\})$ be its Koszul resolution.
Suppose that $\CC$ is a cooperad with finite-dimensional components (so that $\CC^{\vee}$ is an operad) equipped with a morphism $\mu : \hoLie_{n} \to \CC^{\vee}$.
We consider the following convolution Lie algebra:
\[ \gk_{\CC} \coloneqq \hom_{\Sigma}(\Com^{\vee}\{n\}, \CC^{\vee}) = \biggl( \prod_{i \ge 0} \bigl( \CC^{\vee}(i) \otimes \R[-n]^{\otimes i} \bigr)^{\Sigma_{i}}[n], d, [-,-] \biggr). \]
The differential is induced from $\CC$.
The Lie bracket of $f,g \in \gk_{\CC}$ is $[f,g] = f \star g \mp g \star f$, where $\star$ is the convolution product.
Thanks to~\cite[Theorem~6.5.7]{LodayVallette2012}, the morphism $\mu : \hoLie_{n} \to \CC^{\vee}$ can equivalently be seen as a Maurer--Cartan element $\mu \in \gk_{\CC}$.

The twist of $\CC$ with respect to $\mu$ is the dg-cooperad given in each arity by:
\[ \Tw \CC(U) \coloneqq \bigoplus_{i \ge 0} \bigl( \CC(U \sqcup \{1,\dots,i\}) \otimes \R[n]^{\otimes i} \bigr)_{\Sigma_{i}}. \]
The entries labeled by $U$ are called ``external'', whereas the entries that were labeled by $\{1,\dots,i\}$ before taking coinvariants are called ``internal''.
The cooperadic structure is inherited from $\CC$.
Let $\mu_{1} \in \prod_{j \ge 0} \CC^{\vee}(\{1,\dots,j,*\})^{\Sigma_{j}} = \Tw \CC^{\vee}(*)$ (up to shifts and signs) be the element obtained from $\mu$ by summing over all possible ways of distinguishing one of the inputs.
The differential of $x \in \Tw \CC$ is $dx = d_{\CC}x + x \cdot \mu - x \cdot \mu_{1} - \mu_{1} \cdot x$, i.e.\ the sum of the internal differential of $\CC$ with a threefold action of $\mu$ that we now describe term by term (see~\cite{Willwacher2014} for details):
\begin{enumerate*}[label={(\roman*)}]
\item co-insertion of $\mu$ in an internal input of $x$ in all possible ways;
\item co-insertion of $-\mu_{1}$ in an external input of $x$ in all possible ways;
\item co-insertion of $x$ in the external input of $-\mu_{1}$.
\end{enumerate*}
One checks that $\mu - \mu_{1}$ defines a Maurer--Cartan element in $\gk_{\CC} \ltimes \Tw \CC^{\vee}(*)$, so this differential squares to zero.
The compatibility with the cooperad structure is immediate by coassociativity.

\subsection{Twisting right comodules}
\label{sec:twist-right-comod}

We now recall twisting of right comodules (see~\cite[Appendix~C.1]{Willwacher2016} for the dual case of right modules).
Fix $\mu : \hoLie_{n} \to \CC^{\vee}$ as in Section~\ref{sec:twisting-cooperads}.
Suppose that $\MM$ is a right $\CC$-comodule.
Then as a graded module,
\[ \Tw \MM(U) \coloneqq \bigoplus_{i \ge 0} \bigl( \MM(U \sqcup \{1,\dots,i\}) \otimes \R[n]^{\otimes i} \bigr)_{\Sigma_{i}}. \]
This inherits a right $(\Tw\CC)$-comodule structure from the $\CC$-comodule structure of $\MM$.
The differential of $x \in \Tw \MM(U)$ is given by $dx = d_{\MM}x + x \cdot \mu - x \cdot \mu_{1}$ (where one uses the comodule structure instead of the cooperad structure).
Note that since $\MM$ is only a right module, there can be no term of the type $\mu_{1} \cdot x$.

\subsection{Twisting relative cooperads}
\label{sec:twist-relat-coop}

Let us finally deal with relative cooperads (see Section~\ref{sec:operads} for the definition).
The definition is inspired by the case of ``moperads'' (i.e.\ relative operads which can only admit operations with zero or one terrestrial input)~\cite[Appendix~C.3]{Willwacher2016}.

Let $\Lieto_{mn}$ be the relative $\Lie_{n}$-operad governing triples $(\mathfrak{g}, \mathfrak{h}, f)$ where $\mathfrak{g}$ is a $\Lie_{m}$-algebra, $\mathfrak{h}$ is a $\Lie_{n}$-algebra, and $f : \mathfrak{h}[m-n] \to \mathfrak{g}$ is a morphism of shifted Lie algebras.
We define below an operad $\hoLieto_{mn}$ over $\Lieto_{mn}$.
We will not prove that $\Lieto_{mn}$ is Koszul, although this seems doable using techniques similar to~\cite{HoefelLivernet2013}.

Let $\Comto_{mn}$ be the operad governing triples $(A,B,\alpha)$ where $A$ is a $\Com\{n\}$-algebra, $B$ is a $\Com\{m\}$-algebra, and $\alpha : A \to B[n-m]$ is a morphism of shifted commutative algebras.
In particular, $\Comto_{mn}(U,V) \cong \R[-m]^{\otimes U} \otimes \R[-n]^{\otimes V} \otimes \R[m]$ is one-dimensional for all pairs $(U,V) \neq (\varnothing, \varnothing)$ (and $\Comto_{mn}(\varnothing,\varnothing) = 0$).
By definition, $\Comto_{mn}$ is a relative $\Com\{n\}$-operad.
Thus the cobar construction $\hoLieto_{mn} \coloneqq \Omega(\Comto^{\vee}_{mn})$ is a relative $\hoLie_{n}$-operad (since $\hoLie_{n} = \Omega(\Com_{n}^{\vee})$).

Let $\DD$ be a relative $\CC$-cooperad equipped with a map $(c,\mu) : (\hoLieto_{mn}, \hoLie_{n}) \to (\DD, \CC)$.
This map can equivalently be seen as a Maurer--Cartan element in $\gk_{\CC} \oplus \gk_{\CC,\DD}$ (i.e.\ $d\mu + \frac{1}{2}[\mu,\mu] = dc + [\mu,c] + \frac{1}{2} [c,c] = 0$), where $\gk_{\CC,\DD} = \hom_{\Sigma}(\Comto_{mn}^{\vee}, \DD)$.

Let us define the twisted relative $(\Tw\CC)$-cooperad $\Tw \DD$.
As a graded module,
\[ \Tw \DD(U,V) \coloneqq \bigoplus_{i,j \ge 0} \bigl( \DD(U \sqcup \{1,\dots,i\}, V \sqcup \{1,\dots,j\}) \otimes \R[m]^{\otimes i} \otimes \R[n]^{\otimes j} \bigr)_{\Sigma_{i} \times \Sigma_{j}}. \]
The relative $(\Tw \CC)$-cooperad structure is inherited from $\DD$.
Let $c_{1}$ be the element obtained from $c$ by summing over all possible ways of distinguishing one of the terrestrial inputs of $c$ (similarly to how $\mu_{1}$ is defined from $\mu$).
Then the differential of $x \in \Tw \DD(U,V)$ is given by $d_{\DD}x + x \cdot \mu + x \cdot (c - c_{1}) - c_{1} \cdot x$.

\begin{proposition}
  The collection $\Tw \DD$ defines a relative $(\Tw \CC)$-cooperad.
\end{proposition}
\begin{proof}
  Generalizing the proofs of~\cite[Appendix~I]{Willwacher2014} and~\cite[Appendix~C.3]{Willwacher2016} to this setting is straightforward.
  One checks that $\mu + c + c_{1}$ defines a Maurer--Cartan element in $\gk_{\CC} \rtimes \gk_{\CC,\DD} \rtimes \Tw \DD^{\vee}(*,\varnothing)$ -- which acts by cooperadic coderivations on $\Tw \DD$ -- so the differential above squares to zero.
  Compatibility with the cooperad structure follows from the coassociativity of the cooperad structure.
\end{proof}

\section{Compactifications and projections}
\label{sec:semi-algebr-comp}

In this appendix, we sketch a proof of Proposition~\ref{prop:vfm-sa}: $\VFM_{mn}(U,V)$ is a compact SA manifold and a smooth manifold with corners, and the canonical projection maps are SA bundles.
Our proofs are heavily inspired by~\cite[Section~5.9]{LambrechtsVolic2014}.

\begin{wrapfigure}[4]{r}{2.5cm}
  \vspace{-1.5em}
  \begin{tikzpicture}[grow' = up, sibling distance = .5cm, level distance = .5cm]
    \coordinate
    child{
      child{ node {$v_{1}$} }
      child{ node {$v_{2}$}}
    }
    child{ node {$v_{3}$} }
    child[dashed]{
      child[solid] { node {$v_{4}$}}
      child[dashed]{
        child[dashed] { node {$u_{1}$}}
        child[dashed] { node {$u_{2}$}}
      }
    };
  \end{tikzpicture}
\end{wrapfigure}

Let $(U,V)$ be a pair of finite sets.
A relative (rooted) tree $\calT$ with leaves $(U,V)$ is a rooted tree with dashed and full edges.
We require that the leaves with incident full (resp.\ dashed) edge are in bijection with $U$ (resp.\ $V$), that if a vertex has only one incoming edge then this edge is full, and that if an edge is full then all the edges above it are full.
An example is on the side.

For a relative tree $\calT$, we let $V_{\calT}$ be the set of all its vertices, $V^{0}_{\calT} = V_{\calT} \setminus \rooot$, and $V^{*}_{\calT} = V_{\calT} \setminus (U \cup V)$.
The set $V_{\calT}$ is partially ordered by considering that a vertex is smaller than any vertex above it.
For $i \in V_{\calT}$, we let $\inc(i) = \inc_{t}(i) \cup \inc_{a}(i) = \{ \text{incoming dashed edges} \} \cup \{ \text{incoming full edges} \}$ and $\parent(i) \in V_{\calT}$ be the immediate predecessor of $i$.
Finally, we let:
\[ \Conft_{mn}^{\calT} \coloneqq \prod\nolimits_{i \in V^{*}_{\calT}} \Conft_{mn}(\inc_{t}(i), \inc_{a}(i)). \]
The spaces $\Conft_{mn}^{\calT}$ will be used to give a decomposition of $\VFM_{mn}(U,V)$ as in~\cite[Section~5.9.1]{LambrechtsVolic2014}.
Let $\xi = (\xi_{i})_{i \in V^{*}_{\calT}} \in \Conft_{mn}^{\calT}$.
We can represent $\xi_{i}$ by a configuration $\bar\xi_{i} \in \Conf_{mn}(\inc_{t}(v), \inc_{a}(i))$ of radius $1$ and whose barycenter is in $\{0\}^{m} \times \R^{n-m}$.
For $i \in V^{0}_{\calT}$, we let $\xi(i) \coloneqq \bar\xi_{\parent(i)}(i)$.
We then define, for $r > 0$ and $i \in V_{\calT}$:
\[ x(\xi,r,i) \coloneqq \sum\nolimits_{j \in V^{0}_{\calT}, j \le i} \xi(j) \cdot r^{\operatorname{height}(j)} \in \R^{n}. \]
Then $(x(\xi,r,i))_{i \in U \sqcup V}$ is a configuration for $r$ small enough.
Let use define $h_{\calT} : \Conft_{mn}^{\calT} \hookrightarrow \VFM_{mn}(U,V)$ by $h_{\calT}(\xi) = \lim_{r \to 0} (x(\xi,r,i))_{i \in U \sqcup V}$.
The map $h_{\calT}$ is a homeomorphism onto its image, $\{ \operatorname{im}(h_{\calT})\}_{\calT}$ covers $\VFM_{mn}(U,V)$, and the interior of $\VFM_{mn}(U,V)$ is the stratum corresponding to a corolla.

Now let $x = h_{\calT}(\xi) \in \VFM_{mn}(U,V)$.
We want to build an SA chart around $x$.
Let
\[ r_{1} \coloneqq \frac{1}{4} \min \{ \| \bar{\xi}_{i}(a) - \bar{\xi}_{i}(b) \| \mid i \in V^{0}_{\calT}, a \neq b \in \inc(i) \} \cup \{ d(\bar\xi_{i}(v), \R^{m}) \mid i \in V^{0}_{\calT}, v \in \inc_{a}(i) \}. \]
Note that $r_{1} \le \frac{1}{2}$, because $\bar\xi_{i}$ has radius $1$.
Define a neighborhood of $\xi$ by $W = \{ \zeta \in \Conft_{mn}^{\calT} \mid \forall i, \| \xi(i) - \zeta(i) \| \le r_{1}^{\#U+\#V+1}\}$ (thanks to the distance condition, distinct points stay distinct and aerial points stay aerial).
For $\tau \in [0,r_{1}]^{V^{*}_{\calT}}$ with $\tau_{\rooot} = 0$ and $0 \le r \le r_{1}$, we let $y(\zeta,\tau,r,\rooot) \coloneqq 0$ and
\[ y(\zeta,\tau,r,i) \coloneqq y(\zeta,\tau,r,\parent(i)) + \xi(i) \cdot \prod_{j < i} \max(r, \tau_{j}). \]
We can then define $\Phi : W \times [0,r_{1}]^{V^{*}_{\calT} \setminus \{\rooot\}} \to \VFM_{mn}(U,V)$ by $\Phi(\zeta,\tau) \coloneqq \lim_{r \to 0} (y(\xi,\tau,r,i))_{i \in U \sqcup V}$.
We also let $V$ be the image of $\Phi$.
The proof that $\Phi$ is an SA chart onto a compact neighborhood of $x$ is identical to~\cite[Lemma~5.9.3]{LambrechtsVolic2014}.
This proves the first part of Proposition~\ref{prop:vfm-sa}.

We would now like to prove that $p_{U,V} : \VFM_{mn}(U \sqcup I, V \sqcup J) \to \VFM_{mn}(U,V)$ is an SA bundle.
Since the composite of two SA bundles is an SA bundle~\cite[Proposition~8.5]{HardtLambrechtsTurchinVolic2011}, it is sufficient to check that the following are SA bundles:
\begin{equation*}
  p : \VFM_{mn}(U \sqcup *, V) \to \VFM_{mn}(U, V), \quad q : \VFM_{mn}(U, V \sqcup *) \to \VFM_{mn}(U,V).
\end{equation*}
We will describe the fibers explicitly as complements of open balls.
The fiber of $p$ will be almost identical to the one in~\cite[Section~5.9.4]{LambrechtsVolic2014}.
However the fiber of $q$ is slightly different, because the new aerial point cannot touch the ground.

Let $x = h_{\calT}(\xi)$, $r_{1} > 0$, and $W$ as before.
For $\zeta \in W$ and $i \in V_{\calT}$, define $x_{1}(\zeta,i) \coloneqq x(\zeta,r_{1},i)$ and $\varepsilon(i) \coloneqq 4 r_{1}^{\operatorname{height}(i)+1}$.
Let $B_{i}(\zeta) \coloneqq B(x_{1}(\zeta,i), \varepsilon(i))$ be the closed ball.
Then $B_{i}(\zeta) \subset 1/3 B_{j}(\zeta)$ if $i < j$ and $B_{i}(\zeta) \cap B_{j}(\zeta) = \varnothing$ otherwise.

Recall $\phi : \R^{n} \times [0,1] \times [0,2] \times \R^{n} \to \R^{n}$, $(c,r,\varepsilon,x) \mapsto \phi_{r}^{c,\varepsilon}(x)$ from~\cite[Lemma~5.9.5]{LambrechtsVolic2014}.
It is such that $\phi_{r}^{c,\varepsilon}$ is radial, the identity outside $B(c,\varepsilon)$, and shrinks $B(c,\varepsilon/3)$ by a factor $r$.
Moreover for a configuration $x \in \Conf_{B(c,\varepsilon/3)}(k)$, $\phi_{r}^{c,\varepsilon}(x)$ is a configuration in $\Conft_{n}(k)$ that does not depend on $r$, and $\phi$ behaves well with respect to other points $z \in B(c,\varepsilon)$ (see the reference for details).
We note in addition that, thanks to the properties of $\phi$, if $c \in \R^{m}$ then $\phi_{r}^{c,\varepsilon}(\R^{m}) \subset \R^{m}$, and if $c \not\in R^{m}$, then $\phi_{r}^{c,\varepsilon}(\R^{n} \setminus B(\R^{m}, \varepsilon)) \subset \R^{n} \setminus B(\R^{m}, \varepsilon)$ (where $B(\R^{m}, \varepsilon) = \bigcup_{z \in \R^{m}} B(z,\varepsilon)$).

Now, fix $\zeta \in W$ and $\tau \in [0,r_{1}]^{V_{\calT}}$ s.t.\ $\tau_{\rooot} = 0$ and $\tau_{i} = 0$ for $i \in U \sqcup V$ a leaf.
Then for $i \in V_{\calT}$ and $0 < r \le r_{1}$, let $\phi_{r}^{i} = \phi_{\max(r, \tau(i))/r_{1}}^{x_{1}(\zeta,i), \varepsilon(i)}$.
Moreover, let $\phi_{r}$ be the composition (in any order thanks to the disjointness of the balls) of the $\phi_{r}^{i}$ for $i \in U \sqcup V$.
(Despite the notation, $\phi_{r}$ depends on $x$, $\zeta$ and $\tau$.)
We then check, like in~\cite[Lemma~5.9.6]{LambrechtsVolic2014}, that $\phi_{r}(x_{1}(\zeta,i)) = x(\zeta,\tau,r,i)$ for $r > 0$ and $i \in U \sqcup V$.

We can now check the local trivialities of $p$ and $q$.
Let us first deal with $p : \VFM_{mn}(U \sqcup *, V) \to \VFM_{mn}(U,V)$.
We define $F_{\zeta} \coloneqq B(x_{1}(\zeta, \rooot), \varepsilon(\rooot)/2) \setminus \bigcup_{u \in U} B(x_{1}(\zeta,u), \varepsilon(u)/2) \cap \R^{m}$ for $\zeta \in W$, which will be the fiber of $p$ over $h_{\calT}(\zeta)$.
We also set $F \coloneqq B(0,\#U+1) \setminus \bigcup_{i = 1}^{\#U} \mathring{B}((i, 0, \dots, 0), 1/4) \cap \R^{m}$.
Note that $F$ is a compact SA manifold (a closed $m$-ball with $U$ open balls removed).
There is an SA-homeomorphism $\Theta_{\zeta}: F \cong F_{\zeta}$ since $W$ is small enough.
Let
\begin{align*}
  \widehat{\Phi} : W \times [0,1]^{V_{\calT}^{*} \setminus \rooot} \times F & \to \VFM_{mn}(U \sqcup *, V), \\
  (\zeta, \tau, z) & \mapsto \lim_{r \to 0} \bigl( (\phi_{r}(x_{1}(\zeta,i)))_{i \in U \sqcup V}, \phi_{r}(\Theta_{\zeta}(z)) \bigr).
\end{align*}
Then $\widehat{\Phi}$ covers $\Phi$.
The proof that $\widehat{\Phi}(\zeta,\tau,-)$ maps $F_{\zeta}$ homeomorphically onto $p^{-1}(\Phi(\zeta,\tau))$ is identical to the proof in~\cite[Section~5.9.4]{LambrechtsVolic2014}.

Now let us deal with $q : \VFM_{mn}(U, V \sqcup *) \to \VFM_{mn}(U,V)$.
Its fiber over $h_{\calT}(\zeta)$ will be $G_{\zeta} \coloneqq B(x_{1}(\zeta, \rooot), \varepsilon(\rooot)/2) \setminus \bigcup_{v \in V} \mathring{B}(x_{1}(\zeta,v), \varepsilon(v)/2) \setminus \mathring{B}(\R^{m}, r_{1})$.
We also have $G \coloneqq B(0, \#V+1) \setminus \mathring{B}(\R^{m}, 1/4) \setminus \bigcup_{i=1}^{\#V} \mathring{B}((i,0,\dots,0,1,0,\dots,0), 1/4)$ (where the $1$ in the open ball is in position $m+1$).
This is a compact SA manifold (a closed $n$-ball with an open tubular neighborhood of $\R^{m}$ and $V$ open $n$-balls removed) and we have an SA-homeomorphism $\Theta_{\zeta} : G \cong G_{\zeta}$ since, again, $W$ is small enough.
We can then define a chart $\widehat{\Phi} : W \times [0,1]^{V_{\calT}^{*} \setminus \rooot} \times G \to \VFM_{mn}(U \sqcup *, V)$ with a formula similar to the one above.
This map covers $\Phi$, and showing that $\widehat{\Phi}(\zeta,\tau,-)$ maps $G_{\zeta}$ SA-homeomorphically to $q^{-1}(\Phi(\zeta,\tau))$ is a straightforward adaption of the arguments in~\cite[Section~5.9.4]{LambrechtsVolic2014}.
This completes the proof of the second part of Proposition~\ref{prop:vfm-sa}.

\printbibliography
\end{document}

%% file: math.tex
\newcommand{\ashk}{\textnormal{¡}}
\newcommand{\gk}{\mathfrak{g}}
\newcommand{\K}{\Bbbk}
\newcommand{\R}{\mathbb{R}}
\newcommand{\Q}{\mathbb{Q}}
\newcommand{\Z}{\mathbb{Z}}
\newcommand{\PP}{\mathsf{P}}
\newcommand{\QQ}{\mathsf{Q}}
\newcommand{\MM}{\mathsf{M}}
\newcommand{\CC}{\mathsf{C}}
\newcommand{\DD}{\mathsf{D}}

\newcommand{\vol}{\mathrm{vol}}
\DeclareMathOperator{\id}{id}
\DeclareMathOperator{\im}{im}
\DeclareMathOperator{\Tw}{Tw}
\DeclareMathOperator{\Def}{Def}

\newcommand{\BF}{\mathcal{BF}}  
\newcommand{\Conf}{\mathrm{Conf}}
\newcommand{\Conft}{\underline{\Conf}}
\newcommand{\OmPA}{\Omega_{\mathrm{PA}}}

\newcommand{\ee}{\mathsf{e}}
\newcommand{\eV}{\ee^{\vee}}
\newcommand{\ssc}{\mathsf{sc}}
\newcommand{\vsc}{\mathsf{cd}}
\newcommand{\scV}{\ssc^{\vee}}
\newcommand{\vscV}{\vsc^{\vee}}
\newcommand{\Ass}{\mathsf{Ass}}
\newcommand{\Pois}{\mathsf{Pois}}
\newcommand{\Lie}{\mathsf{Lie}}
\newcommand{\Lieto}{\overrightarrow{\Lie}}
\newcommand{\hoLie}{\mathsf{hoLie}}
\newcommand{\hoLieto}{\mathsf{ho}\overrightarrow{\mathsf{Lie}}}
\newcommand{\hoe}{\mathsf{hoe}}

\newcommand{\Com}{\mathsf{Com}}
\newcommand{\Comto}{\overrightarrow{\Com}}
\newcommand{\FM}{\mathsf{FM}}
\newcommand{\VFM}{\mathsf{CFM}}
\newcommand{\SC}{\mathsf{SC}}
\newcommand{\VSC}{\mathsf{CD}}
\newcommand{\ESC}{\mathsf{ESC}}
\newcommand{\Disk}{\mathsf{Disk}}
\newcommand{\fr}{\mathrm{fr}}

\newcommand{\Bij}{\mathsf{Bij}}
\newcommand{\Gra}{\mathsf{Gra}}
\newcommand{\Graphs}{\mathsf{Graphs}}
\newcommand{\graphs}{\mathsf{graphs}}

\newcommand{\SGraphs}{\mathsf{SGraphs}}

\newcommand{\VGra}{\mathsf{CGra}}
\newcommand{\VGraphs}{\mathsf{CGraphs}}
\newcommand{\vgraphs}{\mathsf{cgraphs}}
\newcommand{\fGC}{\mathrm{fGC}}
\newcommand{\GC}{\mathrm{GC}}
\newcommand{\fVGC}{\mathrm{fCGC}}
\newcommand{\VGC}{\mathrm{CGC}}
\newcommand{\HGC}{\mathrm{HGC}}
\newcommand{\fHGC}{\mathrm{fHGC}}
\newcommand{\conn}{\mathrm{cn}}
\newcommand{\calT}{\mathcal{T}}
\DeclareMathOperator{\inc}{in}
\DeclareMathOperator{\parent}{par}
\newcommand{\rooot}{\mathrm{root}}

\tikzset{
  vtx/.style={draw, minimum size = 1ex, inner sep = 2pt}, 
  ext/.style={vtx, fill=white, font=\small}, 
  int/.style={vtx, fill=black}, 
  unk/.style={vtx, fill=gray},  
  exta/.style={circle, ext}, inta/.style={circle, int}, unka/.style={circle, unk},
  extt/.style={semicircle,ext}, intt/.style={semicircle, int}, unkt/.style={semicircle, unk}
}
